\documentclass[reqno,11pt]{amsart}
\usepackage{amssymb, amsmath,latexsym,amsfonts,amsbsy, amsthm}
\usepackage{color}

\allowdisplaybreaks[4]

\let\e=\varepsilon

\let\p=\psi

\def\p{\partial}

\newcommand{\beq}{\begin{equation}}
\newcommand{\eeq}{\end{equation}}
\newcommand{\ben}{\begin{eqnarray}}
\newcommand{\een}{\end{eqnarray}}
\newcommand{\beno}{\begin{eqnarray*}}
\newcommand{\eeno}{\end{eqnarray*}}

\renewcommand{\theequation}{\thesection.\arabic{equation}}

\newtheorem{theorem}{Theorem}[section]

\newtheorem{lemma}[theorem]{Lemma}
\newtheorem{proposition}[theorem]{Proposition}

\newtheorem{Theorem}{Theorem}[section]

\newtheorem{Lemma}[Theorem]{Lemma}

\newtheorem{Remark}[Theorem]{Remark}

\begin{document}
\title[Prandtl-Batchelor in annulus]
{Prandtl-Batchelor flows on an annulus}

%
\author{Mingwen Fei}
\address{School of  Mathematics and Statistics, Anhui Normal University, Wuhu 241002, China}
\email{mwfei@ahnu.edu.cn}

\author{Chen Gao}
\address{The Institute of Mathematical Sciences, The Chinese University of Hong Kong, Shatin, N.T., Hong Kong}
\email{gaochen@amss.ac.cn}

\author{Zhiwu Lin}
\address{School of Mathematics, Georgia Institute of Technology, 30332, Atlanta, GA, USA}
\email{zlin@math.gatech.edu}

\author{Tao Tao}
\address{School of  Mathematics, Shandong University, jinan, Shandong, 250100, China}
\email{taotao@sdu.edu.cn}

\date{\today}
\maketitle

\renewcommand{\theequation}{\thesection.\arabic{equation}}
\setcounter{equation}{0}
\begin{abstract}
For steady two-dimensional Navier-Stokes flows with a single eddy (i.e. nested closed streamlines) in a simply connected domain, Prandtl (1905) and Batchelor (1956) found that in the inviscid limit, the vorticity is constant inside the eddy. In this paper, we consider the generalized Prandtl-Batchelor theory for the forced steady Navier-Stokes equations on an annulus. First, we observe that in the limit of infinite Reynolds number, if forced steady Navier-Stokes solutions has nested closed streamlines on an annulus, then the inviscid limit is a rotating shear flow uniquely determined by the external force and boundary conditions. We call solutions of steady Navier-Stokes equations with the above property Prandtl-Batchelor flows.  Then, by constructing higher order approximate solutions of the forced steady Navier-Stokes equations and establishing the validity of Prandtl boundary layer expansion, we give a rigorous proof of the existence of Prandtl-Batchelor flows on an annulus with the wall velocities slightly different from the rigid-rotations along the same direction.
\end{abstract}

\numberwithin{equation}{section}

\indent

\section{Introduction}

\indent

The study of steady Navier-Stokes equations has a long history, starting with the classical paper of J.Leary \cite{Leray}, in which Leray proved that the nonhomogeneous boundary problem of steady Navier-Stokes equations in a two-dimensional bounded domain $\Omega$ with $C^2$ smooth boundary $\partial \Omega =\bigcup_{j=1}^N\Gamma_j$ has a solution under the assumption $\int_{\Gamma_j}\mathbf{u}\cdot \mathbf{n}dl=0, j=1,\cdot\cdot\cdot, N$. V.Korobkov, K.Pileckas and R.Russo in \cite{KKR} obtained the existence theorem under the assumption $\sum_{j=1}^N\int_{\Gamma_j}\mathbf{u}\cdot \mathbf{n}dl=0$. However, the problem of inviscid limit of steady Navier-Stokes equations is still largely open. One of the main difficulties is that there are infinitely many solutions for the steady Euler equations in a closed region, which is very different from the time-dependent inviscid limit where the leading Euler flow is uniquely determined by the initial and boundary conditions. Thus, a selection principle must be found to obtain the physical solution. It is unknown how to choose the Euler flow except for some special case, for example,
Prandtl in \cite{prandtl} considered the steady motion of slightly viscous incompressible fluid in a simply-connected region. By integrating the steady Navier-Stokes equations along a closed streamline and letting the Reynolds number converge to infinity, Prandtl formally found that the vorticity of steady Euler flow must be constant in a region of nested closed streamlines (i.e. a single eddy).  However, Prandtl didn't give an approach to determine this constant. This important property
was rediscovered later by Batchelor (1956) in \cite{B}. This class of result is now usually referred to as Prandtl-Batchelor theory and such laminar flows are called Prandtl-Batchelor flows in the
literature. Furthermore, for Prandtl-Batchelor flows on a circular disk, the Euler flow must be Couette flows $(ar,0)$ and the
 constant $a$ was given in \cite{B,W} by exactly solving steady the nonlinear Prandtl equations in the Von Mises transformation.
The formula to determine $a$ is usually referred to as the Batchelor-Wood formula. Similar results were given by Feymann and Lagerstrom in \cite{feymann-lagerstrom}. However, for Prandtl-Batchelor flows on a more
general domain, it's difficult to determine the limiting vorticity
constant and only approximate results were obtained in \cite{edwards,feymann-lagerstrom, kim-formula,riley81,van wijngarden,W}. In addition, the Prandtl-Batchelor theory for forced steady Navier-Stokes equations has been considered by H.Okamoto in \cite{OK} on torus when the external force is proportional to viscous force, and derived a necessary condition for the limiting Euler flow. H.Okamoto in \cite{OK1} studied the stationary Kolmogolov flows on two-dimensional torus by numerical experiments and found that under a certain condition, the Navier-Stokes flows are ``nearly singular"(nonsmooth) for large Reynolds numbers.

 The Prandtl-Batchelor theory has important applications in many studies involving
laminar flows with small viscosity, for example, the nonlinear critical
layer theory near shear flows in \cite{maslowe}. The Prandtl-Batchelor theory can also be applied to general two-dimensional advection
diffusion equation of a passive scalar field $\theta(x,y)$
\begin{align*}
u\cdot\nabla\theta-R^{-1}\Delta\theta=0,
\end{align*}
where $u(x,y)$ is a given steady incompressible flow and
$R^{-1}$ is the diffusion coefficient. For example,
homogenization of potential vorticity in the ocean circulation theory
\cite{ocean-pedlosky,rhines-young83,rhines-young82} and flux expulsion in the
self-excited dynamo theory \cite{dormy-moffatt-19,moffatt78,weiss66}.

Despite the importance of Prandtl-Batchelor theory and its wide applications, there
is relatively few mathematical works on this problem. Kim initiated a mathematical study of Prandtl-Batchelor flows on a
disk in a series of works
\cite{K1998,K2000,K2003,kim-childress,kim-free boundary,kim-thesis,kim-formula}. In particular, when the boundary velocity is slightly different from a
constant, the well-posedness of the Prandtl equations under the Batchelor-Wood
condition was shown in \cite{K2000} and some asymptotic study of the boundary
layer expansion was given in \cite{K1998}. However, he didn't prove
 the convergence of the boundary layer expansion in the limit of infinite Reynolds number. Recently, by constructing higher order approximate solutions of the
Navier-Stokes equations and establishing the validity of Prandtl boundary layer expansion, we give a rigorous proof of the
existence of Prandtl-Batchelor flows on a disk with the wall velocity slightly
different from the rigid-rotation in \cite{FGLT}.

In this paper, we extend the Prandtl-Batchelor theory to the forced steady Navier-Stokes equations on an annulus. The limiting Euler flow in this case is selected in the following way. Under some suitable assumptions, if the solution of steady Navier-Stokes equations with nested closed streamlines converges to steady Euler flows, then the limiting Euler flow must be a rotating shear flow $(u_e(r),0)$ satisfying (\ref{second ODE for leading Euler flow}). In fact, F.Hamel and N.Nadirashvili in \cite{HN2} proved that if steady Euler flow with rigid wall boundary conditions in two-dimensional annulus domain has no stagnation point, then the flow is circular in the sense that the streamlines are concentric. By integrating the steady Navier-Stokes equations along a circle and letting the Reynolds number converge to infinity, we can show that the velocity $u_e(r)$ and the external force $F_u$ must satisfy the ODE in (\ref{second ODE for leading Euler flow}). And the Batchelor-Wood formula determines the boundary conditions for the velocity $u_e(r)$ in (\ref{second ODE for leading Euler flow}). Combining the ODE and the boundary conditions, we obtain the system (\ref{second ODE for leading Euler flow}) for the Euler flow. We call the above selection the Prandtl-Batchelor theory since it describes which Euler flow is selected in the limit $\varepsilon\rightarrow 0$. In the Appendix, we will give a proof of the above property under the following assumptions: for
sufficiently small viscosity i) the steady flows of Navier-Stokes equations have nested closed streamlines; ii) any
interior domain of the annulus is separated from the boundary layer uniformly for vanishing viscosity; iii) inside
the interior domain, steady Navier-Stokes solutions converge to a steady Euler solution in
$C^{2}$. However, given a domain, boundary data and external force, it is hard
to understand the streamline structures of the steady
solutions of Navier-Stokes equations and control the boundary layer in a domain, hence it seems difficult to verify the above assumptions
 rigorously. Following the argument of \cite{FGLT}, by constructing higher order approximate solutions and establishing the convergence of Prandtl boundary layer expansion, we construct a class of Prandtl-Batchelor flows to the steady Navier-Stokes equations on an annulus with both the wall velocities slightly different from the rigid-rotations along the same direction.

More precisely, we consider the forced steady Navier-Stokes equations in a two-dimensional annulus centered on zero $A:=B_1(0)\backslash B_{r_0}(0)$
\begin{equation}
\left \{
\begin {array}{ll}
\mathbf{u}^\varepsilon\cdot\nabla\mathbf{u}^\varepsilon+\nabla p^\varepsilon-\varepsilon^2\Delta\mathbf{u}^\varepsilon=\varepsilon^2\mathbf{F},\\[5pt]
\nabla\cdot \mathbf{u}^\varepsilon=0,
\end{array}
\right.\label{ns}
\end{equation}
with rotating boundary conditions
\begin{align}\label{rotating boundary condition}
\mathbf{u}^\varepsilon\big|_{\partial B_1}=(\alpha+\eta f(\theta))\mathbf{t},\
\mathbf{u}^\varepsilon\big|_{\partial B_{r_0}}= (\beta+\eta g(\theta))\mathbf{t},
\end{align}
%
where  $\varepsilon^2>0$ is reciprocal to Reynolds number,  $0<r_0<1$, $\alpha>0, \beta>0$, $\mathbf{u}^\varepsilon$ is the velocity, $p^\varepsilon$ is the pressure, $\mathbf{F}$ is the external force, $\eta$ is a  small number, $\mathbf{t}$ is the unit tangential vector to $\partial B_1$ or $\partial B_{r_0}$, $f(\theta),g(\theta)$ are $2\pi$-periodic smooth functions.

 Let $\mathbf{n}$ be the unit normal to $\partial B_1$ and $$\mathbf{u}^\varepsilon=u^\varepsilon(\theta,r)\mathbf{t}+v^\varepsilon(\theta,r)\mathbf{n}, \ \mathbf{F}=F_u(\theta,r)\mathbf{t}+F_v(\theta,r)\mathbf{n},$$
then  (\ref{ns}) reads
\begin{eqnarray}
\left \{
\begin {array}{lll}
u^\varepsilon u^\varepsilon_\theta+rv^\varepsilon u^\varepsilon_r+u^\varepsilon v^\varepsilon+p^\varepsilon_\theta-\varepsilon^2\big(\frac{u^\varepsilon_{\theta\theta}}{r} + ru^\varepsilon_{rr}
 + u^\varepsilon_{r}+\frac{2}{r} v^\varepsilon_\theta-\frac{ u^\varepsilon}{r}\big)=\varepsilon^2rF_u,\\[5pt]
u^\varepsilon v^\varepsilon_\theta+rv^\varepsilon v^\varepsilon_r-(u^\varepsilon)^2+rp^\varepsilon_r-\varepsilon^2\big(\frac{ v^\varepsilon_{\theta\theta}}{r}+ rv^\varepsilon_{rr}
 +v^\varepsilon_{r}-\frac{2}{r} u^\varepsilon_\theta-\frac{ v^\varepsilon}{r}\big)=\varepsilon^2rF_v,\\[5pt]
 u^\varepsilon_\theta+(rv^\varepsilon)_r=0,\\[5pt]
 u^\varepsilon(\theta,1)=\alpha+\eta f(\theta),~~ v^\varepsilon(\theta,1)=0,\\[5pt]
 u^\varepsilon(\theta,r_0)=\beta+\eta g(\theta),~~ v^\varepsilon(\theta,r_0)=0,\label{NS-curvilnear}
\end{array}
\right.
\end{eqnarray}
where $(\theta,r)\in\Omega:=[0,2\pi]\times [r_0,1]$. Formally, as $\varepsilon\rightarrow 0$,
we obtain the steady Euler equations for $(u_e,v_e)$
\begin{align}\label{equ:Euler}
\left\{
\begin{aligned}
& u_e\partial_\theta u_e+v_er\partial_ru_e+u_ev_e+\partial_\theta p_e=0,\\
&u_e\partial_\theta v_e+v_er\partial_rv_e-(u_e)^2+r\partial_rp_e =0,\\
&\partial_\theta u_e+\partial_r(rv_e)=0.
\end{aligned}
\right.
\end{align}
We will show the existence of solution $(u^\varepsilon, v^\varepsilon)$ to (\ref{NS-curvilnear}) which converges to  a solution of steady Euler equations (\ref{equ:Euler}) satisfying (\ref{second ODE for leading Euler flow}). Following the argument of \cite{FGLT}, we first choose a steady Euler flow $(u_e(r),0)$ with $u_e(r)$ satisfying (\ref{second ODE for leading Euler flow}), then construct a class of Prandtl-Batchelor flows $(u^\varepsilon, v^\varepsilon)$ to (\ref{NS-curvilnear}) by perturbing this steady Euler flow.
Generally, there is a mismatch between the tangential velocities of the Euler flow $u_e$  and the prescribed Navier-Stokes flows $u^\varepsilon$. Due to the mismatch on the boundary, Prandtl in 1904 formally introduced the boundary layer theory to correct this mismatch. However, the justification of this formal boundary expansion is a challenging problem.  The studies on steady Prandtl equations and linearized steady Prandtl equations can be found in \cite{DM, GI3,OS, SWZ, WZ} and the validity of Prandtl boundary layer expansion has also some important progresses, see  \cite{GZ,GZ2,GN,GI1,Iyer1,Iyer2,Iyer3,IM2}. Moreover, the stability in Sobolev space for some class shear flow of Prandtl type has been studied in \cite{CWZ,GMM}.

In the current work, we assume that
\begin{align}\label{assumption on boundary}
\int_0^{2\pi}f(\theta)d\theta=\int_0^{2\pi}g(\theta)d\theta=0,
\end{align}
and take the leading order Euler flow $(u_e(\theta,r),v_e(\theta,r))$ to be the following shear flow
\begin{align}\label{Taylor-Couette flow}
u_e(\theta,r)=u_e(r), \ v_e(\theta,r)=0,
\end{align}
where $u_e(r)$ satisfies the second order ordinary differential equations
\begin{align}\label{second ODE for leading Euler flow}
\left\{
\begin{aligned}
&u''_e(r)+\frac{u'_e(r)}{r}-\frac{u_e(r)}{r^2}=-\bar{F}_u(r),\\
&u_e(1)=\Big(\alpha^2+\frac{\eta^2}{2\pi}\int_0^{2\pi}f^2(\theta)d\theta\Big)^{\frac12},\\ &u_e(r_0)=\Big(\beta^2+\frac{\eta^2}{2\pi}\int_0^{2\pi}g^2(\theta)d\theta\Big)^{\frac12},
\end{aligned}
\right.
\end{align}
where $\bar{F}_u(r)=\frac{1}{2\pi}\int_0^{2\pi}F_u(\theta,r)d\theta, r\in [r_0,1]$.
There is an important basis for this choice---the generalized Prandtl-Batchelor theory, see Appendix C. This theory shows that if the Euler flow is rotating shear $u_e(r) \vec{e}_\theta$ in the vanishing viscosity limit, then $u_e(r)$ must satisfy the second order ordinary differential equation in (\ref{second ODE for leading Euler flow}), and the Batchelor-Wood formula in Lemma \ref{BW} gives the boundary conditions for $u_e(r)$ in (\ref{second ODE for leading Euler flow}).
Moreover, the solvability of linearized second order Euler equations in the matched asymptotic expansion also makes $u_e(r)$ necessary to be this form, see Remark 2.1.
\begin{Remark}
If the external force $\mathbf{F}=\mathbf{0}$, then $u_e(r)=ar+\frac{b}{r}$($a,b$ can be uniquely determined by the boundary condition in (\ref{second ODE for leading Euler flow})) is the Couette-Taylor flow and the associated vorticity is constant $2a$. If the external force $\mathbf{F}=\frac{2c}{r}\mathbf{e}_\theta$, then $u_e(r)=ar+\frac{b}{r}+cr\ln r$ and the associated vorticity is $2a+c+2c\ln r$($a,b$ can be uniquely determined by the boundary condition in (\ref{second ODE for leading Euler flow})).
\end{Remark}

Next we introduce the steady Prandtl equations near $r=1$
\begin{align}
\left \{
\begin {array}{ll}
\big(u_e(1)+u_p^{(0)}\big)\partial_\theta u_p^{(0)}+\big(v_p^{(1)}-v_p^{(1)}(\theta,0)\big)\partial_Yu_p^{(0)}-\partial_{YY}u_p^{(0)}=0,\\[5pt]
\partial_\theta u_p^{(0)}+\partial_Yv_p^{(1)}=0,\\[5pt]
u_p^{(0)}(\theta,Y)=u_p^{(0)}(\theta+2\pi,Y),\ v_p^{(0)}(\theta,Y)=v_p^{(0)}(\theta+2\pi,Y),\\[5pt]
u_p^{(0)}\big|_{Y=0}=\alpha+\eta f(\theta)-u_e(1),\ \lim\limits_{Y\rightarrow -\infty}u_p^{(0)}=\lim\limits_{Y\rightarrow -\infty}v_p^{(1)}=0
\label{prandtl problem near 1}
\end{array}
\right.
\end{align}
and the steady Prandtl equations near $r=r_0$
\begin{eqnarray}
\left \{
\begin {array}{ll}
\big(u_e(r_0)+\hat{u}_p^{(0)}\big)\partial_\theta \hat{u}_p^{(0)}+\big( \hat{v}_p^{(1)}-\hat{v}_p^{(1)}(\theta,0)\big)r_0\partial_Z\hat{u}_p^{(0)}-r_0\partial_{ZZ}\hat{u}_p^{(0)}=0,\\[5pt]
\partial_\theta \hat{u}_p^{(0)}+r_0\partial_Z\hat{v}_p^{(1)}=0,\\[5pt]
\hat{u}_p^{(0)}(\theta,Z)=\hat{u}_p^{(0)}(\theta+2\pi,Z),\ \hat{v}_p^{(0)}(\theta,Z)=\hat{v}_p^{(0)}(\theta+2\pi,Z),\\[5pt]
\hat{u}_p^{(0)}\big|_{Z=0}=\beta+\eta g(\theta)-u_e(r_0),\  \lim\limits_{Z\rightarrow +\infty}\hat{u}_p^{(0)}=\lim\limits_{Z\rightarrow +\infty}\hat{v}_p^{(1)}=0,
\label{Prandtl problem another boundary}
\end{array}
\right.
\end{eqnarray}
where $Y=\frac{r-1}{\e}$ and $Z=\frac{r-r_0}{\e}$.
The above steady Prandtl equations will be derived by matched asymptotic expansions and their solvability will be studied in the next section.

Now our main theorem is stated as follows.

\begin{Theorem}\label{main theorem}
Assume that $f(\theta),g(\theta), F_u(\theta,r), F_v(\theta,r)$ are smooth functions satisfying (\ref{assumption on boundary}) and the solution $u_e(r)$ of (\ref{second ODE for leading Euler flow}) has a positive lower bound, then there exist $\varepsilon_0>0, \eta_0>0,C>0$ such that for any $\varepsilon\in (0,\varepsilon_0), \eta\in (0,\eta_0)$, the Navier-Stokes equations (\ref{NS-curvilnear}) have a solution $(u^\varepsilon(\theta,r), v^\varepsilon(\theta,r))$ which satisfies
\begin{align*}
\Big\|u^\varepsilon(\theta,r)-u_e(r)-u_p^{(0)}\Big(\theta,\frac{r-1}{\varepsilon}\Big)
-\hat{u}_p^{(0)}\Big(\theta,\frac{r-r_0}{\varepsilon}\Big)\Big\|_{L^\infty(\Omega)}\leq C\varepsilon, \\
\|v^\varepsilon\|_{L^\infty(\Omega)}\leq C\varepsilon,
\end{align*}
where $(u_e(r),0)$ is the shear flow satisfying (\ref{second ODE for leading Euler flow}), $(u_p^{(0)}, v_p^{(1)})$ is the solution of steady Prandtl equations  (\ref{prandtl problem near 1}) and $(\hat{u}_p^{(0)},\hat{v}_p^{(1)})$ is the solution of steady Prandtl equations (\ref{Prandtl problem another boundary}).

Moreover, for any $r_0<r_1<r_2<1$, there holds
\begin{align}
\lim_{\varepsilon\rightarrow 0}\|w^\varepsilon-w_e(r)\|_{L^\infty(A_{r_1,r_2})}=0, \nonumber
\end{align}
where $w^\varepsilon(\theta,r)$ is the vorticity of $(u^\varepsilon(\theta,r), v^\varepsilon(\theta,r))$, $w_e(r)$ is the vorticity of $(u_e(r), 0)$ and $A_{r_1,r_2}=B_{r_2}(0)\backslash B_{r_1}(0)$.
\end{Theorem}

\begin{Remark}
If $\bar{F}_u(r)\geq 0, \ \forall r\in [r_0,1]$, then the solution $u_e(r)$ of (\ref{second ODE for leading Euler flow}) has a positive lower bound by maximum principle.
\end{Remark}

\begin{Remark}
We note that the positivity assumption of $u_e(r)$ in Theorem \ref{main theorem} is a sharp condition for structural stability of shear flow $(u_e(r),0)$. The structural stability of rotating shear flow $(u_e(r),0)$ means that\begin{itemize}
                                                        \item  The rotating shear flow $(u_e(r),0)$ is a solution to (\ref{NS-curvilnear}) with external force $(F_u,F_v)=(u''_e(r)+\frac{u'_e(r)}{r}-\frac{u_e(r)}{r^2},0)$ and constant boundary conditions.
                                                        \item Under the small perturbation on external force and boundary conditions, the solution to (\ref{NS-curvilnear}) is still nearly rotating shear flow(``rotating shear flow''+``small error term'').
                                                      \end{itemize}
  On the one hand, Theorem \ref{main theorem} shows that if $u_e(r)$ has a positive lower bound, then the shear flow $(u_e(r),0)$ is structurally stable.  On the other hand, if $u_e(r)$ has a degenerate point, small perturbations on boundary and force will cause hyperbolic singularities for the Euler flow,  then the ``cat eye" structure may appear, thus we can't expect structural stability.  Hence, in this sense, the positivity of $u_e(r)$ is a sharp condition for structural stability of shear flow $(u_e(r),0)$ and also for the existence of PB flow.
\end{Remark}

\begin{Remark}
 We write $\Omega(r)=\frac{u_e(r)}{r}$ as the angular velocity of the Euler flow. From the Rayleigh equation for rotating shear flow
$$(\Omega-c)\Big(D_\ast D-\frac{\alpha^2}{r^2}\Big)\phi-\frac{(rD^2\Omega+3D\Omega)\phi}{r}=0,~~ D=\frac{d}{dr},D_\ast=\frac{d}{dr}+\frac{1}{r},$$
it is easy to deduce that if the shear flow $(u_e(r),0)$ is dynamically unstable for the Euler equations, $(u_e(r)+Cr,0)$ is also dynamically unstable(Lyapunov unstable) for any constant $C$. Furthermore, it is expected that if the shear flow $(u_e(r),0)$ is dynamically unstable for the Euler equations, it is also dynamically unstable for the Navier-Stokes equations with small viscosity. For any dynamically unstable shear flow $(u_e(r),0)$, we choose $C$ such that $u_e(r)+Cr$ has positive below bound. Thus the shear flow $(u_e(r)+Cr,0)$ is dynamically unstable, but structurally stable from Theorem \ref{main theorem}. Hence, dynamical stability and structural stability are totally separate.
\end{Remark}

\begin{Remark}
The condition $\int_0^{2\pi}f(\theta)d\theta=\int_0^{2\pi}g(\theta)d\theta=0$ can be dropped due to the fact
$$\alpha+\eta f(\theta)=\alpha+ \frac{\eta}{2\pi}\int_0^{2\pi}f(\theta)d\theta +\eta \tilde{f}(\theta), $$
where $\int_0^{2\pi} \tilde{f}(\theta) d\theta=0.$ Moreover, the smoothness of $f(\theta),g(\theta), F_u(\theta,r), F_v(\theta,r)$ can be relaxed, but we don't pursue this issue in this manuscript.
\end{Remark}

\begin{Remark}
We believe that our current method can be applicable to the channel flow. More precisely, we consider the two-dimensional forced steady Navier-Stokes equations
\begin{equation}
\left \{
\begin {array}{ll}
\mathbf{u}^\varepsilon\cdot\nabla\mathbf{u}^\varepsilon+\nabla p^\varepsilon-\varepsilon^2\Delta\mathbf{u}^\varepsilon=\varepsilon^2 \mathbf{F}, ~~ (x,y)\in [0,2\pi]\times [0,1]\\[5pt]
\nabla\cdot \mathbf{u}^\varepsilon=0,~~ (x,y)\in [0,2\pi]\times [0,1]\\[5pt]
\mathbf{u}^\varepsilon(x+2\pi,y)=\mathbf{u}^\varepsilon(x,y), ~~ y\in [0,1]\\[5pt]
\mathbf{u}^\varepsilon(x,0)=(\alpha+\eta f(x))\vec{\mathbf{e}}_1, \quad \mathbf{u}^\varepsilon(x,0)=(\beta+\eta g(x))\vec{\mathbf{e}}_1,
\end{array}
\right.\label{ns in channel}
\end{equation}
where $\alpha>0, \beta >0, \vec{\mathbf{e}}_1=(1,0)^T $, $f(x),g(x)$ are $2\pi$-periodic smooth functions and $\eta$ is a small number. In Appendix C, we proved that under some assumption the inviscid limit of the solutions $\mathbf{u}^\varepsilon$ to (\ref{ns in channel}) must satisfy $u''_e(y)=\frac{1}{2\pi}\int_0^{2\pi}F_1(x,y)dx, \ \forall y\in [0,1]$, where $F_1$ is the first component of $\mathbf{F}$. We believe that following the argument of the current paper, when the boundary conditions are slightly different with uniform motion, we can construct solutions $\mathbf{u}^\varepsilon$ to (\ref{ns in channel}) which convergence to the shear flow $(u_e(y),0)^T$ when $\varepsilon\rightarrow 0$.
\end{Remark}

Now we present a sketch of the proof and some key ideas.

{\bf Step 1: Construction of the approximate solution.} We construct an approximate solution $(u^a,v^a)$ by matched asymptotic expansion. The leading order of $(u^a,v^a)$ is
\begin{align*}
\Big(u_e(r)+u_p^{(0)}\Big(\theta,\frac{r-1}{\varepsilon}\Big)+\hat{u}_p^{(0)}\Big(\theta,\frac{r-r_0}{\varepsilon}\Big),0\Big).
\end{align*}
The details of constructing the approximate solution will be given in Section 2. After the construction of the approximate solution, we derive the  equations \eqref{error equation} for the error $(u,v):=(u^\varepsilon-u^a,v^\varepsilon-v^a)$.
Notice that the nonlinear term can be easily handled by higher order approximation, hence we only need to consider the linearized error equations
 \begin{align}\label{linearized error equation}
\left\{
\begin{array}{lll}
u^au_\theta+uu^a_\theta+v^aru_r+vru^a_r+v^au+vu^a+p_\theta-\varepsilon^2\big( ru_{rr}
+\frac{u_{\theta\theta}}{r}+2\frac{v_{\theta}}{r}+u_r-\frac{u}{r}\big)=F_1,\\[5pt]
u^av_\theta+uv^a_\theta+v^arv_r+vrv^a_r-2u u^a+rp_r-\varepsilon^2 \big( rv_{rr}+\frac{v_{\theta\theta}}{r}-2\frac{u_{\theta}}{r}+v_r-\frac{v}{r}\big)=F_2,\\[5pt]
u_\theta+(rv)_r=0,  \\[5pt]
u(\theta+2\pi,r)=u(\theta,r), \ v(\theta+2\pi,r)=v(\theta,r), \\[5pt]
u(\theta,1)=0,\ v(\theta,1)=0, \\[5pt]
 u(\theta,r_0)=0,\ v(\theta,r_0)=0.
\end{array}
\right.
\end{align}

{\bf Step 2: Linear stability estimate for (\ref{linearized error equation}).}
Equations (\ref{linearized error equation}) are the linearized Navier-Stokes equations around the approximate solution $(u^a,v^a)$. For simply, we assume $(u^a,v^a)\approx\Big(u_e(r)+u_p^{(0)}\big(\theta,\frac{r-1}{\varepsilon}\big),0\Big)$ because we can deal with the boundary layer profile near $r=r_0$  in the similar way. Since $|u^a_\theta|=|\partial_\theta u^{(0)}_{p}|\lesssim \eta$, the leading order of the above system can be simplified as
\begin{align}\label{simplied equation}
\left\{
\begin{array}{lll}
u^au_\theta+vru^a_r+vu^a+p_\theta-\varepsilon^2\big( ru_{rr}
+\frac{u_{\theta\theta}}{r}+2\frac{v_{\theta}}{r}+u_r-\frac{u}{r}\big)=F_1,\\[5pt]
u^av_\theta-2u u^a+rp_r-\varepsilon^2 \big( rv_{rr}+\frac{v_{\theta\theta}}{r}-2\frac{u_{\theta}}{r}+v_r-\frac{v}{r}\big)=F_2,\\[5pt]
u_\theta+(rv)_r=0,  \\[5pt]
u(\theta+2\pi,r)=u(\theta,r), \ v(\theta+2\pi,r)=v(\theta,r), \\[5pt]
u(\theta,1)=0,\ v(\theta,1)=0, \\[5pt]
 u(\theta,r_0)=0,\ v(\theta,r_0)=0,
 \end{array}
 \right.
\end{align}
where $u^a$ can be regarded as $u_e(r)+u_p^{(0)}\big(\theta,\frac{r-1}{\varepsilon}\big)$.

The linear stability estimate consists of a basic energy estimate and positivity estimate. In fact, it is easy to know that the basic energy estimate is not good enough for closing the estimate because $\e$ is small. The key point is the following observation which gives the important positivity estimate: $u^a$ is strictly positive, we should make use of the terms $u^au_\theta$ and $u^av_\theta$ to obtain a positive quantity from the convective term. To do this, we choose $\big(-\big(\frac{r^2v}{u^a}\big)_r,\big(\frac{r v}{u^a}\big)_\theta\big)$ as the multiplier, the pressure is eliminated due to the divergence-free condition and the diffusion term is dominated easily due to the small viscosity $\e^2$. Here we only give a sketch of deriving the key positive quantity as follows. Firstly one has
\begin{align*}
&-\int_{r_0}^{1}\int_0^{2\pi}\big(u^au_\theta+vru^a_r+vu^a\big)\Big(\frac{r^2v}{u^a}\Big)_rd\theta dr\\
&\quad \quad +\int_{r_0}^{1}\int_0^{2\pi}\big(u^av_\theta-2uu^a\big)\Big(\frac{rv}{u^a}\Big)_\theta d\theta dr\\
=&\int_{r_0}^{1}\int_0^{2\pi}\big(u^au_\theta+vru^a_r+vu^a\big)\Big(\frac{ru_\theta}{u^a}-rv\Big(\frac{r}{u^a}\Big)_r\Big)d\theta dr\\
&\quad \quad +\int_{r_0}^{1}\int_0^{2\pi}\big(u^av_\theta-2uu^a\big)\Big(\frac{rv_\theta}{u^a}+rv\Big(\frac{1}{u^a}\Big)_\theta\Big) d\theta dr\\
\approx&\int_{r_0}^{1}\int_0^{2\pi}\big(u^au_\theta+vru^a_r+vu^a\big)\Big(\frac{ru_\theta}{u^a}-rv\Big(\frac{r}{u^a}\Big)_r\Big)d\theta dr\\
&\quad \quad +\int_{r_0}^{1}\int_0^{2\pi}\big(u^av_\theta-2uu^a\big)\Big(\frac{rv_\theta}{u^a}\Big) d\theta dr\\
\approx&\int_{r_0}^{1}\int_0^{2\pi}\big(ru^2_\theta +rv^2_\theta \big)d\theta dr+\int_{r_0}^{1}\int_0^{2\pi}\frac{1}{u^a}\Big(ru^a_{rr}+u^a_r-\frac{u^a}{r}\Big)(rv)^2d\theta dr,
\end{align*}
here the details of obtaining the last step will be given in subsection \ref{section:Linear stability estimate}.

Since $u^a\approx u_e(r)+u^{(0)}_p(\theta,\frac{r-1}{\e})$, the boundary layer profile $u^{(0)}_p(\theta,Y)$ is small in the Prandtl variable $Y=\frac{r-1}{\e}$, we deduce that
\begin{align*}
&\int_{r_0}^{1}\int_0^{2\pi}\frac{1}{u^a}\Big(r\p^2_r u^{(0)}_p+\p_r u^{(0)}_p-\frac{u^{(0)}_p}{r}\Big)(rv)^2d\theta dr\\
=&\int_{r_0}^{1}\int_0^{2\pi}\frac{1}{u^a}\Big(rY^2\p^2_Y u^{(0)}_p+\e Y^2\p_Y u^{(0)}_p-\frac{(\e Y)^2u^{(0)}_p}{r}\Big)\Big(\frac{rv}{r-1}\Big)^2d\theta dr\\
\lesssim &\eta \int_{r_0}^{1}\int_0^{2\pi}\Big(\frac{rv}{r-1}\Big)^2d\theta dr\lesssim \eta \int_{r_0}^{1}\int_0^{2\pi}\big((rv)_r\big)^2 d\theta dr \lesssim \eta \int_{r_0}^{1}\int_0^{2\pi}u_\theta^2 d\theta dr,
\end{align*}
where we have used the Hardy inequality.

 Moreover, if the Euler flow is Taylor-Coutte flow, that is $u_e(r)=a r+\frac{b}{r}$, then $ru''_e+u'_e-\frac{u_e}{r}=0$. So
\begin{align*}
&\int_{r_0}^{1}\int_0^{2\pi}\big(ru^2_\theta +rv^2_\theta \big)d\theta dr+\int_{r_0}^{1}\int_0^{2\pi}\frac{1}{u^a}\Big(ru^a_{rr}+u^a_r-\frac{u^a}{r}\Big)(rv)^2d\theta dr\\
&\geq(1-C\eta)\int_{r_0}^{1}\int_0^{2\pi}\big(ru^2_\theta +rv^2_\theta \big)d\theta dr,
\end{align*}
where the constant $C$ is independent of $\eta$. So we obtain a positive quantity $\|(u_\theta,v_\theta)\|_2$. When $u_e$ is strictly positive, we can also obtain the same estimate which will be given in the text.

 However, the positive quantity $\|(u_\theta,v_\theta)\|_2$ don't contain zero-frequency of $(u,v)$ which is needed to obtain the $L^\infty$ estimate for the error $(u,v)$. Notice that $\int_0^{2\pi}v(\theta,r) d\theta=0$ because of the divergence-free condition and the boundary conditions, we can dominate $\|v\|_2$ by $\|v_\theta\|_2$ using the  Poincar\'{e} inequality. However $u_0(r):=\frac{1}{2\pi}\int_0^{2\pi}u(\theta,r) d\theta\neq 0$, we need to obtain the estimate of $|u_0(r)|$ from the basic energy estimate.

 Now we give a sketch of the basic energy estimate.  Choose $u_0$ as a multiplier to the first equation in (\ref{simplied equation}). The diffusion term is
\begin{align*}
&\int_{r_0}^{1}\int_0^{2\pi}-\e^2\Big( ru_{rr}
+\frac{u_{\theta\theta}}{r}+2\frac{v_{\theta}}{r}+u_r-\frac{u}{r}\Big)u_0(r)d\theta dr=\e^2\int_{r_0}^{1}\int_0^{2\pi}\Big(r|u'_0|^2+\frac{u_0^2}{r}\Big)d\theta dr.
\end{align*}
We decompose the approximate solution as $u^a\approx u_e(r)+u^{(0)}_p(\theta,\frac{r-1}{\e})$.  Due to the Euler flow is radial, we deduce that the following quantity vanishes
\begin{align*}
&\int_{r_0}^{1}\int_0^{2\pi}\big( u_eu_\theta+vru_{er}+vu_e+p_\theta\big)u_0d\theta dr=0,
\end{align*}
where we used $\int_0^{2\pi}v(\theta,r) d\theta=0$. Moreover, the Prandtl part can be handled by the Hardy inequality as above
\begin{align*}
&\int_{r_0}^{1}\int_0^{2\pi}\big( u_p^{(0)}u_\theta+vr\p_r u_{p}^{(0)}+vu_p^{(0)}\big)u_0d\theta dr\\
=&\int_{r_0}^{1}\int_0^{2\pi}\e\big( Yu_p^{(0)}u_\theta+\frac{vr}{r-1} Y^2\p_Y u_{p}^{(0)}+v Yu_p^{(0)}\big)\frac{u_0}{r-1}d\theta dr
\lesssim\eta\e\|(u_\theta,v_\theta)\|_2\|u'_0\|_2.
\end{align*}
Thus, we obtain that $\varepsilon^2\|u'_0\|_2\lesssim \|(u_\theta,v_\theta)\|_2$. One can see Lemma \ref{basic energy estimate} for the basic energy estimate.

 Combining the positivity estimate and basic energy estimate, we obtain the linear stability of equations (\ref{simplied equation}).

{\bf Step 3: $L^\infty$ estimate for (\ref{linearized error equation}).} To close the nonlinearity, we need the $L^\infty$ estimate. Using the anisotropic Sobolev embedding, we only need to estimate $\|(u_{r\theta},v_{r\theta})\|_2$ which can be obtained by studying the Stokes equations.

Finally, combining the linear stability estimate and  $L^\infty$ estimate, we establish the well-posedness of the error equations by contraction mapping theorem.\\

The paper is organized as follows. In Section 2, we construct the approximate solution by the matched asymptotic expansion method. In Section 3, we derive the error equations and establish the linear stability estimate which consists of the basic energy estimate and positivity estimate. In Section 4, we firstly establish the ``$H^2$ estimate" for a Stokes system, then obtain $L^\infty$ estimate by the anisotropic Sobolev embedding, finally complete the proof of Theorem \ref{main theorem} by combining the linear stability estimate and  $L^\infty$ estimate.

\smallskip


\section{Construction of approximate solutions}
\indent

In this section, we construct an approximate solution of the Navier-Stokes equations (\ref{NS-curvilnear}) by the matched asymptotic expansion.

\subsection{Euler expansions}

\indent

Away from the boundary, we make the following formal expansions
\begin{align*}
&u^{\varepsilon}(\theta,r)=u_e^{(0)}(\theta,r)+\varepsilon u_e^{(1)}(\theta,r)+\text{h.o.t.},\\[5pt]
&v^{\varepsilon}(\theta,r)=v_e^{(0)}(\theta,r)+\varepsilon v_e^{(1)}(\theta,r)+\text{h.o.t.},\\[5pt]
&p^{\varepsilon}(\theta,r)=p_e^{(0)}(\theta,r)+\varepsilon p_e^{(1)}(\theta,r)+\text{h.o.t.},
\end{align*}
here and in what follows, ``h.o.t." means higher order terms.

\subsubsection{Equations for $(u_e^{(0)},v_e^{(0)},p_e^{(0)})$}

\indent

By substituting the above expansions into (\ref{NS-curvilnear}) and collecting the $\varepsilon$-zeroth order terms, we deduce that $(u_e^{(0)},v_e^{(0)},p_e^{(0)})$ satisfies the following steady nonlinear Euler equations
\begin{eqnarray}
\left \{
\begin {array}{ll}
u_e^{(0)} \partial_\theta u_e^{(0)}+rv_e^{(0)} \partial_ru_e^{(0)}+u_e^{(0)} v_e^{(0)}+\partial_\theta p_e^{(0)}=0,\\[7pt]
 u_e^{(0)} \partial_\theta v_e^{(0)}+rv_e^{(0)} \partial_rv_e^{(0)}-(u_e^{(0)})^2+r\partial_rp_e^{(0)}=0,\\[7pt]
 \partial_\theta u_e^{(0)}+\partial_r(rv_e^{(0)})=0.\label{outer-leading order equation}
\end{array}
\right.
\end{eqnarray}

In our current work, we assume that the leading order Euler flow $(u_e^{(0)},v_e^{(0)})$ is the shear flow in (\ref{Taylor-Couette flow}). Then equations (\ref{outer-leading order equation}) lead to
 \begin{align}
\partial_\theta p_e^{(0)}(\theta,r)=0,\  \partial_rp_e^{(0)}(\theta,r)=\frac{1}{r}u_e^2.\nonumber
 \end{align}
 Thus we can write
 \begin{align}
 p_e^{(0)}(\theta,r)=p_e(r),\  p'_e(r)=\frac{1}{r}u_e^2(r).\label{outer-leading order pressure}
 \end{align}

\subsubsection{Equations for $(u_e^{(1)},v_e^{(1)},p_e^{(1)})$}
\indent

By collecting the $\varepsilon$-order terms, we deduce that $(u_e^{(1)},v_e^{(1)},p_e^{(1)})$ satisfies the following linearized Euler equations in $\Omega$
\begin{eqnarray}
\left \{
\begin {array}{ll}
u_e(r) \partial_\theta u_e^{(1)}+rv_e^{(1)} u_e'(r)+u_e(r) v_e^{(1)}+\partial_\theta p_e^{(1)}=0,\\[5pt]
u_e(r) \partial_\theta v_e^{(1)}-2u_e(r)u_e^{(1)}+r\partial_rp_e^{(1)}=0,\\[5pt]
\partial_\theta u_e^{(1)}+\partial_r(rv_e^{(1)})=0,\label{outer-1 order equation}
\end{array}
\right.
\end{eqnarray}
which are equipped with the boundary conditions
\begin{align}
 v_e^{(1)}|_{r=1}=-v_p^{(1)}|_{Y=0},\ v_e^{(1)}|_{r=r_0}=-\hat{v}_p^{(1)}|_{Z=0},\ v_e^{(1)}(\theta,r)=v_e^{(1)}(\theta+2\pi,r) ,\label{outer-1 order-bc}
\end{align}
where $v_p^{(1)}, \hat{v}_p^{(1)}$ is defined in \eqref{prandtl problem 1} and  \eqref{prandtl problem another boundary} respectively in next subsection.

\subsection{Prandtl expansions near $\partial B_1$}
\indent

We  introduce the scaled variable  $Y=\frac{r-1}{\varepsilon}\in (-\infty,0]$ and make the following Prandtl expansions near $\partial B_1$
\begin{align}\label{expansion near boundary}
\begin{aligned}
&u^\varepsilon=u_e^{(0)}(\theta,r)+u_p^{(0)}(\theta,Y)+\varepsilon\big[u_e^{(1)}(\theta,r)+u_p^{(1)}(\theta,Y)\big]
+\varepsilon^2\big[u_e^{(2)}(\theta,r)+u_p^{(2)}(\theta,Y)\big]+\text{h.o.t.},\\[5pt]
&v^\varepsilon=v_e^{(0)}(\theta,r)+v_p^{(0)}(\theta,Y)+\varepsilon\big[v_e^{(1)}(\theta,r)+v_p^{(1)}(\theta,Y)\big]
+\varepsilon^2\big[v_e^{(2)}(\theta,r)+v_p^{(2)}(\theta,Y)\big]+\text{h.o.t.},\\[5pt]
&p^\varepsilon=p_e^{(0)}(\theta,r)+p_p^{(0)}(\theta,Y)+\varepsilon\big[p_e^{(1)}(\theta,r)+p_p^{(1)}(\theta,Y)\big]
+\varepsilon\big[p_e^{(2)}(\theta,r)+p_p^{(2)}(\theta,Y)\big]+\text{h.o.t.},
\end{aligned}
\end{align}
where as $Y\rightarrow -\infty$
\begin{align}
 \partial_\theta^l\partial_Y^mv_p^{(i)}(\theta,Y)\rightarrow 0, \ \partial_\theta^l\partial_Y^mp_p^{(i)}(\theta,Y)\rightarrow 0,\label{matching condition}
\end{align}
here $l,m\geq 0, i=0,1,\cdots,$ and satisfy the following boundary conditions
\begin{align*}
&u_e^{(0)}(\theta,1)+u_p^{(0)}(\theta,0)=\alpha+\eta f(\theta), \ u_e^{(i)}(\theta,1)+u_p^{(i)}(\theta,0)=0,\ i\geq 1,\\
&v_e^{(i)}(\theta,1)+v_p^{(i)}(\theta,0)=0, i\geq 0, \ \lim_{Y\rightarrow -\infty}(u_p^{(0)}(\theta,Y),\partial_Yu_p^{(1)}(\theta,Y))=(0,0).
\end{align*}

\subsubsection{Equations for $(v_p^{(0)}, p_p^{(0)})$}

\indent

By substituting the above expansions into (\ref{NS-curvilnear}) and collecting the $\frac{1}{\varepsilon}$ order terms, we get
\begin{align*}
\partial_Yv_p^{(0)}(\theta,Y)=0,\  \partial_Yp_p^{(0)}(\theta,Y)=0,
\end{align*}
which together with (\ref{matching condition}) implies
\begin{align*}
v_p^{(0)}=0, \ p_p^{(0)}=0.
\end{align*}

\subsubsection{Equations for $(u_p^{(0)},v_p^{(1)},p_p^{(1)})$}

\indent

By substituting the above expansions into (\ref{NS-curvilnear}) and collecting the $\varepsilon$-zeroth order terms, we obtain
 the following steady Prandtl equations for $(u_p^{(0)},v_p^{(1)})$
\begin{eqnarray}
\left \{
\begin {array}{ll}
\big(u_e(1)+u_p^{(0)}\big)\partial_\theta u_p^{(0)}+\big( v_p^{(1)}- v_p^{(1)}(\theta,0)\big)\partial_Yu_p^{(0)}-\partial_{YY}u_p^{(0)}=0,\\[5pt]
\partial_\theta u_p^{(0)}+\partial_Yv_p^{(1)}=0,\\[5pt]
u_p^{(0)}(\theta,Y)=u_p^{(0)}(\theta+2\pi,Y),\ v_p^{(0)}(\theta,Y)=v_p^{(0)}(\theta+2\pi,Y),\\[5pt]
u_p^{(0)}\big|_{Y=0}=\alpha+\eta f(\theta)-u_e(1),\  \lim\limits_{Y\rightarrow -\infty}u_p^{(0)}=\lim\limits_{Y\rightarrow -\infty}v_p^{(1)}=0
\label{prandtl problem 1}
\end{array}
\right.
\end{eqnarray}
and the pressure $p_p^{(1)}$ satisfies
\begin{align}\label{equation of first pressure}
\partial_Yp_p^{(1)}(\theta, Y)=(u_p^{(0)})^2(\theta,Y)+2u_e(1)u_p^{(0)}(\theta,Y), \ \lim_{Y\rightarrow -\infty}p_p^{(1)}(\theta,Y)=0.
\end{align}

\subsubsection{Equations for $(u_p^{(1)},v_p^{(2)})$}

\indent

By substituting the above expansions into the first and third equation of (\ref{NS-curvilnear}) and collecting the $\varepsilon$-order terms, we obtain
 the following linearized steady Prandtl equations for $(u_p^{(1)},v_p^{(2)})$
\begin{eqnarray}
\left \{
\begin {array}{ll}
\big(u_e(1)+u_p^{(0)}\big)\partial_\theta u_p^{(1)}+\big(v_e^{(1)}(\theta,1)+ v_p^{(1)}\big)\partial_Yu_p^{(1)}+(v_p^{(2)}-v_p^{(2)}(\theta,0))\partial_{Y}u_p^{(0)}\\[5pt]
\quad \quad \quad \quad +u_p^{(1)}\partial_\theta u_p^{(0)}-\partial_{YY}u_p^{(1)}=f_1(\theta,Y),\\[5pt]
\partial_\theta u_p^{(1)}+\partial_Yv_p^{(2)}+\partial_Y(Yv_p^{(1)})=0,\\[5pt]
u_p^{(1)}(\theta,Y)=u_p^{(1)}(\theta+2\pi,Y),\ v_p^{(2)}(\theta,Y)=v_p^{(2)}(\theta+2\pi,Y),\\[5pt]
u_p^{(1)}\big|_{Y=0}=-u_e^{(1)}\big|_{r=1},\ \lim\limits_{Y\rightarrow -\infty}\partial_Yu_p^{(1)}=\lim\limits_{Y\rightarrow -\infty}v_p^{(2)}=0
\label{first linearized prandtl problem near 1}
\end{array}
\right.
\end{eqnarray}
 and
\begin{align*}
f_1(\theta,Y)=&-\partial_\theta p_p^{(1)}+Y\partial_{YY}u_p^{(0)}+\partial_Yu_p^{(0)}\\[5pt]
&-u_p^{(0)}\big(\partial_\theta u_e^{(1)}(\theta,1)+v_e^{(1)}(\theta,1)+v_p^{(1)}\big)-(u'_e(1)Y+u_e^{(1)}(\theta,1))\partial_\theta u_p^{(0)}\\[5pt]
&-(\partial_rv_e^{(1)}(\theta,1)+v_e^{(1)}(\theta,1))Y\partial_Y u_p^{(0)}-(u_e'(1)+Y\partial_Yu_p^{(0)}+u_e(1))v_p^{(1)}.
\end{align*}
\subsection{Prandtl expansions near $\partial B_{r_0}$}

\indent

Similarly as above, we introduce the scaled variable  $Z=\frac{r-r_0}{\varepsilon}\in[0,+\infty)$ and make the following Prandtl expansions near $\partial B_{r_0}$
\begin{align*}
&u^\varepsilon=u_e^{(0)}(\theta,r)+\widehat{u}_p^{(0)}(\theta,Z)+\varepsilon\big[u_e^{(1)}(\theta,r)+\widehat{u}_p^{(1)}(\theta,Z)\big]
+\varepsilon^2\big[u_e^{(2)}(\theta,r)+\widehat{u}_p^{(2)}(\theta,Z)\big]+\text{h.o.t.},\\[5pt]
&v^\varepsilon=v_e^{(0)}(\theta,r)+\widehat{v}_p^{(0)}(\theta,Z)+\varepsilon\big[v_e^{(1)}(\theta,r)+\widehat{v}_p^{(1)}(\theta,Z)\big]+
\varepsilon^2\big[v_e^{(2)}(\theta,r)+\widehat{v}_p^{(2)}(\theta,Z)\big]+\text{h.o.t.},\\[5pt]
&p^\varepsilon=p_e^{(0)}(\theta,r)+\widehat{p}_p^{(0)}(\theta,Z)+\varepsilon\big[p_e^{(1)}(\theta,r)+\widehat{p}_p^{(1)}(\theta,Z)\big]+
\varepsilon^2\big[p_e^{(1)}(\theta,r)+\widehat{p}_p^{(1)}(\theta,Z)\big]+\text{h.o.t.},
\end{align*}
where as $Z\rightarrow +\infty$
\begin{align*}
 \partial_\theta^l\partial_Z^m\widehat{v}_p^{(i)}(\theta,Z)\rightarrow 0, \ \partial_\theta^l\partial_Z^m\widehat{p}_p^{(i)}(\theta,Z)\rightarrow 0,
\end{align*}
here $l,m\geq 0, i=0,1,\cdots,$ and  there hold the following boundary conditions
\begin{align*}
&u_e^{(0)}(\theta,r_0)+\hat{u}_p^{(0)}(\theta,0)=\beta+\eta g(\theta), \ u_e^{(i)}(\theta,r_0)+\hat{u}_p^{(i)}(\theta,0)=0,\ i\geq 1,\\
&v_e^{(i)}(\theta,r_0)+\hat{v}_p^{(i)}(\theta,0)=0, i\geq 0, \ \lim_{Z\rightarrow +\infty}(\hat{u}_p^{(0)}(\theta,Z),\partial_Z\hat{u}_p^{(1)}(\theta,Z))=(0,0).
\end{align*}

\subsubsection{Equations for $(\hat{v}_p^{(0)}, \hat{p}_p^{(0)})$}

\indent

By substituting the above expansions into (\ref{NS-curvilnear}) and collecting the $\frac{1}{\varepsilon}$ order terms, we get
\begin{align*}
\partial_Z\hat{v}_p^{(0)}(\theta,Z)=0,\  \partial_Z\hat{p}_p^{(0)}(\theta,Z)=0,
\end{align*}
which together with (\ref{matching condition}) implies
\begin{align*}
\hat{v}_p^{(0)}=0,\  \hat{p}_p^{(0)}=0.
\end{align*}

\subsubsection{Equations for $(u_p^{(0)},v_p^{(1)},p_p^{(1)})$}

\indent

By substituting the above expansions into (\ref{NS-curvilnear}) and collecting the $\varepsilon$-zeroth order terms, we obtain
 the following steady Prandtl equations for $(\hat{u}_p^{(0)},\hat{v}_p^{(1)})$
\begin{eqnarray}
\left \{
\begin {array}{ll}
\big(u_e(r_0)+\hat{u}_p^{(0)}\big)\partial_\theta \hat{u}_p^{(0)}+\big(\hat{v}_p^{(1)}-\hat{v}_p^{(1)}(\theta,0)\big)r_0\partial_Z\hat{u}_p^{(0)}-r_0\partial_{ZZ}\hat{u}_p^{(0)}=0,\\[5pt]
\partial_\theta \hat{u}_p^{(0)}+r_0\partial_Z\hat{v}_p^{(1)}=0,\\[5pt]
\hat{u}_p^{(0)}(\theta,Z)=\hat{u}_p^{(0)}(\theta+2\pi,Z),\ \hat{v}_p^{(0)}(\theta,Z)=\hat{v}_p^{(0)}(\theta+2\pi,Z),\\[5pt]
\hat{u}_p^{(0)}\big|_{Z=0}=\beta+\eta g(\theta)-u_e(r_0),\  \lim\limits_{Z\rightarrow +\infty}\hat{u}_p^{(0)}=\lim\limits_{Z\rightarrow +\infty}\hat{v}_p^{(1)}=0
\label{prandtl problem another boundary}
\end{array}
\right.
\end{eqnarray}
and the pressure $\hat{p}_p^{(1)}$ satisfies
\begin{align}\label{equation of first pressure another}
r_0\partial_Z\hat{p}_p^{(1)}(\theta, Z)=(\hat{u}_p^{(0)})^2(\theta,Z)+2u_e(r_0)\hat{u}_p^{(0)}(\theta,Z), \ \lim_{Z\rightarrow +\infty}\hat{p}_p^{(1)}(\theta,Z)=0.
\end{align}

\subsubsection{Equations for $(\hat{u}_p^{(1)},\hat{v}_p^{(2)})$}

\indent

By substituting the above expansions into the first and third equation of (\ref{NS-curvilnear}) and collecting the $\varepsilon$-order terms, we obtain
 the following linearized steady Prandtl equations for $(\hat{u}_p^{(1)},\hat{v}_p^{(2)})$
\begin{eqnarray}
\left \{
\begin {array}{ll}
\big(u_e(r_0)+\hat{u}_p^{(0)}\big)\partial_\theta \hat{u}_p^{(1)}+\big(v_e^{(1)}(\theta,r_0)+ \hat{v}_p^{(1)}\big)r_0\partial_Z\hat{u}_p^{(1)}+\hat{u}_p^{(1)}\partial_\theta \hat{u}_p^{(0)}\\[5pt]
\quad \quad \quad \quad \quad \quad \quad \quad \quad +(\hat{v}_p^{(2)}-\hat{v}_p^{(2)}(\theta,0))r_0\partial_Z\hat{u}_p^{(0)}-r_0\partial_{ZZ}\hat{u}_p^{(1)}=\hat{f}_1(\theta,Z),\\[5pt]
\partial_\theta \hat{u}_p^{(1)}+r_0\partial_Z\hat{v}_p^{(2)}+\partial_Z(Z\hat{v}_p^{(1)})=0,\\[5pt]
\hat{u}_p^{(1)}(\theta,Z)=\hat{u}_p^{(1)}(\theta+2\pi,Z),\ \hat{v}_p^{(2)}(\theta,Z)=\hat{v}_p^{(2)}(\theta+2\pi,Z),\\[5pt]
\hat{u}_p^{(1)}\big|_{Z=0}=-u_e^{(1)}\big|_{r=r_0},\  \lim\limits_{Z\rightarrow +\infty}\partial_Z\hat{u}_p^{(1)}=\lim\limits_{Z\rightarrow +\infty}\hat{v}_p^{(2)}=0,
\label{first linearized prandtl problem another boundary}
\end{array}
\right.
\end{eqnarray}
where
\begin{align*}
\hat{f}_1(\theta,Y)=&-\partial_\theta \hat{p}_p^{(1)}+Z\partial_{ZZ}\hat{u}_p^{(0)}+\partial_Z\hat{u}_p^{(0)}\\[5pt]
&-\hat{u}_p^{(0)}(\partial_\theta u_e^{(1)}(\theta,r_0)+v_e^{(1)}(\theta,r_0)+\hat{v}_p^{(1)})-(u'_e(r_0)Z+u_e^{(1)}(\theta,r_0))\partial_\theta \hat{u}_p^{(0)}\\[5pt]
&-(r_0\partial_rv_e^{(1)}(\theta,r_0)+v_e^{(1)}(\theta,r_0)+\hat{v}_p^{(1)})Z\partial_Z u_p^{(0)}-[u_e(r_0)+r_0u'_e(r_0)]\hat{v}_p^{(1)}.
\end{align*}

Before performing the higher order expansions we then aim to solve the above Euler equations and Prandtl equations.
\subsection{Solvabilities of Euler equations and Prandtl equations}
\indent

The order in which we solve the equation is as follows
\begin{align*}
(u_e(r),0)\rightarrow (u_p^{(0)},v_p^{(1)})/(\hat{u}_p^{(0)},\hat{v}_p^{(1)})\rightarrow (u_e^{(1)},v_e^{(1)})\rightarrow(u_p^{(1)},v_p^{(2)})/(\hat{u}_p^{(1)},\hat{v}_p^{(2)}).
\end{align*}

\subsubsection{Prandtl system and its solvability}
\indent

We first derive some necessary conditions for the solvability of Prandtl equations, one can also refer to Lemma 2.1 in \cite{FGLT}.
\begin{lemma}\label{BW}(Batchelor-Wood formula \cite{K1998}\cite{K2000}) Let $\int_0^{2\pi}f(\theta)d\theta=\int_0^{2\pi}g(\theta)d\theta=0$. If the nonlinear Prandtle equations (\ref{prandtl problem 1}) has a solution $(u_p^{(0)}, v_p^{(1)})$ and (\ref{prandtl problem another boundary}) have a solution $(\hat{u}_p^{(0)}, \hat{v}_p^{(1)})$ which satisfy
\beno
&&u_e(1)+u_p^{(0)}(\theta,Y)>0, \ \forall Y\leq 0; \quad u_e(r_0)+\hat{u}_p^{(0)}(\theta,Z)>0, \ \forall Z\geq 0, \\[5pt]
&&\|u_p^{(0)}\|_\infty+\|\hat{u}_p^{(0)}\|_\infty\leq M,
\eeno
where $M>0$ is a constant,  then there hold
\begin{align}
u_{e}^2(1)=\alpha^2+\frac{\eta^2}{2\pi}\int_0^{2\pi}f^2(\theta)d\theta,\label{BW formula}
\end{align}
and
\begin{align}
u_{e}^2(r_0)=\beta^2+\frac{\eta^2}{2\pi}\int_0^{2\pi}g^2(\theta)d\theta.\label{BW formula-1}
\end{align}
\end{lemma}
\begin{proof} We only need to prove $(\ref{BW formula})$ since $(\ref{BW formula-1})$ can be derived similarly.
We introduce the von Mises variable
\begin{align*}
\psi=\int_0^Y\big(u_e(1)+u_p^{(0)}(\theta,z)\big)dz,\  \mathcal{U}^{(0)}(\theta,\psi)=u_e(1)+u_p^{(0)}(\theta,Y).
\end{align*}
Then from (\ref{prandtl problem 1}), we deduce that $\mathcal{U}^{(0)}$ satisfies
\begin{eqnarray}
\left \{
\begin {array}{ll}
2\mathcal{U}^{(0)}_\theta=\big((\mathcal{U}^{(0)})^2\big)_{\psi\psi},\\[5pt]
\mathcal{U}^{(0)}(\theta,\psi)=\mathcal{U}^{(0)}(\theta+2\pi,\psi),\\[5pt]
\mathcal{U}^{(0)}\big|_{\psi=0}=\alpha+\eta f(\theta),\ \ \lim_{\psi\rightarrow-\infty}\mathcal{U}^{(0)}=u_{e}(1).\label{modified prandtl equation}
\end{array}
\right.
\end{eqnarray}
Here we have used the facts:
\begin{align*}
\partial_\theta u_p^{(0)}&=\mathcal{U}^{(0)}_\theta+\mathcal{U}^{(0)}_\psi\int_0^Y\partial_\theta u_p^{(0)}(\theta,z)\big)dz
\\[5pt]&=\mathcal{U}^{(0)}_\theta+\mathcal{U}^{(0)}_\psi\big(v_p^{(1)}(\theta,0)-v_p^{(1)}(\theta,Y)\big)
\\[5pt]&=\mathcal{U}^{(0)}_\theta-\mathcal{U}^{(0)}_\psi\big(v^{(1)}(\theta,1)+v_p^{(1)}(\theta,Y)\big),
\\[5pt]\partial_Yu_p^{(0)}&=\mathcal{U}^{(0)}_\psi\mathcal{U}^{(0)}.
\end{align*}

Integrating the first equation in (\ref{modified prandtl equation}) from $0$ to $2\pi$ about $\theta$ leads to
\begin{align*}
\frac{\partial^2}{\partial\psi^2}\int_0^{2\pi}(\mathcal{U}^{(0)})^2(\theta,\psi)d\theta=0.
\end{align*}
Notice that $\mathcal{U}^{(0)}$ is bounded at $\psi\rightarrow-\infty$, we deduce that
\begin{align*}
\frac{\partial}{\partial\psi}\int_0^{2\pi}(\mathcal{U}^{(0)})^2(\theta,\psi)d\theta=0.
\end{align*}
Therefore combining the boundary condition in (\ref{modified prandtl equation}), we deduce that
\begin{align*}
u_{e}^2(1)=\frac{1}{2\pi}\int_0^{2\pi}\big(\alpha+\eta f(\theta)\big)^2d\theta&=\alpha^2+\frac{\alpha\eta}{\pi}\int_0^{2\pi}f(\theta)d\theta+\frac{\eta^2}{2\pi}\int_0^{2\pi}f^2(\theta)d\theta
\\&=\alpha^2+\frac{\eta^2}{2\pi}\int_0^{2\pi}f^2(\theta)d\theta.
\end{align*}
Thus, we complete the proof of this lemma.
\end{proof}


Next we aim to solve the steady Prandtl equations (\ref{prandtl problem 1}), one can refer to Corollary 2.2 in \cite{FGLT}.
\begin{proposition}\label{decay estimates} There exists $\eta_0>0$ such that for any $\eta\in (0,\eta_0)$ and any $j,k,l\in \mathbb{N}\cup \{0\}$, the equations (\ref{modified prandtl equation}) have a unique solution $\mathcal{U}^{(0)}$ which satisfies
\begin{align*}
\int_{-\infty}^0\int_0^{2\pi}\Big|\partial_\theta^j\partial_\psi^k \big(\mathcal{U}^{(0)}-u_e(1)\big)\Big|^2\big<\psi\big>^{2l}d\theta d\psi\leq C(j,k,l)\eta^2,
\end{align*}
here $\big<\psi\big>=\sqrt{1+\psi^2}$.

\end{proposition}
\begin{proof}

We will use the contraction mapping theorem  to  prove the desired conclusions and divide the proof into five steps.

{\bf Step 1:} {\it Derivation of equivalent equations.}

Let $Q(\theta,\psi):=(\mathcal{U}^{(0)})^2(\theta,\psi)-u_{e}^2(1)$ and we rewrite (\ref{modified prandtl equation}) as
\begin{eqnarray}
\left \{
\begin {array}{ll}
Q_\theta=\mathcal{U}^{(0)}Q_{\psi\psi},\\[7pt]
Q(\theta,\psi)=Q(\theta+2\pi,\psi),\\[5pt]
Q\big|_{\psi=0}=\alpha^2+2\alpha\eta f(\theta)+\eta^2 f^2(\theta)-u_{e}^2(1),\  Q\big|_{\psi\rightarrow-\infty}=0.\label{modified prandtl equation-1}
\end{array}
\right.
\end{eqnarray}

Defining
\begin{align*}
\mathcal{G}(Q)=Q-2u_{e}(1)\sqrt{Q+u_e^2(1)}+2u_{e}^2(1),
\end{align*}
then (\ref{modified prandtl equation-1}) is equivalent to
\begin{eqnarray}
\left \{
\begin {array}{ll}
Q_\theta-u_{e}(1)Q_{\psi\psi}=(\mathcal{G}(Q))_\theta,\\[7pt]
Q(\theta,\psi)=Q(\theta+2\pi,\psi),\\[5pt]
Q\big|_{\psi=0}=\alpha^2+2\alpha\eta f(\theta)+\eta^2 f^2(\theta)-u_{e}^2(1),\  Q\big|_{\psi\rightarrow-\infty}=0.\label{modified prandtl equation-2}
\end{array}
\right.
\end{eqnarray}

Let $Q_0$ be the solution to
\begin{eqnarray}
\left \{
\begin {array}{ll}
(Q_0)_\theta=u_{e}(1)(Q_0)_{\psi\psi},\\[7pt]
Q_0(\theta,\psi)=Q_0(\theta+2\pi,\psi),\\[5pt]
Q_0\big|_{\psi=0}=\alpha^2+2\alpha\eta f(\theta)+\eta^2 f^2(\theta)-u_{e}^2(1),\  Q_0\big|_{\psi\rightarrow-\infty}=0,\label{Q0}
\end{array}
\right.
\end{eqnarray}
which will be solved in Appendix A,  then (\ref{modified prandtl equation-2}) is equivalent to
\begin{eqnarray}
\left \{
\begin {array}{ll}
Q_\theta-u_{e}(1)Q_{\psi\psi}=(\mathcal{H}(Q))_\theta,\\[7pt]
Q(\theta,\psi)=Q(\theta+2\pi,\psi),\\[5pt]
Q\big|_{\psi=0}=0,\  Q\big|_{\psi\rightarrow-\infty}=0,\label{modified prandtl equation-3}
\end{array}
\right.
\end{eqnarray}
where
\begin{align*}
\mathcal{H}(Q)=Q+Q_0-2u_{e}(1)\sqrt{Q+Q_0+u_e^2(1)}+2u_{e}^2(1).
\end{align*}

Defining the linear operator $\mathcal{L}:\Lambda\longmapsto \Phi$ such that
\begin{eqnarray}
\Phi=\mathcal{L} \Lambda\Longleftrightarrow\left \{
\begin {array}{ll}
\Phi_\theta-u_{e}(1)\Phi_{\psi\psi}=\Lambda_\theta,\\[7pt]
\Phi(\theta,\psi)=\Phi(\theta+2\pi,\psi),\\[5pt]
\Phi\big|_{\psi=0}=0,\  \Phi\big|_{\psi\rightarrow-\infty}=0,\label{definition L}
\end{array}
\right.
\end{eqnarray}
then (\ref{modified prandtl equation-3}) is equivalent to
\begin{eqnarray}
\left \{
\begin {array}{ll}
Q=\mathcal{L} (\mathcal{H}(Q)),\\[7pt]
Q(\theta,\psi)=Q(\theta+2\pi,\psi),\\[5pt]
Q\big|_{\psi=0}=0,\  Q\big|_{\psi=-\infty}=0.\label{operator form}
\end{array}
\right.
\end{eqnarray}

Defining the function space $X$ as follows
\begin{align*}
X=\bigg\{&Q:Q(\theta,\psi)=Q(\theta+2\pi,\psi),\ Q\big|_{\psi=0}=Q\big|_{\psi\rightarrow-\infty}=0,\\&
\quad \quad  \|Q\|_X^2:=\sum\limits_{j+k\leq m, l\geq 0}\int_{-\infty}^0\int_0^{2\pi}\big|\partial_\theta^j\partial_\psi^k Q\big|^2\big<\psi\big>^ld\theta d\psi<+\infty\bigg\}
\end{align*}
and a ball $B_0$ in $X$
\begin{align*}
B_0=\big\{Q\in X: \|Q\|_X\leq \nu\big\},
\end{align*}
here $m$ is a positive integer and  $\nu$ is a small number which will be determined later.

 In the next three steps  we aim to  verify that $\mathcal{L}\circ\mathcal{H}$ is a contraction map from $B_0$ to $B_0$ with suitable small $\nu$.

{\bf Step 2:}{ \it  Boundedness of $\mathcal{L}$ in $X$}. In this step, we prove that for any $m\geq 0, l\geq 0$, there holds
\begin{align}\label{estimate for $L$}
\sum_{j+k\leq m}\big\|\partial^j_\theta\partial^k_\psi\Phi\big<\psi\big>^l\big\|_2\leq C(m,l)\sum_{j+k\leq m, q\leq l}\big\|\partial^j_\theta\partial^k_\psi\Lambda\big<\psi\big>^q\big\|_2.
\end{align}

First, we prove (\ref{estimate for $L$}) for $l=0$. Multiplying   the equation in (\ref{definition L}) by $\Phi$  and integrating with respect to $(\theta,\psi)\in(0,2\pi)\times(-\infty,0)$, we obtain for any $\lambda>0$
\begin{align}
u_{e}(1)\|\Phi_\psi\|_2\leq C(\lambda)\|\Lambda\|_2+\lambda\|\Phi_\theta\|_2.\label{iorder-derivative-1}
\end{align}
Then, multiplying   the equation in (\ref{definition L}) by $\Phi_\theta$  and integrating with respect to $(\theta,\psi)\in(0,2\pi)\times(-\infty,0)$, one has
\begin{align}
\|\Phi_\theta\|_2\leq C(\lambda)\|\Lambda_\theta\|_2+\lambda\|\Phi_\theta\|_2.\label{iorder-derivative-2}
\end{align}
Combining (\ref{iorder-derivative-1})-(\ref{iorder-derivative-2}) and choosing small $\lambda>0$ we get
\begin{align}
\|(\Phi_\psi,\Phi_\theta)\|_2\leq C\|(\Lambda,\Lambda_\theta)\|_2.\label{iorder-derivative}
\end{align}

Integrating (\ref{definition L}) with respect to $\theta\in(0,2\pi)$ gives
\begin{align*}
\frac{d^2}{d\psi^2}\int_0^{2\pi}\Phi (\theta,\psi)d\theta=0,
\end{align*}
which and $\Phi|_{\psi=0}=\Phi|_{\psi\rightarrow-\infty}=0$ imply
\begin{align}
\int_0^{2\pi}\Phi (\theta,\psi)d\theta=0 .\label{zero average}
\end{align}
Due to  (\ref{zero average}) and the Poincar\'{e} inequality we have
\begin{align}
\|\Phi\|_2\leq C\|\Phi_\theta\|_2\leq C\|(\Lambda,\Lambda_\theta)\|_2.\label{L2}
\end{align}

For any $j\geq 0, k\geq 2$, from the equation (\ref{definition L}), we deduce that
\begin{align}
\|\partial_\theta^j\partial_\psi^k\Phi\|_2\leq C (\|\partial_\theta^j\partial_\psi^{k-2}\Phi_\theta\|_2+\|\partial_\theta^j\partial_\psi^{k-2}\Lambda_\theta\|_2).\label{higher-order-derivative}
\end{align}
For any $j\geq 0$, applying $\partial^j_\theta$ to the equation in (\ref{definition L}), multiplying the resultant equation by   $\partial^j_\theta\Phi$, integrating with respect to $(\theta,\psi)\in(0,2\pi)\times(-\infty,0)$ and using the Young inequality  one has
\begin{align}
\|\partial^j_\theta\partial_\psi\Phi\|_2\leq C\|\partial^{j+1}_\theta\Lambda\|_2+C\|\partial^j_\theta\Phi\|_2.\label{2order-derivative-3}
\end{align}
For any $j\geq 0$, multiplying the equation in (\ref{definition L}) by   $\partial^{2j-1}_\theta\Phi$, integrating with respect to $(\theta,\psi)\in(0,2\pi)\times(-\infty,0)$, one can get
\begin{align}
\|\partial^j_\theta\Phi\|^2_2=\int_{-\infty}^0\int_0^{2\pi}\partial^{2j-1}_\theta\Phi\Lambda_\theta d\theta d\psi\leq C\|\partial^j_\theta\Phi\|_2\|\partial^j_\theta\Lambda\|_2.\label{2order-derivative-2}
\end{align}

Based on (\ref{higher-order-derivative}), (\ref{2order-derivative-3}) and (\ref{2order-derivative-2}), we obtain that for any $m\geq 2$, there holds
\begin{align}
\sum_{j+k=m}\|\partial^j_\theta\partial^k_\psi\Phi\|_2\leq C\sum_{j+k\leq m-1}\|\partial^j_\theta\partial^k_\psi\Phi\|_2+C\sum_{j+k\leq m}\|\partial^j_\theta\partial^k_\psi\Lambda\|_2.\nonumber
\end{align}
Thus, by induction and (\ref{iorder-derivative}), (\ref{L2}), we deduce that for any $m\geq 0$, there holds
\begin{align}\label{estimate without weight}
\sum_{j+k\leq m}\|\partial^j_\theta\partial^k_\psi\Phi\|_2\leq C(m)\sum_{j+k\leq m}\|\partial^j_\theta\partial^k_\psi\Lambda\|_2.
\end{align}

Next, we prove (\ref{estimate for $L$}) for $m\leq1, l\geq 1$.
Multiplying the equation in (\ref{definition L}) by $\Phi\psi^{2l}$, integrating with respect to $(\theta,\psi)\in(0,2\pi)\times(-\infty,0)$, we have
\begin{align}
\int_{-\infty}^0\int_0^{2\pi}\Phi_\psi^2\psi^{2l}d\theta d\psi&\leq C\int_{-\infty}^0\int_0^{2\pi}\Phi^2\psi^{2(l-1)}d\theta d\psi
\nonumber\\&\quad+ C(\lambda)\int_{-\infty}^0\int_0^{2\pi}\Lambda^2\psi^{2l}d\theta d\psi+\lambda\int_{-\infty}^0\int_0^{2\pi}\Phi_\theta^2\psi^{2l}d\theta d\psi
.\label{iorder-high-weight-derivative-1}
\end{align}
Multiplying   the equation in (\ref{definition L}) by $\Phi_\theta\psi^{2l}$, integrating with respect to $(\theta,\psi)\in(0,2\pi)\times(-\infty,0)$, we obtain
\begin{align}
\int_{-\infty}^0\int_0^{2\pi}\Phi_\theta^2\psi^{2l}d\theta d\psi&\leq C\int_{-\infty}^0\int_0^{2\pi}\Phi_\psi^{2}\psi^{2(l-1)}d\theta d\psi
+C\int_{-\infty}^0\int_0^{2\pi}\Lambda_\theta^2\psi^{2l}d\theta d\psi.\label{iorder-high-weight-derivative-2}
\end{align}
Combining (\ref{iorder-high-weight-derivative-1})-(\ref{iorder-high-weight-derivative-2}), using the Poincar\'{e} inequality and choosing small $\lambda>0$ we get
\begin{align*}
&\int_{-\infty}^0\int_0^{2\pi}\big(\Phi^2+\Phi_\psi^2+\Phi_\theta^2\big)\psi^{2l}d\theta d\psi\nonumber\\
\leq&\int_{-\infty}^0\int_0^{2\pi}\big(\Phi^2+\Phi_\psi^2+\Phi_\theta^2\big)\psi^{2(l-1)}d\theta d\psi+ C\int_{-\infty}^0\int_0^{2\pi}\big(\Lambda^2\psi^{2l}+\Lambda_\theta^2\psi^{2l}\big)d\theta d\psi.
\end{align*}
By induction on $l$ and using (\ref{estimate without weight}), we deduce that for any $l\geq 1$, there holds
\begin{align}\label{estimate on lower order derivative with weight}
\big\|\big(\Phi_\psi,\Phi_\theta,\Phi\big)\psi^{l}\big\|_2\leq C(l)\sum_{q\leq l}\big\|\big(\Lambda,\Lambda_\theta\big)\psi^q\big\|_2.
\end{align}

Finally, we prove (\ref{estimate for $L$}) any $m\geq 2, l\geq 1$. When $j+k=m, k\geq 2$, from the equation (\ref{definition L}), we deduce that
\begin{align}
\Big\|\partial_\theta^j\partial_\psi^k\Phi \psi^l\Big\|_2\leq C \Big(\Big\|\partial_\theta^j\partial_\psi^{k-2}\partial_\theta\Phi\psi^l\Big\|_2
+\Big\|\partial_\theta^j\partial_\psi^{k-2}\partial_\theta\Lambda\psi^l\Big\|_2\Big).\label{higher-order-derivative-weight}
\end{align}
Applying $\partial^{m-1}_\theta$ to the equation in (\ref{definition L}), multiplying the resultant equation by   $\partial^{m-1}_\theta\Phi \psi^{2l}$, integrating with respect to $(\theta,\psi)\in(0,2\pi)\times(-\infty,0)$ and using the Young inequality  one has
\begin{align}
\Big\|\partial_\theta^{m-1}\partial_\psi\Phi \psi^l\Big\|_2\leq C\Big\|\partial_\theta^m\Lambda \psi^l\Big\|_2+C\Big\|\partial_\theta^{m-1}\Phi \psi^l\Big\|_2+C(l)\Big\|\partial_\theta^{m-1}\Phi \psi^{l-1}\Big\|_2.\label{2order-derivative-3-weight}
\end{align}
Multiplying the equation in (\ref{definition L}) by   $\partial^{2m-1}_\theta\Phi \psi^{2l}$, integrating with respect to $(\theta,\psi)\in(0,2\pi)\times(-\infty,0)$, one can get
\begin{align}
\Big\|\partial^m_\theta\Phi \psi^l\Big\|_2\leq C\Big\|\partial^m_\theta\Lambda  \psi^l \Big\|_2+C(l)\Big\|\partial_\theta^{m-1}\partial_\psi\Phi \psi^{l-1}\Big\|_2.\label{2order-derivative-2-weight}
\end{align}

Thanks to (\ref{higher-order-derivative-weight}), (\ref{2order-derivative-3-weight}) and (\ref{2order-derivative-2-weight}), we obtain that for any $m\geq 2, l\geq 1$, there holds
\begin{align}
\sum_{j+k=m}\Big\|\partial^j_\theta\partial^k_\psi\Phi \psi^l\Big\|_2\leq C(m,l)\sum_{j+k\leq m-1, q\leq l}\Big\|\partial^j_\theta\partial^k_\psi\Phi \psi^q\Big\|_2+C(m,l)\sum_{j+k\leq m}\Big\|\partial^j_\theta\partial^k_\psi\Lambda \psi^l\Big\|_2.\nonumber
\end{align}
Thus, by induction $m$ and using (\ref{estimate on lower order derivative with weight}), we deduce that for any $m\geq 2, l\geq 1$, there holds
\begin{align}\label{estimate with weight}
\sum_{j+k\leq m}\Big\|\partial^j_\theta\partial^k_\psi\Phi \psi^l\Big\|_2\leq C(m,l)\sum_{j+k\leq m,q\leq l}\Big\|\partial^j_\theta\partial^k_\psi\Lambda \psi^q\Big\|_2.
\end{align}
This complete the proof of (\ref{estimate for $L$}).
Thus, we obtain
\begin{align}
\|\mathcal{L}\Lambda\|_{X}\leq C \|\Lambda\|_{X},\ \forall \Lambda\in X.\label{boundedness}
\end{align}

{\bf Step 3:}{ \it $\mathcal{L}\circ\mathcal{H}$ is a continuous map from $B_0$ to $B_0$.} In this section, we first prove that for any $m\geq 2, l\geq 0$, there holds
\begin{align}\label{nonlinear estimate on H}
\sum_{j+k\leq m}\Big\|\partial_\theta^j\partial_\psi^k \mathcal{H}(Q)\big<\psi\big>^l\Big\|_2\leq C(m,l)\Big(\sum_{j+k\leq m}\Big\|\partial_\theta^j\partial_\psi^k Q\big<\psi\big>^l\Big\|_2+\sum_{j+k\leq m}\Big\|\partial_\theta^j\partial_\psi^k Q_0\big<\psi\big>^l\Big\|_2\Big)^2.
\end{align}

Set
\beno
\tilde{H}(x)=x-2u_e(1)\sqrt{x+u_e^2(1)}+2u^2_e(1),\ |x|\ll u_e^2(1),
\eeno
then $\mathcal{H}(Q)=\tilde{H}(Q+Q_0).$ Direct computation gives
\begin{align}\label{estimate on H}
|\tilde{H}'(x)|\leq C|x|, \ |\tilde{H}^{(k)}(x)|\leq C,\ k\geq 2.
\end{align}
Since
\begin{align*}
\mathcal{H}(Q)=\big(\sqrt{Q+Q_0+u_e^2(1)}-u_e(1)\big)^2=\bigg(\frac{Q+Q_0}{\sqrt{Q+Q_0+u_e^2(1)}+u_e(1)}\bigg)^2,
\end{align*}
it's easy to get
\begin{align*}
\mathcal{H}^2(Q)\big<\psi\big>^l\leq C\|Q+Q_0\|_{L^\infty}^2\big(Q^2\big<\psi\big>^l
+Q_0^2\big<\psi\big>^l\big), \ l\geq 0.
\end{align*}
Using (\ref{estimate on H}), we deduce that
\begin{align*}
\big(\partial_\psi\mathcal{H}(Q)\big)^2\big<\psi\big>^l\leq& C\|Q+Q_0\|_{L^\infty}^2
\big(Q_\psi^2\big<\psi\big>^l+(Q_0)_\psi^2\big<\psi\big>^l\big),\ l\geq 0,\\
\big(\partial_\theta\mathcal{H}(Q)\big)^2\big<\psi\big>^l\leq& C\|Q+Q_0\|_{L^\infty}^2
\big(Q_\theta^2\big<\psi\big>^l+(Q_0)_\theta^2\big<\psi\big>^l\big),\ l\geq 0.
\end{align*}
Thus, we obtain
\begin{align}\label{lower order estimate on H}
&\sum_{j+k\leq 1}\Big\|\partial_\theta^j\partial_\psi^k \mathcal{H}(Q)\big<\psi\big>^l\Big\|_2\nonumber\\
\leq& C\|Q+Q_0\|_{L^\infty}\Big(\sum_{j+k\leq 1}\Big\|\partial_\theta^j\partial_\psi^k Q\big<\psi\big>^l\Big\|_2+\sum_{j+k\leq 1}\Big\|\partial_\theta^j\partial_\psi^k Q_0\big<\psi\big>^l\Big\|_2\Big).
\end{align}

Since
\begin{align*}
\partial_{\psi\theta}\mathcal{H}(Q)=\tilde{H}'(Q+Q_0)(\partial_{\psi\theta}Q+\partial_{\psi\theta}Q_0)
+\tilde{H}''(Q+Q_0)(\partial_{\theta}Q+\partial_{\theta}Q_0)(\partial_{\psi}Q+\partial_{\psi}Q_0),
\end{align*}
thus, using (\ref{estimate on H}), we obtain
\begin{align*}
\Big\|\partial_{\psi\theta}\mathcal{H}(Q)\big<\psi\big>^l\Big\|_{L^2}&\leq C\|Q+Q_0\|_{L^\infty}\Big(\Big\|Q_{\psi\theta}\big<\psi\big>^l\Big\|_{L^2}+\Big\|(Q_0)_{\psi\theta}\big<\psi\big>^l\Big\|_{L^2}\Big)
\nonumber\\&\quad+C\big\|Q_{\psi}
+(Q_0)_{\psi}\big\|_{L^4}\big\|(Q_{\theta}+(Q_0)_{\theta})\big<\psi\big>^l\big\|_{L^4}
\nonumber\\&\leq C\|Q+Q_0\|_{L^\infty}\Big(\Big\|Q_{\psi\theta}\big<\psi\big>^l\Big\|_{L^2}+\Big\|(Q_0)_{\psi\theta}\big<\psi\big>^l\Big\|_{L^2}\Big)
\nonumber\\&\quad+C\big\|Q_{\psi}
+(Q_0)_{\psi}\big\|_{H^1}\big\|(Q_{\theta}+(Q_0)_{\theta})\big<\psi\big>^l\big\|_{H^1}.
\end{align*}
Hence by Sobolev imbedding, we deduce that
\begin{align*}
\Big\|\partial_{\psi\theta}\mathcal{H}(Q)\big<\psi\big>^l\Big\|_2\leq C\Big(\sum_{j+k\leq 2}\Big\|\partial_\theta^j\partial_\psi^k Q\big<\psi\big>^l\Big\|_2+\sum_{j+k\leq 2}\Big\|\partial_\theta^j\partial_\psi^k Q_0\big<\psi\big>^l\Big\|_2\Big)^2.
\end{align*}

Same estimates hold for $\Big\|\partial_{\theta\theta}\mathcal{H}(Q)\big<\psi\big>^l\Big\|_2$ and $\Big\|\partial_{\psi\psi}\mathcal{H}(Q)\big<\psi\big>^l\Big\|_2$. Thus, combing the estimate (\ref{lower order estimate on H}), we arrive at
\begin{align*}
\sum_{j+k\leq 2}\Big\|\partial_\theta^j\partial_\psi^k \mathcal{H}(Q)\big<\psi\big>^l\Big\|_2
\leq C\Big(\sum_{j+k\leq 2}\Big\|\partial_\theta^j\partial_\psi^k Q\big<\psi\big>^l\Big\|_2+\sum_{j+k\leq 2}\Big\|\partial_\theta^j\partial_\psi^k Q_0\big<\psi\big>^l\Big\|_2\Big)^2.
\end{align*}
Using (\ref{estimate on H}) and repeating the above arguments, we obtain (\ref{nonlinear estimate on H}).

Consequently,  if we take $\nu=\|Q_0\|_{X}$ and $\eta$ small enough, then
\begin{align*}
\|\mathcal{L}\mathcal{H}(Q)\|_{X}\leq C\|\mathcal{H}(Q)\|_{X}\leq 4C \nu^2\leq 4C\eta \nu\leq \nu,\  \forall Q\in B_0,
\end{align*}
here we have used (\ref{norm of q0}). Thus, $\mathcal{L}\mathcal{H}$ is a continuous map from $B_0$ to $B_0$ for small $\eta$.

{\bf Step 4:}{ \it $\mathcal{L}\circ\mathcal{H}$ is a contraction map in $B_0$.}

Noting firstly that
\begin{align*}
\mathcal{H}(Q_1)-\mathcal{H}(Q_2)&=(Q_1-Q_2)\bigg(1-\frac{2u_e(1)}{\sqrt{Q_1+Q_0+u_e^2(1)}+\sqrt{Q_2+Q_0+u_e^2(1)}}\bigg)
\end{align*}
and
\begin{align*}
&1-\frac{2u_e(1)}{\sqrt{Q_1+Q_0+u_e^2(1)}+\sqrt{Q_2+Q_0+u_e^2(1)}}
\nonumber\\&=\frac{Q_1+Q_0}{\big(\sqrt{Q_1+Q_0+u_e^2(1)}+\sqrt{Q_2+Q_0+u_e^2(1)}\big)\big(\sqrt{Q_1+Q_0+u_e^2(1)}+u_e(1)\big)}
\nonumber\\&\quad+\frac{Q_2+Q_0}{\big(\sqrt{Q_1+Q_0+u_e^2(1)}+\sqrt{Q_2+Q_0+u_e^2(1)}\big)\big(\sqrt{Q_2+Q_0+u_e^2(1)}+u_e(1)\big)}.
\end{align*}

With the help of the similar arguments in Step 3, we can obtain that there exist $\eta_0>0$ such that for any $\eta\in (0,\eta_0)$, there holds
\begin{align*}
\big\|\mathcal{L}\mathcal{H}(Q_1)-\mathcal{L}\mathcal{H}(Q_2)\big\|_X
&\leq C\big\|\mathcal{H}(Q_1)-\mathcal{H}(Q_2)\big\|_X
\\[5pt]&\leq C\big\|Q_1-Q_2\big\|_X\big(\|Q_0\|_X+\|Q_1\|_X+\|Q_2\|_X\big)
\\[5pt]&\leq C\|Q_0\|_X\big\|Q_1-Q_2\big\|_X
\\[5pt]&\leq\frac{1}{2}\big\|Q_1-Q_2\big\|_X,
\end{align*}
that is, $\mathcal{L}\mathcal{H}$ is a contraction map in $B_0$.

{\bf Step 5:} { \it Existence and uniqueness of the system (\ref{modified prandtl equation}).} By the standard contraction mapping principle, we know that there exists $\eta_0>0$ such that for any $\eta\in (0,\eta_0)$ and any $j,k,l\in \mathbb{N}\cup \{0\}$, the equation (\ref{modified prandtl equation}) has a unique solution $\mathcal{U}^{(0)}$ which satisfies
\begin{align*}
\int_{-\infty}^0\int_0^{2\pi}\Big|\partial_\theta^j\partial_\psi^k \big(\mathcal{U}^{(0)}-u_e(1)\big)\Big|^2\big<\psi\big>^{2l}d\theta d\psi\leq C(j,k,l)\eta^2,
\end{align*}
this completes the proof of this proposition.
\end{proof}
Return to the equations (\ref{prandtl problem 1}), we have the following result.
\begin{proposition} There exists $\eta_0>0$ such that for any $\eta\in(0,\eta_0)$, the equations (\ref{prandtl problem 1}) have a unique solution $(u_p^{(0)},v_p^{(1)})$ which satisfies
\begin{align}\label{decay behavior-prandtl}
\begin{aligned}
&\sum_{j+k\leq m}\int_{-\infty}^0\int_0^{2\pi}\Big|\partial_\theta^j\partial_Y^k (u_p^{(0)},v_p^{(1)})\Big|^2\big<Y\big>^{2l}d\theta dY\leq C(m,l)\eta^2, \ \forall m,l \geq 0, \\
&\int_0^{2\pi}v_p^{(1)}(\theta,Y)d\theta=0,\ \forall Y\leq 0.
\end{aligned}
\end{align}
\end{proposition}
Finally, solving (\ref{equation of first pressure}), we obtain $p_p^{(1)}(\theta,Y)$  which decay very fast as $Y\rightarrow -\infty$.

By the same argument, we can obtain the well-posedness of equations (\ref{prandtl problem another boundary}).
\begin{proposition}There exists $\eta_0>0$ such that for any $\eta\in(0,\eta_0)$, the equations (\ref{prandtl problem another boundary}) have a unique solution $(\hat{u}_p^{(0)},\hat{v}_p^{(1)})$ which satisfies
\begin{align}\label{decay behavior-prandtl-another}
\begin{aligned}
&\sum_{j+k\leq m}\int_{-\infty}^0\int_0^{2\pi}\Big|\partial_\theta^j\partial_Y^k (\hat{u}_p^{(0)},\hat{v}_p^{(1)})\Big|^2\big<Y\big>^{2l}d\theta dY\leq C(m,l)\eta^2, \ m,l \geq 0, \\
&\int_0^{2\pi}\hat{v}_p^{(1)}(\theta,Z)d\theta=0,\ \forall Z\geq 0.
\end{aligned}
\end{align}
\end{proposition}
Similarly, we can obtain $\hat{p}_p^{(1)}$ by solving \eqref{equation of first pressure another} and $\hat{p}_p^{(1)}$ decay very fast as $Z\rightarrow +\infty$.\\

Here we give a general strategy for constructing high order approximate solution. We first construct $(u_e^{(1)}, v_e^{(1)}, p_e^{(1)})$ by solving the linearized Euler equations (\ref{outer-1 order equation}) and $(u_p^{(1)}, v_p^{(2)})$ by solving the linearized Prandtl equations (\ref{first linearized prandtl problem near 1}). Then, because $A_{1\infty}:=\lim_{Y\rightarrow-\infty}u_p^{(1)}$ is a nonzero constant, we need to change it to $\tilde{u}_p^{(1)}=u_p^{(1)}-A_{1\infty}$. Finally, we modify $u_e^{(1)}$ into $\tilde{u}_e^{(1)}$ by adding a radial function, see (\ref{modify Euler}) for the details. Notice the structure of the linearized Euler equations (\ref{outer-1 order equation}) and the fact that only the value of $u_e^{(1)}$ at $r=1$ appear in the equation (\ref{first linearized prandtl problem near 1}), we can easily deduce that the modified $(\tilde{u}_e^{(1)}, \tilde{v}_e^{(1)}, \tilde{p}_e^{(1)})$ and $(\tilde{u}_p^{(1)}, \tilde{v}_p^{(2)})$ still satisfies the equation (\ref{outer-1 order equation}) and (\ref{first linearized prandtl problem near 1}). The higher order approximate solutions $(\tilde{u}_e^{(i)}, \tilde{v}_e^{(i)}, \tilde{p}_e^{(i)})(i\geq 2)$ and $(\tilde{u}_p^{(i)}, \tilde{v}_p^{(i+1)})(i\geq2)$ can be constructed by the same approach.

\subsubsection{Linearized Euler equations for $(u_e^{(1)}, v_e^{(1)}, p_e^{(1)})$ and their solvability}
\begin{proposition}
There exists $\eta_0>0$ such that for any $\eta\in(0,\eta_0)$, the linearized Euler equations (\ref{outer-1 order equation}) have a solution $(u_e^{(1)}, v_e^{(1)}, p_e^{(1)})$ which satisfies
\begin{align}\label{Estimate of first linearized Euler equation}
\|\partial^k_\theta\partial^j_r(u_e^{(1)},v_e^{(1)})\|_2\leq C(k,j)\eta, \ \forall j,k\geq 0.
\end{align}
\end{proposition}
\begin{proof}
 Eliminating pressure $p_e^{(1)}$ in the equation (\ref{outer-1 order equation}), we obtain
\begin{align*}
-u_e(r)\Big(\partial_{rr}v_e^{(1)}+\frac{\partial_{\theta\theta} v_e^{(1)}}{r^2}+\frac{3\partial_rv_e^{(1)}}{r}\Big)+v_e^{(1)}\Big(u''_e(r)+\frac{u'_e(r)}{r}-\frac{2u_e(r)}{r^2}\Big)=0.
\end{align*}
Notice that $\triangle =\partial_{rr}+\frac{\partial_{\theta\theta}}{r^2}+\frac{\partial_r}{r}$, we deduce that
\begin{align*}
-u_e(r)\triangle(rv_e^{(1)})+\Big(u''_e(r)+\frac{u'_e(r)}{r}-\frac{u_e(r)}{r^2}\Big)(rv_e^{(1)})=0,
\end{align*}
hence we obtain the following equations for $rv_e^{(1)}$ in $\Omega$
\begin{eqnarray}\label{Euler equation normal}
\left \{
\begin {array}{ll}
-\triangle(rv_e^{(1)})+\frac{1}{u_e(r)}\Big(u''_e(r)+\frac{u'_e(r)}{r}-\frac{u_e(r)}{r^2}\Big)(rv_e^{(1)})=0,\\[5pt]
rv_e^{(1)}|_{r=1}=-v_p^{(1)}|_{Y=0},\quad rv_e^{(1)}|_{r=r_0}=-r_0\hat{v}_p^{(1)}|_{Z=0}.
\end{array}
\right.
\end{eqnarray}

Firstly, we deduce that equations (\ref{Euler equation normal}) a unique solution $rv_e^{(1)}$.
In fact, we assume that $f(\theta,r)$ solving the equations
\begin{eqnarray}
\left \{
\begin {array}{ll}
-\triangle f+\frac{1}{u_e(r)}\Big(u''_e(r)+\frac{u'_e(r)}{r}-\frac{u_e(r)}{r^2}\Big)f=0,\\[5pt]
f|_{r=1}=0,\ f|_{r=r_0}=0,\nonumber
\end{array}
\right.
\end{eqnarray}
then multiple $rf$ and integrating in $\Omega$, we obtain
\begin{align*}
\int_0^{2\pi}\int_{r_0}^1\Big(r(\partial_rf)^2+(\partial_\theta f)^2+\frac{1}{u_e(r)}\Big(ru''_e(r)+u'_e(r)-\frac{u_e(r)}{r}\Big)f^2\Big)d\theta dr=0.
\end{align*}
Let $g=\frac{f}{u_e}$, we arrive at
\begin{align*}
\int_0^{2\pi}\int_{r_0}^1\Big(ru^2_e(r)(\partial_rg)^2+\frac{u_e^2(r)g^2_\theta}{r}-\frac{u_e^2(r)g^2}{r}\Big)d\theta dr=0.
\end{align*}
Notice that $\int_0^{2\pi}g(\theta,r)d\theta=0, \ \forall r\in [r_0,1]$, by Poincar\'{e} inequality, we deduce that $g=0$, hence $f=0$. Thus, equation (\ref{Euler equation normal}) has a unique solution.

Then, notice (\ref{decay behavior-prandtl}) and (\ref{decay behavior-prandtl-another}), we deduce that
\begin{align*}
\|\partial^k_\theta\partial^j_rv_e^{(1)}\|_2\leq C(k,j)\eta, \ \forall k,j\geq 0.
\end{align*}
Moreover, due to $\int_0^{2\pi}v_p^{(1)}(\theta,0)d\theta=\int_0^{2\pi}\hat{v}_p^{(1)}(\theta,0)d\theta=0$ and the divergence-free condition, we deduce that
\begin{align}
\int_0^{2\pi}v_e^{(1)}(\theta,r)d\theta=0,\ \forall r\in [r_0,1].\label{ve1zeroav}
\end{align}

Then we can construct $u_e^{(1)}$ by solving the following equation
\begin{eqnarray}
\left \{
\begin {array}{ll}
\partial_\theta u_e^{(1)}+\partial_r(rv_e^{(1)})=0,\\[5pt]
u_e^{(1)}(\theta,r)=u_e^{(1)}(\theta+2\pi,r).\nonumber
\end{array}
\right.
\end{eqnarray}
After obtaining $(u_e^{(1)}, v_e^{(1)})$, we construct $p_e^{(1)}$ as following
\begin{align*}
p_e^{(1)}(\theta,r):=\phi(r)-\int_0^\theta[u_e(r)\partial_{\theta'} u_e^{(1)}+ru'_e v_e^{(1)}+u_ev_e^{(1)}](\theta',r)d\theta',
\end{align*}
where $\phi(r)$ is a function which satisfies
\begin{align*}
r\partial_r\phi(r)+u_e(r)\partial_\theta v_e^{(1)}(0,r)-2u_e(r)u_e^{(1)}(0,r)=0.
\end{align*}
Combining the equations of $(u_e^{(1)}, v_e^{(1)})$, it's direct to obtain
\begin{align*}
u_e \partial_\theta v_e^{(1)}-2u_eu_e^{(1)}+r\partial_rp_e^{(1)}=0.
\end{align*}
Hence, $(u_e^{(1)}, v_e^{(1)},p_e^{(1)})$ solves the equation (\ref{outer-1 order equation}) and satisfies (\ref{Estimate of first linearized Euler equation}).
\end{proof}

\subsubsection{Linearized Prandtl equations for  $(u_p^{(1)},v_p^{(2)})$, $(\hat{u}_p^{(1)},\hat{v}_p^{(2)})$ and their solvabilities}
\indent

In this subsubsection, we consider the solvabilities of \eqref{first linearized prandtl problem near 1} and \eqref{first linearized prandtl problem another boundary}. One can also refer to Proposition 2.5 in \cite{FGLT}.

\begin{proposition}\label{decay estimates of linearized Prandtl} There exists $\eta_0>0$ such that for any $\eta\in(0,\eta_0)$, the equations (\ref{first linearized prandtl problem near 1}) have a unique solution $(u_p^{(1)},v_p^{(2)})$ which satisfies
\begin{align} \label{decay behavior-prandtl-1}
\begin{aligned}
&\sum_{j+k\leq m}\int_{-\infty}^0\int_0^{2\pi}\big|\partial_\theta^j\partial_Y^k \big(u_p^{(1)}-A_{1\infty},v_p^{(2)}\big)\big|^2\big<Y\big>^{2l}d\theta dY\leq C(m,l)\eta^2, \ \forall m,l\geq 0,\\
&\int_0^{2\pi}v_p^{(2)}(\theta,Y)d\theta=0, \ \forall Y\leq 0,
\end{aligned}
\end{align}
where
$A_{1\infty}:=\lim\limits_{Y\rightarrow -\infty}u_p^{(1)}(\theta,Y)$ is a constant which satisfies $|A_{1\infty}|\leq C\eta.$
\end{proposition}
\begin{proof}
Let $\eta\in C_c^\infty ((-\infty,0])$ satisfy
\begin{align*}
\eta(0)=1,\ \int_0^{+\infty}\eta(y)dy=0.
\end{align*}
For simplicity, we set
\begin{align*}
\bar{u}:&=u_e(1)+u_p^{(0)}, \ \bar{v}:=v_e^{(1)}(\theta,1)+v_p^{(1)},\\
u:&=u_p^{(1)}+u_e^{(1)}(\theta,1)\eta(Y), \ v:=v_p^{(2)}-v_p^{(2)}(\theta,0)+Y v_p^{(1)}-\partial_\theta u_e^{(1)}(\theta,1)\int_0^Y\eta(z)dz.
\end{align*}
Then, the equations (\ref{first linearized prandtl problem near 1}) reduce to
\begin{eqnarray}\label{new linearized Prandtl equation}
\left \{
\begin {array}{ll}
\bar{u}\partial_\theta u+\bar{v}\partial_Yu+u\partial_\theta \bar{u}+v\partial_Y\bar{u}-\partial_{YY}u=\tilde{f},\\[7pt]
\partial_\theta u+\partial_Yv=0,\\[5pt]
u(\theta,Y)=u(\theta+2\pi,Y),\ v(\theta,Y)=v(\theta+2\pi,Y)\\[5pt]
u|_{Y=0}=v|_{Y=0}=0,\  \lim\limits_{Y\rightarrow -\infty}\partial_Yu=0,
\end{array}
\right.
\end{eqnarray}
where $\tilde{f}(\theta,Y)$ is $2\pi$-periodic function and decays fast as $Y\rightarrow -\infty$.

We can solve the equations (\ref{new linearized Prandtl equation}) by considering the following approximate system.
Let $\gamma>0$ be a constant, we consider the following elliptic equation
\begin{eqnarray}\label{appro linear prandtl}
\left \{
\begin {array}{ll}
\bar{u}\partial_\theta u^{\gamma}+\bar{v}\partial_Yu^{\gamma}+\big[\int_Y^0\p_{\theta}u^{\gamma}(\theta,z)dz \big]\partial_Y\bar{u}+u^{\gamma}\partial_\theta \bar{u}-\partial_{YY}u^{\gamma}-\gamma \p_{\theta\theta} u^{\gamma}=\tilde{f},\\[7pt]
u^{\gamma}(\theta,Y)=u^{\gamma}(\theta+2\pi,Y),\\[5pt]
u^{\gamma}|_{Y=0}=0.
\end{array}
\right.
\end{eqnarray}

We expect the solution of this equation is in $\dot{H}^1_0=\{u|\p_{\theta}u\in L^2,\p_Y u\in L^2, u|_{Y=0}=0\}$ rather than $H^1_0=\{u|u\in L^2,\p_{\theta}u\in L^2,\p_Y u\in L^2, u|_{Y=0}=0\}$. Now we establish apriori estimate of equation (\ref{appro linear prandtl}).
Multiplying the first equation in (\ref{appro linear prandtl}) by $u^{\gamma}$ and integrating in $(\theta,r)\in(0,2\pi)\times(-\infty,0)$, we obtain that
\begin{align*}
&\int_{-\infty}^{0}\int_0^{2\pi}\Big[\bar{u}\partial_\theta u^{\gamma}+\bar{v}\partial_Yu^{\gamma}+\Big(\int_Y^0\p_{\theta}u^{\gamma}(\theta,z)dz \Big)\partial_Y\bar{u}+u^{\gamma}\partial_\theta \bar{u}-\partial_{YY}u^{\gamma}-\gamma \p_{\theta\theta} u^{\gamma}\Big]u^{\gamma} d\theta dY\nonumber\\
=&\int_{-\infty}^{0}\int_0^{2\pi}\tilde{f}u^{\gamma} d\theta dY.
\end{align*}

It's easy to get
\beno
\int_{-\infty}^{0}\int_0^{2\pi}\big[-\partial_{YY}u^{\gamma}-\gamma \p_{\theta\theta} u^{\gamma}\big]u^{\gamma} d\theta dY=\|\partial_Yu^{\gamma}\|_2^2 +\gamma \|\p_{\theta} u^{\gamma}\|_2^2.
\eeno
Recall the estimates (\ref{decay behavior-prandtl}) and (\ref{Estimate of first linearized Euler equation}), we have
\begin{align*}
 &\big|\p^j_{\theta}\p^k_{Y}(\bar{u}-u_e(1))\big<Y\big>^l\big|\leq C(j,k,l)\eta,\\[5pt] &\big|\p^j_{\theta}\p^k_{Y}(\bar{v}-v^{(1)}_e(\theta,1))\big<Y\big>^l\big|\leq C(j,k,l)\eta,\
\big|\p^j_\theta v^{(1)}_e(\theta,1)\big|\leq C(j)\eta,
\end{align*}
thus we can deduce that
\begin{align*}
&-\int_{-\infty}^{0}\int_0^{2\pi}\Big[\bar{u}\partial_\theta u^{\gamma}+\bar{v}\partial_Yu^{\gamma}+\Big(\int_Y^0\p_{\theta}u^{\gamma}(\theta,z)dz \Big)\partial_Y\bar{u}+u^{\gamma}\partial_\theta \bar{u}\Big]u^{\gamma} d\theta dY+\int_{-\infty}^{0}\int_0^{2\pi}\tilde{f}u^{\gamma} d\theta dY\\
\leq&\int_{-\infty}^{0}\int_0^{2\pi}\frac{1}{2}\big[\p_\theta\bar{u}+\p_Y\bar{v}\big]\big(u^{\gamma}\big)^2 d\theta dr+\|Y^2\bar{u}_Y\|_{\infty}\Big\|\frac{\int_Y^0\p_{\theta}u^{\gamma} dz}{Y}\Big\|_2\Big\|\frac{u^{\gamma}}{Y}\Big\|_2\\[3pt]
    &+\|Y^2\bar{u}_\theta\|_{\infty}\Big\|\frac{u^{\gamma}}{Y}\Big\|_2^2+\|Y\tilde{f}\|_2\Big\|\frac{u^{\gamma}}{Y}\Big\|_2\\[5pt]
\leq& C\eta \big[\|\p_{\theta}u^{\gamma}\|_2^2+\|\p_Yu^{\gamma}\|_2^2\big]+C\|Y\tilde{f}\|_2\|\p_Y u^{\gamma}\|_2,
\end{align*}
where we use $\bar{u}_\theta+\bar{v}_Y=0$ and the Hardy inequality
$$
\Big\|\frac{\int_Y^0\p_{\theta}u^{\gamma} dz}{Y}\Big\|_2\leq C\|\p_{\theta}u^{\gamma}\|_2,\ \Big\|\frac{u^{\gamma}}{Y}\Big\|_2\leq C\|\p_Yu^{\gamma}\|_2.
$$

By collecting the above estimates, we obtain
\begin{align*}
&\|\partial_Yu^{\gamma}\|_2^2 +\gamma \|\p_{\theta} u^{\gamma}\|_2^2\leq C\eta \big[\|\p_{\theta}u^{\gamma}\|_2^2+\|\p_Yu^{\gamma}\|_2^2\big]+C\|Y\tilde{f}\|_2\|\p_Y u^{\gamma}\|_2,
\end{align*}
where $C$ is independent on $\eta$ and $\gamma$. If $\eta$ is small enough, we deduce that
\begin{align}\label{energy estimate prandtl}
&\|\partial_Yu^{\gamma}\|_2^2 +\gamma \|\p_{\theta} u^{\gamma}\|_2^2\leq C\eta \|\p_{\theta}u^{\gamma}\|_2^2+C\|Y\tilde{f}\|_2^2.
\end{align}

Next we multiply the first equation in (\ref{appro linear prandtl}) by $\p_{\theta}u^{\gamma}$ and integrate in $(\theta,r)\in(0,2\pi)\times(-\infty,0)$, we arrive at
\begin{align*}
&\int_{-\infty}^{0}\int_0^{2\pi}\Big[\bar{u}\partial_\theta u^{\gamma}+\bar{v}\partial_Yu^{\gamma}+\Big(\int_Y^0\p_{\theta}u^{\gamma}(\theta,z)dz \Big)\partial_Y\bar{u}+u^{\gamma}\partial_\theta \bar{u}-\partial_{YY}u^{\gamma}-\gamma \p_{\theta\theta} u^{\gamma}\Big]\p_{\theta}u^{\gamma} d\theta dY\nonumber\\
=&\int_{-\infty}^{0}\int_0^{2\pi}\tilde{f}\p_{\theta}u^{\gamma} d\theta dY.
\end{align*}
It's direct to obtain
\beno
\int_{-\infty}^{0}\int_0^{2\pi}\bar{u}\partial_\theta u^{\gamma} \p_{\theta}u^{\gamma} d\theta dY= \int_{-\infty}^{0}\int_0^{2\pi}\bar{u}\big|\partial_\theta u^{\gamma}\big|^2d\theta dY\geq(\alpha-C\eta)\|\p_\theta u^{\gamma}\|^2_2.
\eeno
The diffusion term can be computed as follows
\begin{align*}
\int_{-\infty}^{0}\int_0^{2\pi}\big[-\partial_{YY}u^{\gamma}-\gamma \p_{\theta\theta} u^{\gamma}\big]\p_\theta u^{\gamma} d\theta dY=\int_{-\infty}^{0}\int_0^{2\pi}\Big(\frac{1}{2}\p_{\theta}\big(\p_{Y}u^{\gamma}\big)^2-\frac{\gamma}{2}\p_{\theta}\big(\p_\theta u^{\gamma}\big)^2\Big)d\theta dY=0.
\end{align*}

Moreover, there holds
\begin{align*}
&-\int_{-\infty}^{0}\int_0^{2\pi}\Big[\bar{v}\partial_Yu^{\gamma}+\Big(\int_Y^0\p_{\theta}u^{\gamma}(\theta,z)dz \Big)\partial_Y\bar{u}+u^{\gamma}\partial_\theta \bar{u}\Big]\p_{\theta}u^{\gamma} d\theta dY+\int_{-\infty}^{0}\int_0^{2\pi}\tilde{f}\p_{\theta}u^{\gamma} d\theta dY\\
\leq& \|\bar{v}\|_{\infty}\|\p_Y u^{\gamma}\|_2\|\p_{\theta}u^{\gamma}\|_2+\|Y \bar{u}_Y\|_{\infty}\Big\|\frac{\int_Y^0\p_{\theta}u^{\gamma} dz}{Y}\Big\|_2\|\p_{\theta} u^{\gamma}\|_2\\[5pt]
&+\|Y\p_{\theta}\bar{u}\|_{\infty}\Big\|\frac{u^{\gamma}}{Y}\Big\|_2\|\p_{\theta}u^{\gamma}\|_2+\|\tilde{f}\|_2\|\p_{\theta}u^{\gamma}\|_2\\[5pt]
\leq&C\eta\big[\|\p_Yu^{\gamma}\|^2_2+\|\p_\theta u^{\gamma}\|^2_2\big]+\|\tilde{f}\|_2\|\p_{\theta}u^{\gamma}\|_2.
\end{align*}
Thus, we obtain
\begin{align*}
&\alpha\|\p_{\theta}u^{\gamma}\|_2^2\leq C \eta\big[\|\p_\theta u^{\gamma}\|_2^2+\|\p_Yu^{\gamma}\|_2^2\big]+\|\tilde{f}\|_2\|\p_{\theta}u^{\gamma}\|_2.
\end{align*}
We then choose $\eta_0$ small enough such that for any $\eta\in (0,\eta_0)$, there holds
\begin{align}\label{positive estimate prandtl}
&\alpha\|\p_{\theta}u^{\gamma}\|_2^2\leq C \eta\big\|\p_Yu^{\gamma}\|_2^2+C\|\tilde{f}\|_2^2.
\end{align}
It follows from (\ref{energy estimate prandtl}) and (\ref{positive estimate prandtl}) that  we have
\begin{align}\label{closed prandtl}
&\alpha\|\p_{\theta}u^{\gamma}\|_2^2+\|\partial_Yu^{\gamma}\|_2^2 +\gamma \|\p_{\theta} u^{\gamma}\|_2^2\leq C\|Y\tilde{f}\|_2^2+C\|\tilde{f}\|_2^2.
\end{align}

According to the first equation in (\ref{appro linear prandtl}), we deduce
\begin{align}\label{second derivative estimate}
&\|\p_{YY}u^{\gamma}\|^2_2+2\gamma\|\p_{\theta Y}u^{\gamma}\|^2_2+\gamma^2\|\p_{\theta \theta}u^{\gamma}\|^2_2=\Big\|\partial_{YY}u^{\gamma}+\gamma \p_{\theta\theta} u^{\gamma}\Big\|_2^2\nonumber\\
=&\bigg\|\bar{u}\partial_\theta u^{\gamma}+\bar{v}\partial_Yu^{\gamma}+\Big(\int_Y^0\p_{\theta}u^{\gamma}(\theta,z)dz \Big)\partial_Y\bar{u}+u^{\gamma}\partial_\theta \bar{u}-\tilde{f}\bigg\|_2^2\nonumber\\[5pt]
\leq &C\big[\|\tilde{f}\|^2_2+\|\p_{\theta}u^{\gamma}\|_2^2+\|\partial_Yu^{\gamma}\|_2^2\big]\leq C\|Y\tilde{f}\|_2^2+C\|\tilde{f}\|_2^2.
\end{align}

Collecting the estimates (\ref{closed prandtl}) and (\ref{second derivative estimate}), we obtain
\begin{align*}
&\alpha\|\p_{\theta}u^{\gamma}\|_2^2+\|\partial_Yu^{\gamma}\|_2^2 +\|\p_{YY}u^{\gamma}\|^2_2+\gamma \|\p_{\theta} u^{\gamma}\|_2^2+2\gamma\|\p_{\theta Y}u^{\gamma}\|^2_2+\gamma^2\|\p_{\theta \theta}u^{\gamma}\|^2_2\leq C\|\big<Y\big>\tilde{f}\|_2^2
\end{align*}
which shows the existence and uniqueness of solution for system (\ref{appro linear prandtl}) for any $\gamma>0$ in $\dot{H}^1_0$. Moreover the solution is smooth if $\tilde{f}$ is smooth. Set
\begin{align*}
u:=\lim_{\gamma\rightarrow0} u^{\gamma},\ v:=\int_{Y}^0\p_{\theta}u(\theta,d)dz,
\end{align*}
then
\begin{align}\label{new linear prandtl estimate}
&\alpha\|\p_{\theta}u\|_2^2+\|\partial_Y u\|_2^2 +\|\p_{YY} u\|^2_2\leq C\|\big<Y\big>\tilde{f}\|^2_2
\end{align}
and $(u,v)$ solves the system (\ref{new linearized Prandtl equation}) except the boundary condition $\lim\limits_{Y\rightarrow -\infty}\partial_Yu=0$.

Finally, we show that the derivatives of $u,v$ decay fast as $Y\rightarrow -\infty$. Let $\psi=\int_0^Y \bar{u}(\theta,z)dz$ and
\begin{align*}
\tilde{u}(\theta,\psi)=u(\theta, Y(\theta,\psi)), \ \tilde{v}(\theta,\psi)=v(\theta, Y(\theta,\psi)), \
F(\theta,\psi)=\frac{\tilde{f}(\theta, Y(\theta,\psi))}{\bar{u}(\theta, Y(\theta,\psi))},
\end{align*}
then there holds
\begin{eqnarray}\label{linear prandtl in new variable}
\left \{
\begin {array}{ll}
\partial_\theta \tilde{u}-\partial_\psi(a(\theta,\psi)\partial_\psi \tilde{u})+b(\theta,\psi)\tilde{u}+c(\theta,\psi)\tilde{v}=F(\theta,\psi),\\[7pt]
\tilde{u}(\theta+2\pi,\psi)=\tilde{u}(\theta,\psi),\\[5pt]
\tilde{u}(\theta, 0)=0, \ \lim_{\psi\rightarrow -\infty}\partial_\psi\tilde{u}(\theta, \psi)=0,
\end{array}
\right.
\end{eqnarray}
where
\begin{align*}
a(\theta,\psi)=\bar{u}(\theta, Y(\theta,\psi)),
 \ b(\theta,\psi)=\frac{\partial_\theta\bar{u}(\theta, Y(\theta,\psi))}{\bar{u}(\theta, Y(\theta,\psi))},
 \ c(\theta,\psi)=\frac{\partial_Y\bar{u}(\theta, Y(\theta,\psi))}{\bar{u}(\theta, Y(\theta,\psi))}.
\end{align*}
Notice that there exist $\eta_0>0$ such that for any $\eta\in (0,\eta_0)$, there holds
\begin{align*}
\frac{\alpha}{2}\leq \bar{u}(\theta,Y)\leq\alpha, \ \forall (\theta, Y)\in [0,2\pi]\times (-\infty,0].
\end{align*}
 Thus, we deduce that $\frac{\alpha}{2}\leq \frac{|\psi|}{|Y|}\leq \alpha$.

 \begin{align}\label{weight estimate in new variable}
 \|\partial_\theta \tilde{u} \psi^l\|_2+\|\partial_\psi \tilde{u}_{\neq}\psi^l\|_2+\|\partial_{\psi\psi} \tilde{u} \psi^l\|_2\leq C(l)\|F \big<\psi\big>^{l+1}\|_2,
 \end{align}
 where
 $$\tilde{u}_{\neq}=\tilde{u}-u_0(\psi), \ u_0(\psi)=\frac{1}{2\pi}\int_0^{2\pi}\tilde{u}(\theta,\psi)d\theta.$$
 From (\ref{new linear prandtl estimate}), we deduce that
 \begin{align*}
 \|\partial_\theta \tilde{u} \|_2+\|\partial_\psi \tilde{u}\|_2+\|\partial_{\psi\psi} \tilde{u} \|_2\leq C\|F \big<\psi\big>\|_2,
 \end{align*}
 thus (\ref{weight estimate in new variable}) holds for $l=0$.

 For any $l\geq 1$, multiplying $\tilde{u}_{\neq}\psi^{2l}$ in (\ref{linear prandtl in new variable}) and integrating in $[0,2\pi]\times (-\infty,0]$, we obtain
 \begin{align*}
 &\underbrace{\int_0^{2\pi}\int_{-\infty}^0\partial_\theta \tilde{u}\tilde{u}_{\neq}\psi^{2l}d\psi d\theta}_{I_1} -\underbrace{\int_0^{2\pi}\int_{-\infty}^0\partial_\psi(a(\theta,\psi)\partial_\psi \tilde{u})\tilde{u}_{\neq}\psi^{2l}d\psi d\theta}_{I_2}\\
 =&\underbrace{\int_0^{2\pi}\int_{-\infty}^0[F(\theta,\psi)-b(\theta,\psi)\tilde{u}-c(\theta,\psi)\tilde{v}]\tilde{u}_{\neq}\psi^{2l}d\psi d\theta}_{I_3}.
 \end{align*}
Obviously, $I_1=0$. Due to the fast decay of $b(\theta,\psi), c(\theta,\psi)$ as $\psi\rightarrow -\infty$, we deduce that
\begin{align*}
|I_3|\leq& C\|F \big<\psi\big>^{l+1}\|_2\|\partial_\theta\tilde{u}_{\neq}\psi^{l-1}\|_2+C(\|\partial_\theta \tilde{u} \|^2_2+\|\partial_\psi \tilde{u}\|^2_2)\\[5pt]
\leq &C\|F \big<\psi\big>^{l+1}\|_2^2+C\|\partial_\theta\tilde{u}\psi^{l-1}\|_2^2.
\end{align*}
Moreover, there holds
\begin{align*}
I_2=\underbrace{\int_0^{2\pi}\int_{-\infty}^0a(\theta,\psi)\partial_\psi \tilde{u}\partial_\psi\tilde{u}_{\neq}\psi^{2l}d\psi d\theta}_{I_{21}}
+\underbrace{2l\int_0^{2\pi}\int_{-\infty}^0a(\theta,\psi)\partial_\psi \tilde{u}\tilde{u}_{\neq}\psi^{2l-1}d\psi d\theta}_{I_{22}}.
\end{align*}
Notice that $a(\theta,\psi)=u_e(1)+u_p^{(0)}(\theta,Y(\theta,\psi))$, we deduce that
\begin{align*}
I_{21}=&u_e(1)\int_0^{2\pi}\int_{-\infty}^0\partial_\psi \tilde{u}\partial_\psi\tilde{u}_{\neq}\psi^{2l}d\psi d\theta+\int_0^{2\pi}\int_{-\infty}^0u_p^{(0)}(\theta,Y(\theta,\psi))\partial_\psi \tilde{u}\partial_\psi\tilde{u}_{\neq}\psi^{2l}d\psi d\theta\\
=&u_e(1)\int_0^{2\pi}\int_{-\infty}^0\partial_\psi \tilde{u}_{\neq}\partial_\psi\tilde{u}_{\neq}\psi^{2l}d\psi d\theta+\int_0^{2\pi}\int_{-\infty}^0u_p^{(0)}(\theta,Y(\theta,\psi))\partial_\psi \tilde{u}\partial_\psi\tilde{u}_{\neq}\psi^{2l}d\psi d\theta\\[5pt]
\geq& \frac{\alpha}{2}\|\partial_\psi \tilde{u}_{\neq}\psi^l\|^2_2-C\|\partial_\psi \tilde{u}\|^2_2.
\end{align*}
Moreover, by the H\"{o}lder inequality and the Poincar\'{e} inequality,  there holds
\begin{align*}
|I_{22}|\leq C(l)\|\partial_\psi \tilde{u}_{\neq}\psi^l\|_2\|\partial_\theta \tilde{u}\psi^{l-1}\|_2.
\end{align*}
Thus, we obtain
\begin{align}\label{first order norm estimate}
\|\partial_\psi \tilde{u}_{\neq}\psi^l\|_2\leq C(l)\|\partial_\theta \tilde{u}\psi^{l-1}\|_2+C\|F \big<\psi\big>^{l+1}\|_2.
\end{align}

Noticing
\begin{align*}
\partial_\theta \tilde{u}-a(\theta,\psi)\partial_{\psi\psi }\tilde{u}=F(\theta,\psi)+\partial_\psi a(\theta,\psi)\partial_\psi \tilde{u}-b(\theta,\psi)\tilde{u}-c(\theta,\psi)\tilde{v},
\end{align*}
we deduce that
\begin{align*}
\|[\partial_\theta \tilde{u}-a(\theta,\psi)\partial_{\psi\psi }\tilde{u}]\psi^l\|_2^2=\|[F(\theta,\psi)+\partial_\psi a(\theta,\psi)\partial_\psi \tilde{u}-b(\theta,\psi)\tilde{u}-c(\theta,\psi)\tilde{v}]\psi^l\|_2^2.
\end{align*}
The right side can be controlled by
\begin{align*}
C\|F \big<\psi\big>^{l}\|_2^2+C(\|\partial_\theta \tilde{u} \|^2_2+\|\partial_\psi \tilde{u}\|^2_2).
\end{align*}
Moreover, there holds
\begin{align*}
&\|[\partial_\theta \tilde{u}-a(\theta,\psi)\partial_{\psi\psi }\tilde{u}]\psi^l\|_2^2\\
=&\|\partial_\theta \tilde{u} \psi^l\|_2^2+\| a(\theta,\psi)\partial_{\psi\psi }\tilde{u}\psi^l\|_2^2-2\int_0^{2\pi}\int_{-\infty}^0a(\theta,\psi)\partial_\theta \tilde{u}\partial_{\psi\psi }\tilde{u}\psi^{2l}d\psi d\theta\\
\geq &\|\partial_\theta \tilde{u} \psi^l\|_2^2+\frac{\alpha}{2}\|\partial_{\psi\psi }\tilde{u}\psi^l\|_2^2\underbrace{-2\int_0^{2\pi}\int_{-\infty}^0a(\theta,\psi)\partial_\theta \tilde{u}\partial_{\psi\psi}\tilde{u}\psi^{2l}d\psi d\theta}_{I}.
\end{align*}
Integrating by parts, we deduce that
\begin{align*}
I=&\underbrace{2\int_0^{2\pi}\int_{-\infty}^0\partial_\psi a(\theta,\psi)\partial_\theta\tilde{u}\partial_{\psi}\tilde{u}\psi^{2l}d\psi d\theta}_{I_1}+\underbrace{4l\int_0^{2\pi}\int_{-\infty}^0a(\theta,\psi)\partial_\theta\tilde{u}\partial_{\psi}\tilde{u}\psi^{2l-1}d\psi d\theta}_{I_2}\\
&+\underbrace{2\int_0^{2\pi}\int_{-\infty}^0a(\theta,\psi)\partial_{\theta \psi}\tilde{u}\partial_{\psi}\tilde{u}\psi^{2l}d\psi d\theta}_{I_3}.
\end{align*}
Obviously, there holds
\begin{align*}
|I_1|+|I_3|\leq C(\|\partial_\theta \tilde{u} \|^2_2+\|\partial_\psi \tilde{u}\|^2_2).
\end{align*}
Moreover,
\begin{align*}
|I_2|=&\Big|4lu_e(1)\int_0^{2\pi}\int_{-\infty}^0\partial_\theta\tilde{u}\psi^{2l-1}\partial_{\psi}\tilde{u}d\psi d\theta+4l\int_0^{2\pi}\int_{-\infty}^0u_p^{(0)}(\theta,Y(\theta,\psi))\partial_\theta\tilde{u}\psi^{2l-1}\partial_{\psi}\tilde{u}d\psi d\theta\Big|\\
=&\Big|4lu_e(1)\int_0^{2\pi}\int_{-\infty}^0\partial_\theta\tilde{u}\psi^{2l-1}\partial_{\psi}\tilde{u}_{\neq}d\psi d\theta+4l\int_0^{2\pi}\int_{-\infty}^0u_p^{(0)}(\theta,Y(\theta,\psi))\partial_\theta\tilde{u}\psi^{2l-1}\partial_{\psi}\tilde{u}d\psi d\theta\Big|\\[5pt]
\leq & C(l)\|\partial_\theta \tilde{u}\psi^l \|_2\|\partial_\psi \tilde{u}_{\neq}\psi^{l-1}\|_2+C(l)(\|\partial_\theta \tilde{u} \|^2_2+\|\partial_\psi \tilde{u}\|^2_2).
\end{align*}
Thus, we obtain
\begin{align}\label{second estimate on norm}
\|\partial_\theta \tilde{u} \psi^l\|_2+\|\partial_{\psi\psi }\tilde{u}\psi^l\|_2\leq C(l)\|\partial_\psi \tilde{u}_{\neq}\psi^{l-1}\|_2+C(l)\|F \big<\psi\big>^{l+1}\|_2.
\end{align}
Combining the estimates (\ref{first order norm estimate}) and (\ref{second estimate on norm}), we obtain that for any $l\geq 1$, there holds
\begin{align*}
\|\partial_\theta \tilde{u} \psi^l\|_2+\|\partial_\psi \tilde{u}_{\neq}\psi^l\|_2+\|\partial_{\psi\psi} \tilde{u} \psi^l\|_2\leq C(l)(\|\partial_\psi \tilde{u}_{\neq}\psi^{l-1}\|_2+\|\partial_\theta \tilde{u}\psi^{l-1}\|_2)+C(l)\|F \big<\psi\big>^{l}\|_2.
\end{align*}
Thus, by induction, we obtain (\ref{weight estimate in new variable}).

Furthermore, by the Hardy inequality, we have
\begin{align*}
\|\partial_\psi \tilde{u}\psi^{l-1}\|_2\leq C\|\partial_{\psi\psi} \tilde{u} \psi^{l}\|_2,
\end{align*}
hence there holds
\begin{align*}
 \|\partial_\theta \tilde{u} \psi^l\|_2+\|\partial_\psi \tilde{u}\psi^{l-1}\|_2+\|\partial_{\psi\psi} \tilde{u} \psi^l\|_2\leq C\|F \big<\psi\big>^{l+1}\|_2.
 \end{align*}
Returning to the original variable, we obtain that for any $l\geq 1$, there holds
\begin{align*}
\big\|\big<Y\big>^{l}\p_{YY}u\big\|_2^2+\|\big<Y\big>^{l-1}\p_{\theta}u\|^2_2+\|\big<Y\big>^{l-1}\p_{Y}u\|^2_2\leq C( l)\|\big<Y\big>^{l+1}\tilde{f}\|^2_2.
\end{align*}

Furthermore, by induction, we obtain that for any $m\in \mathbb{N}_+, l\in \mathbb{N}$, there holds
\begin{align*}
&\sum_{j+k\leq m}\Big(\big\|\big<Y\big>^{l}\p^j_{\theta}\p^k_Y\p_{YY}u\big\|_2^2+\big\|\big<Y\big>^{l-1}\p^j_{\theta}\p^k_Y \p_\theta u\big\|_2^2+\big\|\big<Y\big>^{l-1}\p^j_{\theta}\p^k_Y \p_Y u\big\|_2^2\Big)\\
 &\leq C(m,l)\sum\limits_{j+k\leq m}\big\|\big<Y\big>^{l+1}\p^{j}_{\theta}\p^{k}_Y\tilde{f}\big\|^2_2\leq C(m,l)\eta^2.
\end{align*}

Noticing that  $\lim\limits_{Y\rightarrow -\infty}(u_\theta,u_Y)=0$ and $A_{1\infty}:=\lim\limits_{Y\rightarrow -\infty}u(\theta,Y)$ is a constant independent on $\theta$,  then by the Hardy inequality  we have for any $l\geq2$
\begin{align*}
\|Y^{l-2} (u-A_{1\infty})\|_2\leq C(l) \|Y^{l-1} \partial_{Y}u\|_2\leq C(l)\eta^2.
\end{align*}
This completes the proof of this proposition.
\end{proof}
Next, we construct the pressure $p_p^{(2)}(\theta,Y)$. Consider the equation
\begin{align}\label{equation for second pressure}
\partial_Yp_p^{(2)}(\theta, Y)=g_1(\theta,Y), \quad \lim_{Y\rightarrow -\infty}p_p^{(2)}(\theta,Y)=0,
\end{align}
where
\begin{align*}
g_1(\theta,Y)=&-Y\partial_Yp_p^{(1)}+\partial_{YY}v_p^{(1)}-u_e(1)\partial_\theta v_p^{(1)}-u_p^{(0)} (\partial_\theta v_e^{(1)}(\theta,1)+\partial_\theta v_p^{(1)})\\[5pt]
&-\partial_Yv_p^{(1)}(v_e^{(1)}(\theta,1)+v_p^{(1)})-2Yu'_e(1)u_p^{(0)}+u_e(1)\tilde{u}_p^{(1)}+[u_e^{(1)}(\theta,1)+A_1]u_p^{(0)}+u_p^{(0)}\tilde{u}_p^{(1)},
\end{align*}
here and below,
$$\tilde{u}_p^{(1)}=u_p^{(1)}-A_{1\infty}.$$
$g_1(\theta,Y)$ can be obtained by replacing $u_p^{(1)}$ by $\tilde{u}_p^{(1)}$ in the expansion (\ref{expansion near boundary}) and putting the new expansion into the second equation of (\ref{NS-curvilnear}), then collecting the $\varepsilon^1$-order terms together.
 Notice that $g_1(\theta,Y)$ decay fast as $Y\rightarrow -\infty$, we can get $p_p^{(2)}(\theta,Y)$ by solving (\ref{expansion near boundary}) and deduce that $p_p^{(2)}(\theta,Y)$ decay fast as $Y\rightarrow -\infty$.

By the same argument as above, we can obtain the well-posedness of  equations (\ref{first linearized prandtl problem another boundary}).
\begin{proposition} There exists $\eta_0>0$ such that for any $\eta\in(0,\eta_0)$, the equations (\ref{first linearized prandtl problem another boundary}) have a unique solution $(\hat{u}_p^{(1)},\hat{v}_p^{(2)})$ which satisfies
\begin{align}
\begin{aligned}
&\sum_{j+k\leq m}\int_0^{+\infty}\int_0^{2\pi}\big|\partial_\theta^j\partial_Z^k \big(\hat{u}_p^{(1)}-\hat{A}_{1\infty},v_p^{(2)}\big)\big|^2\big<Z\big>^{2l}d\theta dZ\leq C(m,l)\eta^2, \ \forall m,l\geq 0, \\
&\int_0^{2\pi}\hat{v}_p^{(2)}(\theta,Z)d\theta=0, \ \forall Z\geq 0, \label{decay behavior-prandtl-1-another}
\end{aligned}
\end{align}
where
$\hat{A}_{1\infty}:=\lim\limits_{Z \rightarrow +\infty}\hat{u}_p^{(1)}(\theta,Z)$ is a constant which satisfies $|\hat{A}_{1\infty}|\leq C\eta.$
\end{proposition}
Similarly, we can also construct $\hat{p}_p^{(2)}(\theta,Z)$ which decays fast as $Z\rightarrow +\infty$.

\subsubsection{Linearized Euler system for $(u_e^{(2)}, v_e^{(2)}, p_e^{(2)})$ and its solvability}

\indent

Let $r_1=\frac{1+2r_0}{3}, r_2=\frac{2+r_0}{3}$ and $\chi(r)\in C^\infty([r_0,1])$ be an increasing smooth function such that
\begin{align*}
\chi(r)=
\left\{
\begin{array}{lll}
0, \ r\in [r_0,r_1], \\[5pt]
1, \ r\in [r_2,1]
\end{array}
\right.
\end{align*}
and
\begin{align*}
\bar{u}_e^{(1)}(\theta,r):&=u_e^{(1)}(\theta,r)+\chi(r)A_{1\infty}+(1-\chi(r))\hat{A}_{1\infty}, \ \bar{v}_e^{(1)}(\theta,r)=v_e^{(1)}(\theta,r),\\
\bar{p}_e^{(1)}(\theta,r):&=p_e^{(1)}(\theta,r)+\int_{r_0}^r\frac{2u_e(s)}{s}[\chi(s)A_{1\infty}+(1-\chi(s))\hat{A}_{1\infty}]ds,
\end{align*}
then $(\bar{u}_e^{(1)},\bar{v}_e^{(1)},\bar{p}_e^{(1)})$ also satisfies the linearized Euler equations (\ref{outer-1 order equation})
with the boundary condition (\ref{outer-1 order-bc}).

Putting
\begin{align*}
&u^{\varepsilon}(\theta,r)=u_e(r)+\varepsilon \bar{u}_e^{(1)}(\theta,r)+\varepsilon^2 u_e^{(2)}(\theta,r)+\text{h.o.t.},\\[5pt]
&v^{\varepsilon}(\theta,r)=\varepsilon \bar{v}_e^{(1)}(\theta,r)+\varepsilon^2 v_e^{(2)}(\theta,r)+\text{h.o.t.},\\[5pt]
&p^{\varepsilon}(\theta,r)=p_e(r)+\varepsilon \bar{p}_e^{(1)}(\theta,r)+\varepsilon^2 p_e^{(2)}(\theta,r)+\text{h.o.t.}
\end{align*}
into the Navier-Stokes equations (\ref{NS-curvilnear}), we obtain the following linearized Euler equations for $(u_e^{(2)},v_e^{(2)}, p_e^{(2)})$
\begin{eqnarray}
\left \{
\begin {array}{ll}
u_e(r) \partial_\theta u_e^{(2)}+rv_e^{(2)}u'_e(r)+u_e(r)v_e^{(2)}+\partial_\theta p_e^{(2)}+\bar{u}_e^{(1)}\partial_\theta \bar{u}_e^{(1)}\\[5pt]
\quad \quad \quad \quad +\bar{v}_e^{(1)}r\partial_r\bar{u}_e^{(1)}+\bar{u}_e^{(1)}\bar{v}_e^{(1)}=ru''_e(r)+u'_e(r)-\frac{u_e(r)}{r}+rF_u,\\[5pt]
u_e(r) \partial_\theta v_e^{(2)}-2u_e(r)u_e^{(2)}+r\partial_rp_e^{(2)}+\bar{u}_e^{(1)}\partial_\theta \bar{v}_e^{(1)}+\bar{v}_e^{(1)}r\partial_r \bar{v}_e^{(1)}-(\bar{u}_e^{(1)})^2=rF_v,\\[7pt]
\partial_\theta u_e^{(2)}+\partial_r(rv_e^{(2)})=0,\label{outer-2 order equation}
\end{array}
\right.
\end{eqnarray}
 with the boundary conditions
\begin{align}\label{boundary condition for second Euler}
 v_e^{(2)}|_{r=1}=-v_p^{(2)}|_{Y=0},\ v_e^{(2)}|_{r=r_0}=-\hat{v}_p^{(2)}|_{Z=0}, \ v_e^{(2)}(\theta,r)=v_e^{(2)}(\theta+2\pi,r).
\end{align}

\begin{proposition}\label{solvability of second Euler equation}
The linearized Euler system  (\ref{outer-2 order equation}) has a solution $(u_e^{(2)}, v_e^{(2)}, p_e^{(2)})$ which satisfies
\begin{align}\label{Estimate of second linearized Euler equation}
\|\partial^k_\theta\partial^j_r(u_e^{(2)},v_e^{(2)})\|_2\leq C(j,k),  \ \forall j,k\geq 0.
\end{align}
\end{proposition}
\begin{proof}
First, we show that there holds
\begin{align}
\int_0^{2\pi}[\bar{u}_e^{(1)}\partial_\theta \bar{u}_e^{(1)}+\bar{v}_e^{(1)}r\partial_r\bar{u}_e^{(1)}+\bar{u}_e^{(1)}\bar{v}_e^{(1)}]d\theta=0, \ \forall r\in [r_0,1].
\label{necessary condition for the solvability-2out-modi}
\end{align}

We set
\beno
\bar{v}_e^{(1)}(\theta,r)=\sum_{n\neq 0}V_n(r)e^{in\theta}, \ \bar{u}_e^{(1)}(\theta,r)=a(r)-\sum_{n\neq 0}\frac{(rV_n)'(r)}{in}e^{in\theta},
\eeno
here we have used \eqref{ve1zeroav} and the third equation of \eqref{outer-1 order equation}.

Then, it's easy to obtain that
\begin{align}
&\bar{v}_e^{(1)}r\partial_r\bar{u}_e^{(1)}(\theta,r)+\bar{u}_e^{(1)}\bar{v}_e^{(1)}(\theta,r)\nonumber \\[5pt]
&=(ra'(r)+a(r))\bar{v}_e^{(1)}(\theta,r)-\bar{v}_e^{(1)}(\theta,r)\sum_{n\neq 0}\frac{e^{in\theta}}{in}[r(rV_n)''+(rV_n)'].\label{Estimate of second linearized Euler equation-1}
\end{align}
Moreover, due to $-\triangle(r\bar{v}_e^{(1)})+U_e(r)(r\bar{v}_e^{(1)})=0$, where $U_e(r)=\frac{1}{u_e(r)}\Big(u''_e(r)+\frac{u'_e(r)}{r}-\frac{u_e(r)}{r^2}\Big),$
we deduce that
\beno
\sum_{n\neq 0}e^{in\theta}[r^2(rV_n)''+r(rV_n)'-(n^2+r^2U_e(r))rV_n]=0,
\eeno
hence there holds $r(rV_n)''+(rV_n)'=(n^2+r^2U_e(r))V_n$.

 Thus, by \eqref{Estimate of second linearized Euler equation-1} we deduce that
 \begin{align}\label{identity of Euler-1}
&\bar{v}_e^{(1)}r\partial_r\bar{u}_e^{(1)}(\theta,r)+\bar{u}_e^{(1)}\bar{v}_e^{(1)}(\theta,r)\nonumber\\[5pt]
=&(ra'(r)+a(r))\bar{v}_e^{(1)}(\theta,r)+\bar{v}_e^{(1)}(\theta,r)\Big(\sum_{n\neq 0}ine^{in\theta}V_n-r^2U_e(r)\sum_{n\neq 0}\frac{e^{in\theta}}{in}V_n\Big)\nonumber\\
=&(ra'(r)+a(r))\bar{v}_e^{(1)}(\theta,r)+\bar{v}_e^{(1)}(\theta,r)\partial_\theta \bar{v}_e^{(1)}(\theta,r)
-\frac{r^2U_e(r)}{2}\partial_\theta\Big(\sum_{n\neq 0}\frac{e^{in\theta}}{in}V_n\Big)^2,
\end{align}
hence \eqref{necessary condition for the solvability-2out-modi} can be obtained.

Then, we can  follow the line of the construction of $(u_e^{(1)}, v_e^{(1)}, p_e^{(1)})$ to construct $(u_e^{(2)}, v_e^{(2)}, p_e^{(2)})$: eliminating pressure $p_e^{(2)}$ in the equation (\ref{outer-2 order equation}) and obtaining the elliptic equation for $rv_e^{(2)}$:
\begin{align*}
&-u_e(r)\triangle(rv_e^{(2)})+\Big(u''_e(r)+\frac{u'_e(r)}{r}-\frac{u_e(r)}{r^2}\Big)(rv_e^{(2)})\\[5pt]
=&-r\partial_r\big(\bar{u}_e^{(1)}\partial_\theta \bar{u}_e^{(1)}+\bar{v}_e^{(1)}r\partial_r\bar{u}_e^{(1)}+\bar{u}_e^{(1)}\bar{v}_e^{(1)}\big)+\partial_\theta\big(\bar{u}_e^{(1)}\partial_\theta \bar{v}_e^{(1)}+\bar{v}_e^{(1)}r\partial_r \bar{v}_e^{(1)}-(\bar{u}_e^{(1)})^2\big)\\[5pt]
&+r\partial_r\Big(ru''_e(r)+u'_e(r)-\frac{u_e(r)}{r}+rF_u\Big)-\partial_\theta(rF_v).
\end{align*}
 This elliptic equation with boundary condition (\ref{boundary condition for second Euler}) is solvable by noticing (\ref{second ODE for leading Euler flow}) and (\ref{necessary condition for the solvability-2out-modi}). Then $u_e^{(2)}$ can be obtained from the divergence-free condition; finally, construct pressure $p_e^{(2)}$ from the first equation in  (\ref{outer-2 order equation}). Thus, $(u_e^{(2)}, v_e^{(2)},p_e^{(2)})$ satisfies the equation (\ref{outer-2 order equation}) and (\ref{Estimate of second linearized Euler equation}).\\
\end{proof}

\begin{Remark}
In this remark, we show that the solvability of (\ref{outer-2 order equation}) also determine the form of $u_e(r)$. In fact, integrating the first equation of (\ref{outer-2 order equation}) with respect to $\theta$ in $[0,2\pi]$, notice $\int_0^{2\pi}v_e^{(2)}(\theta,r)d\theta=0$ and (\ref{necessary condition for the solvability-2out-modi}), we deduce that
$$u''_e(r)+\frac{u'_e(r)}{r}-\frac{u_e(r)}{r^2}+\bar{F}_u=0,$$
this is the equation in (\ref{second ODE for leading Euler flow}).
\end{Remark}

Let $A_1(r)$ be a smooth function such that
\begin{align}\label{corrector of first order Euler equation}
\left\{
\begin{array}{ll}
rA''_1(r)+A'_1(r)-\frac{A_1(r)}{r}\\[5pt]
\quad \quad \quad =\frac{1}{2\pi}\int_0^{2\pi}\Big(v^{(1)}_e\partial_r(ru^{(2)}_e)+v^{(2)}_e\partial_r(r\bar{u}^{(1)}_e)-\Big(r\partial_{rr}\bar{u}^{(1)}_e
 + \partial_r\bar{u}^{(1)}_e-\frac{\bar{u}^{(1)}_e}{r}\Big)\Big)d\theta,\\[10pt]
 A_1(r_0)=A_1(1)=0,
 \end{array}
 \right.
\end{align}
then $\|\partial_r^kA_1(r)\|_\infty\leq C(k)\eta.$

Set
\begin{align}\label{modify Euler}
\tilde{u}_e^{(1)}=\bar{u}_e^{(1)}+A_1(r),\ \tilde{v}_e^{(1)}=\bar{v}_e^{(1)},\ \tilde{p}_e^{(1)}=\bar{p}_e^{(1)}+\int_{r_0}^r\frac{2u_e(s)}{s}A_1(s)ds,
\end{align}
then $(\tilde{u}_e^{(1)},\tilde{v}_e^{(1)}, \tilde{p}_e^{(1)})$ is also a solution of (\ref{outer-1 order equation}) with the boundary condition (\ref{outer-1 order-bc}), and there holds
 \begin{align}\label{estimate on modified first Euler}
 \|\partial_\theta^j\partial_r^k(\tilde{u}^{(1)}_e,\tilde{v}_e^{(1)})\|_\infty\leq C(j,k)\eta, \ \forall k,j\geq 0.
 \end{align}
Thus, by the same argument as Proposition \ref{solvability of second Euler equation}, we deduce that the following linearized Euler system
\begin{eqnarray}
\left \{
\begin {array}{ll}
u_e(r) \partial_\theta u_e^{(2)}+rv_e^{(2)}u'_e(r)+u_e(r)v_e^{(2)}+\partial_\theta p_e^{(2)}+\tilde{u}_e^{(1)}\partial_\theta \tilde{u}_e^{(1)}\\[5pt]
\quad \quad \quad \quad \quad +\tilde{v}_e^{(1)}r\partial_r\tilde{u}_e^{(1)}+\tilde{u}_e^{(1)}\tilde{v}_e^{(1)}=ru''_e(r)+u'_e(r)-\frac{u_e(r)}{r}+rF_\theta,\\[5pt]
u_e(r) \partial_\theta v_e^{(2)}-2u_e(r)u_e^{(2)}+r\partial_rp_e^{(2)}+\tilde{u}_e^{(1)}\partial_\theta \tilde{v}_e^{(1)}+\tilde{v}_e^{(1)}r\partial_r \tilde{v}_e^{(1)}-(\tilde{u}_e^{(1)})^2=rF_r,\\[7pt]
\partial_\theta u_e^{(2)}+r\partial_rv_e^{(2)}+v_e^{(2)}=0\label{outer-2 order equation-2}
\end{array}
\right.
\end{eqnarray}
 with the boundary conditions
\begin{align}\label{boundary condition for second Euler equation}
 v_e^{(2)}|_{r=1}=-v_p^{(2)}|_{Y=0},\ v_e^{(2)}|_{r=r_0}=-\hat{v}_p^{(2)}|_{Z=0}, \ v_e^{(2)}(\theta,r)=v_e^{(2)}(\theta+2\pi,r),
\end{align}
has a solution $(\bar{u}_e^{(2)}, \bar{v}_e^{(2)}, \bar{p}_e^{(2)})$ which satisfies
\begin{align}\label{Estimate of modified second linearized Euler equation}
\|\partial_\theta^j\partial_r^k(\bar{u}^{(2)}_e,\bar{v}_e^{(2)})\|_\infty\leq C(j,k).
\end{align}

\subsubsection{Linearized Prandtl equations for  $(u_p^{(2)},v_p^{(3)})$, $(\hat{u}_p^{(2)},\hat{v}_p^{(3)})$ and their solvabilities}
\indent

Let
\begin{align*}
&u^\varepsilon(\theta,r)=u_e(r)+u_p^{(0)}(\theta,Y)+\varepsilon\big[\tilde{u}_e^{(1)}(\theta,r)+\tilde{u}_p^{(1)}(\theta,Y)\big]
+\varepsilon^2\big[\bar{u}_e^{(2)}(\theta,r)+u_p^{(2)}(\theta,Y)\big]+\text{h.o.t.},\\[5pt]
&v^\varepsilon(\theta,r)=\varepsilon\big[\tilde{v}_e^{(1)}(\theta,r)+v_p^{(1)}(\theta,Y)\big]
+\varepsilon^2\big[\bar{v}_e^{(2)}(\theta,r)+v_p^{(2)}(\theta,Y)\big]+\varepsilon^3[v_e^{(3)}(\theta,r)+v_p^{(3)}(\theta,Y)]+\text{h.o.t.},\\[5pt]
&p^\varepsilon(\theta,r)=p_e(r)+\varepsilon\big[\tilde{p}_e^{(1)}(\theta,r)+p_p^{(1)}(\theta,Y)\big]
+\varepsilon\big[\bar{p}_e^{(2)}(\theta,r)+p_p^{(2)}(\theta,Y)\big]+\varepsilon^3p_p^{(3)}(\theta,Y)+\text{h.o.t.},
\end{align*}
with the boundary conditions
\begin{align*}
 \bar{u}_e^{(2)}(\theta,1)+u_p^{(2)}(\theta,0)=0,\ v_e^{(3)}(\theta,1)+v_p^{(3)}(\theta,0)=0, \ \lim_{Y\rightarrow -\infty}(\partial_Yu_p^{(2)},v_p^{(3)})(\theta,Y)=(0,0),
\end{align*}
we obtain the following linearized steady Prandtl equations for $(u_p^{(2)},v_p^{(3)})$
\begin{eqnarray}
\left \{
\begin {array}{ll}
\big(u_e(1)+u_p^{(0)}\big)\partial_\theta u_p^{(2)}+\big(v_e^{(1)}(\theta,1)+ v_p^{(1)}\big)\partial_Yu_p^{(2)}+(v_p^{(3)}-v_p^{(3)}(\theta,0))\partial_{Y}u_p^{(0)}\\[5pt]
\quad \quad \quad \quad +(u_p^{(2)}+\bar{u}_e^{(2)}(\theta,1))\partial_\theta u_p^{(0)}-\partial_{YY}u_p^{(2)}=f_2(\theta,Y),\\[5pt]
\partial_\theta u_p^{(2)}+\partial_Yv_p^{(3)}+\partial_Y(Yv_p^{(2)})=0,\\[5pt]
u_p^{(2)}(\theta,Y)=u_p^{(2)}(\theta+2\pi,Y),\ v_p^{(3)}(\theta,Y)=v_p^{(3)}(\theta+2\pi,Y),\\[5pt]
u_p^{(2)}\big|_{Y=0}=-\bar{u}_e^{(2)}\big|_{r=1},\  \lim\limits_{Y\rightarrow -\infty}\partial_Yu_p^{(2)}(\theta,Y)= \lim\limits_{Y\rightarrow -\infty}v_p^{(3)}(\theta,Y)=0
\label{second linearized prandtl problem near 1}
\end{array}
\right.
\end{eqnarray}
here
{\small\begin{align*}
f_2(\theta,Y)=&-\partial_\theta p_p^{(2)}+Y\partial_{YY}u_p^{(1)}+\partial_Yu_p^{(1)}+\partial_{\theta\theta}u_p^{(0)}-u_p^{(0)}
-\tilde{u}_p^{(1)}\partial_\theta u_p^{(1)}-v_p^{(2)}\partial_Y u_p^{(1)}-\sum_{i+j=2}v_p^{(i)}Y\partial_Y u_p^{(j)}\\[5pt]
&-\sum_{k=0}^2\sum_{i+j=2-k, (k,j)\neq (0,2)}\Big(\frac{\partial_r^k\tilde{u}_e^{(i)}(\theta,1)}{k!}Y^k \partial_\theta \tilde{u}_p^{(j)}+\tilde{u}_p^{(j)}\frac{\partial_r^k\partial_\theta \tilde{u}_e^{(i)}(\theta,1)}{k!}Y^k\Big)+u_e^{(2)}(\theta,1)\partial_\theta u_p^{(0)}\\[5pt]
&-\sum_{k=0}^1\sum_{i+j=2-k}\Big(\frac{\partial_r^k\tilde{v}_e^{(i)}(\theta,1)}{k!}Y^{k+1}\partial_Y \tilde{u}_p^{(j)}+v_p^{(i)}\frac{\partial_r^k(r\partial_r \tilde{u}_e^{(j)})(\theta,1)}{k!}Y^k\Big)\\[5pt]
&-\sum_{k=0}^2\sum_{i+j=3-k, (k,j)\neq (0,2),(0,0)}\frac{\partial_r^k\tilde{v}_e^{(i)}(\theta,1)}{k!}Y^k \partial_Y \tilde{u}_p^{(j)},
\end{align*}}
where $\tilde{u}_p^{(0)}=u_p^{(0)}, \ \tilde{u}_e^{(0)}=u_e(r),\ \tilde{u}_e^{(2)}=\bar{u}_e^{(2)}, \ \tilde{v}_e^{(2)}=\bar{v}_e^{(2)}.$\\

Similarly, let
\begin{align*}
&u^\varepsilon(\theta,r)=u_e(r)+\widehat{u}_p^{(0)}(\theta,Z)+\varepsilon\big[\tilde{u}_e^{(1)}(\theta,r)+(\widehat{u}_p^{(1)}(\theta,Z)-\hat{A}_1)\big]
+\varepsilon^2\big[\bar{u}_e^{(2)}(\theta,r)+\widehat{u}_p^{(2)}(\theta,Z)\big]+\text{h.o.t.},\\[5pt]
&v^\varepsilon(\theta,r)=\varepsilon\big[\tilde{v}_e^{(1)}(\theta,r)+\widehat{v}_p^{(1)}(\theta,Z)\big]+
\varepsilon^2\big[\bar{v}_e^{(2)}(\theta,r)+\widehat{v}_p^{(2)}(\theta,Z)\big]+\varepsilon^3[v_e^{(3)}(\theta,r)+\widehat{v}_p^{(3)}(\theta,Z)]+\text{h.o.t.},\\[5pt]
&p^\varepsilon(\theta,r)=p_e(r)+\varepsilon\big[\tilde{p}_e^{(1)}(\theta,r)+\widehat{p}_p^{(1)}(\theta,Z)\big]+
\varepsilon^2\big[\tilde{p}_e^{(2)}(\theta,r)+\widehat{p}_p^{(2)}(\theta,Z)\big]+\varepsilon\widehat{p}_p^{(3)}(\theta,Z)+\text{h.o.t.},
\end{align*}
with the following boundary conditions
\begin{align*}
 \bar{u}_e^{(2)}(\theta,r_0)+\hat{u}_p^{(2)}(\theta,0)=0, \ v_e^{(3)}(\theta,r_0)+\widehat{v}_p^{(3)}(\theta,0)=0,\ \lim_{Z\rightarrow +\infty}(\partial_Z\hat{u}_p^{(2)},\hat{v}_p^{(3)})(\theta,Z)=(0,0),
\end{align*}
we obtain the following linearized Prandlt equations for $(\hat{u}_p^{(2)},\hat{v}_p^{(3)})$
\begin{eqnarray}
\left \{
\begin {array}{ll}
\big(u_e(r_0)+\hat{u}_p^{(0)}\big)\partial_\theta \hat{u}_p^{(2)}+\big(v_e^{(1)}(\theta,r_0)+ \hat{v}_p^{(1)}\big)r_0\partial_Z\hat{u}_p^{(2)}+(\hat{u}_p^{(2)}+\bar{u}_e^{(2)}(\theta,r_0))\partial_\theta \hat{u}_p^{(0)}\\[5pt]
\quad \quad \quad \quad \quad \quad \quad \quad \quad +(\hat{v}_p^{(3)}-\hat{v}_p^{(3)}(\theta,0))r_0\partial_Z\hat{u}_p^{(0)}-r_0\partial_{ZZ}\hat{u}_p^{(2)}=\hat{f}_2(\theta,Z),\\[5pt]
\partial_\theta \hat{u}_p^{(2)}+r_0\partial_Z\hat{v}_p^{(3)}+\partial_Z(Z\hat{v}_p^{(2)})=0,\\[5pt]
\hat{u}_p^{(2)}(\theta,Z)=\hat{u}_p^{(2)}(\theta+2\pi,Z),\ \hat{v}_p^{(3)}(\theta,Z)=\hat{v}_p^{(3)}(\theta+2\pi,Z),\\[5pt]
\hat{u}_p^{(2)}\big|_{Z=0}=-u_e^{(2)}\big|_{r=r_0},\  \lim\limits_{Z\rightarrow +\infty}(\partial_Z\hat{u}_p^{(2)},\hat{v}_p^{(3)}(\theta,Z))=(0,0).
\label{second linearized prandtl problem another boundary}
\end{array}
\right.
\end{eqnarray}
The expressions of functions $\hat{f}_2(\theta,Z)$ are similar to $f_2(\theta,Y)$, we omit the details here.

\begin{proposition}\label{decay estimates of higher order linearized Prandtl} There exists $\eta_0>0$ such that for any $\eta\in(0,\eta_0)$, the equations (\ref{second linearized prandtl problem near 1}) have a unique solution $(u_p^{(2)},v_p^{(3)})$  and the equations (\ref{second linearized prandtl problem another boundary}) have a unique solution $(\hat{u}_p^{(2)},\hat{v}_p^{(3)})$ which satisfies
\begin{align}\label{decay behavior-prandtl-2}
\begin{aligned}
&\sum_{j+k\leq m}\int_{-\infty}^0\int_0^{2\pi}\Big|\partial_\theta^j\partial_Y^k (u_p^{(2)}-A_{2\infty},v_p^{(3)} )\Big|^2\big<Y\big>^{2l}d\theta dY\leq C(m,l),\ \forall m,l\geq 0, \\
&\sum_{j+k\leq m}\int_0^{+\infty}\int_0^{2\pi}\Big|\partial_\theta^j\partial_Z^k (\hat{u}_p^{(2)}-\hat{A}_{2\infty},\hat{v}_p^{(3)})\Big|^2\big<Z\big>^{2l}d\theta dZ\leq C(m,l),\ \forall m,l\geq 0,\\
& \int_0^{2\pi}v_p^{(3)}(\theta, Y)d\theta=\int_0^{2\pi}\hat{v}_p^{(3)}(\theta, Z)d\theta=0, \ \forall \ Y\leq 0,\ Z\geq 0
\end{aligned}
\end{align}
where
\begin{align*}
A_{2\infty}:=\lim_{Y\rightarrow -\infty}u_p^{(2)}(\theta,Y),\ \hat{A}_{2\infty}:=\lim_{Z\rightarrow +\infty}\hat{u}_p^{(2)}(\theta,Z).
\end{align*}
and satisfies $|A_{2\infty}|+|\hat{A}_{2\infty}|\leq C$.
\end{proposition}
The proof is same with Proposition \ref{decay estimates of linearized Prandtl} by noticing that $f_2(\theta,Y)$  decays very fast as $Y\rightarrow -\infty$, we omit the details.

We construct the pressure $p_p^{(3)}(\theta,Y)$ by considering the equation
\begin{align}\label{equation of second pressure}
\partial_Yp_p^{(3)}(\theta, Y)=g_2(\theta,Y), \ \lim_{Y\rightarrow -\infty}p_p^{(3)}(\theta,Y)=0,
\end{align}
where
{\small\begin{align*}
g_2(\theta,Y)=&\partial_{YY}v_p^{(2)}+Y\partial_{YY}v_p^{(1)}+\partial_Yv_p^{(1)}-2\partial_\theta u_p^{(0)}-Y\partial_Yp_p^{(2)}-\sum_{i+j=3} v_p^{(i)}\partial_Yv_p^{(j)}\\[5pt]
&-\sum_{i+j=2}\Big(\tilde{u}_p^{(i)}\partial_\theta v_p^{(j)}+v_p^{(i)}Y\partial_Yv_p^{(j)}-\tilde{u}_p^{(i)}\tilde{u}_p^{(j)}+\tilde{v}_e^{(i)}(\theta,1)Y\partial_Yv_p^{(j)}
+v_p^{(i)}\partial_r\tilde{v}_e^{(j)}(\theta,1)\Big)\\[5pt]
&-\sum_{k=0}^1\sum_{i+j=2-k}\Big(\frac{\partial_r^k\tilde{u}_e^{(i)}(\theta,1)}{k!}Y^k\partial_\theta v_p^{(j)}+\frac{\partial_r^k\partial_\theta \tilde{v}_e^{(j)}(\theta,1)}{k!}Y^k\partial_\theta \tilde{u}_p^{(i)}\Big)\\[5pt]
&-\sum_{k=0}^2\sum_{i+j=2-k}\Big(\frac{\partial_r^k\tilde{u}_e^{(i)}(\theta,1)}{k!}Y^k\tilde{u}_p^{(j)}+\frac{\partial_r^k \tilde{u}_e^{(j)}(\theta,1)}{k!}Y^k \tilde{u}_p^{(i)}\Big),
\end{align*}}
where $\tilde{u}_p^{(0)}=u_p^{(0)},  \ \tilde{u}_e^{(0)}=u_e(r),\ \tilde{u}_e^{(2)}=\bar{u}_e^{(2)}+A_{2\infty}, \ \tilde{v}_e^{(2)}=\bar{v}_e^{(2)}$, and here and below $\tilde{u}_p^{(2)}=u_p^{(2)}-A_{2\infty}$.  $g_2(\theta,Y)$ can be derived by the same argument as $g_1(\theta,Y)$.
Moreover, notice that $g_2(\theta,Y)$ decay fast as $Y\rightarrow -\infty$, we can obtain $p_p^{(3)}$ by solving (\ref{equation of second pressure}) and $p_p^{(3)}$ also decay fast as $Y\rightarrow -\infty$.

Similarly, we can construct $\hat{p}^{(3)}_p(\theta,Z)$ and $\hat{p}^{(3)}_p(\theta,Z)$ which also decay fast as $Z\rightarrow +\infty$.

\subsubsection{Linearized Euler system for $(u_e^{(3)},v_e^{(3)},p_e^{(3)})$ and its solvability}
\indent

Let
\begin{align*}
\tilde{u}_e^{(2)}(\theta,r):&=\bar{u}_e^{(2)}(\theta,r)+\chi(r)A_{2\infty}+(1-\chi(r))\hat{A}_{2\infty}, \ \tilde{v}_e^{(2)}(\theta,r)=\bar{v}_e^{(2)}(\theta,r),\\[5pt]
\tilde{p}_e^{(2)}(\theta,r):&=\bar{p}_e^{(2)}(\theta,r)+\int_{r_0}^r\frac{2u_e(s)}{s}[\chi(s)A_{2\infty}+(1-\chi(s))\hat{A}_{2\infty}]ds,
\end{align*}
then $(\tilde{u}_e^{(2)},\tilde{v}_e^{(2)},\tilde{p}_e^{(2)})$ also satisfies the linearized Euler equations (\ref{outer-2 order equation-2}) with the boundary conditions (\ref{boundary condition for second Euler equation}), and there holds
 \begin{align}\label{estimate on modified second Euler}
 \|\partial_\theta^j\partial_r^k(\tilde{u}^{(2)}_e,\tilde{v}_e^{(2)})\|_\infty\leq C(j,k), \ \forall k,j\geq 0.
 \end{align}
Putting
\begin{align*}
&u^{\varepsilon}(\theta,r)=u_e(r)+\sum_{i=1}^3\varepsilon^{i} \tilde{u}_e^{(i)}(\theta,r)+\text{h.o.t.},\\[5pt]  &v^{\varepsilon}(\theta,r)=\sum_{i=1}^3\varepsilon^{i} \tilde{v}_e^{(i)}(\theta,r)+\text{h.o.t.}, \\[5pt]
&p^{\varepsilon}(\theta,r)=p_e(r)+\sum_{i=1}^3\varepsilon^{i} \tilde{p}_e^{(i)}(\theta,r)+\text{h.o.t.},
\end{align*}
into the Navier-Stokes equations (\ref{NS-curvilnear}), we find that $(\tilde{u}_e^{(3)},\tilde{v}_e^{(3)},\tilde{p}_e^{(3)})$ satisfies the following linearized Euler equations in $\Omega$
\begin{eqnarray}
\left \{
\begin {array}{ll}
u_e(r) \partial_\theta \tilde{u}_e^{(3)}+r\tilde{v}_e^{(3)}u'_e(r)+u_e(r)\tilde{v}_e^{(3)}+\partial_\theta \tilde{p}_e^{(3)}+f_e(\theta,r)=0,\\[5pt]
u_e(r) \partial_\theta \tilde{v}_e^{(3)}-2u_e(r)\tilde{u}_e^{(3)}+r\partial_r\tilde{p}_e^{(3)}+g_e(\theta,r)=0,\\[7pt]
\partial_\theta \tilde{u}_e^{(3)}+\partial_r(r\tilde{v}_e^{(3)})=0,\label{outer-3 order equation}
\end{array}
\right.
\end{eqnarray}
equipped with the following boundary conditions
\begin{align}
 \tilde{v}_e^{(3)}|_{r=1}=-v_p^{(3)}|_{Y=0},\ \tilde{v}_e^{(3)}|_{r=r_0}=-\hat{v}_p^{(3)}|_{Z=0}, \ \tilde{v}_e^{(3)}(\theta,r)=\tilde{v}_e^{(3)}(\theta+2\pi,r),\nonumber
\end{align}
where
\begin{align*}
f_e(\theta,r)=&\tilde{u}_e^{(1)}\partial_\theta \tilde{u}_e^{(2)}+\tilde{u}_e^{(2)}\partial_\theta \tilde{u}_e^{(1)}+\tilde{v}_e^{(1)}r\partial_r\tilde{u}_e^{(2)}+\tilde{v}_e^{(2)}r\partial_r\tilde{u}_e^{(1)}
+\tilde{u}_e^{(1)}\tilde{v}_e^{(2)}+\tilde{u}_e^{(2)}\tilde{v}_e^{(1)}\\
&-\Big(\frac{\partial_{\theta\theta}\tilde{u}_e^{(1)}}{r}+r\partial_{rr}\tilde{u}_e^{(1)}+\partial_r\tilde{u}_e^{(1)}
-\frac{2}{r}\partial_\theta \tilde{v}_e^{(1)}-\frac{\tilde{u}_e^{(1)}}{r}\Big),\\[5pt]
g_e(\theta,r)=&\tilde{u}_e^{(1)}\partial_\theta \tilde{v}_e^{(2)}+\tilde{u}_e^{(2)}\partial_\theta \tilde{v}_e^{(1)}+\tilde{v}_e^{(1)}r\partial_r \tilde{v}_e^{(2)}+\tilde{v}_e^{(2)}r\partial_r \tilde{v}_e^{(1)}-2\tilde{u}_e^{(1)}\tilde{u}_e^{(2)}\\[5pt]
&-\Big(\frac{\partial_{\theta\theta}\tilde{v}_e^{(1)}}{r}+r\partial_{rr}\tilde{v}_e^{(1)}+\partial_r\tilde{v}_e^{(1)}
-\frac{2}{r}\partial_\theta \tilde{u}_e^{(1)}+\frac{\tilde{v}_e^{(1)}}{r}\Big).
\end{align*}

\begin{proposition}\label{Estimate of third linearized Euler equation}
The linearized Euler equations  (\ref{outer-3 order equation}) have a solution $(\tilde{u}_e^{(3)}, \tilde{v}_e^{(3)}, \tilde{p}_e^{(3)})$ which satisfies
\begin{align}\label{Estimate of third linearized Euler equation}
\|\partial^k_\theta\partial^j_r(\tilde{u}_e^{(3)},\tilde{v}_e^{(3)})\|_2\leq C(j,k),  \ \forall j,k\geq 0.
\end{align}
\end{proposition}
\begin{proof}
We only need to check $\int_0^{2\pi}f_e(\theta,r)d\theta=0$, that is,
\begin{align}\label{identity goal}
\int_0^{2\pi}\Big(\tilde{v}^{(1)}_e\partial_r(r\tilde{u}^{(2)}_e)+\tilde{v}^{(2)}_e\partial_r(r\tilde{u}^{(1)}_e)
-\Big(r\partial_{rr}\tilde{u}^{(1)}_e
 + \partial_r\tilde{u}^{(1)}_e-\frac{\tilde{u}^{(1)}_e}{r}\Big)\Big)d\theta=0, \ \forall r\in [r_0,1].
\end{align}
By (\ref{corrector of first order Euler equation}), we only need to prove
\begin{align*}
&\int_0^{2\pi}\Big(v^{(1)}_e\partial_r(ru^{(2)}_e)+v^{(2)}_e\partial_r(r\bar{u}^{(1)}_e)\Big)(\theta,r)d\theta\\
=&\int_0^{2\pi}\Big(\tilde{v}^{(1)}_e\partial_r(r\tilde{u}^{(2)}_e)+\tilde{v}^{(2)}_e\partial_r(r\tilde{u}^{(1)}_e)\Big)(\theta,r)d\theta, \ \forall r\in [r_0,1].
\end{align*}
Notice that $v_e^{(1)}=\tilde{v}_e^{(1)}$, $\int_0^{2\pi}v^{(2)}_e d\theta=\int_0^{2\pi}\tilde{v}^{(2)}_e d\theta=0$ and $\tilde{u}^{(1)}_e=\bar{u}^{(1)}_e+A_1(r)$, thus we only need to prove
\begin{align*}
\int_0^{2\pi}\Big(v^{(1)}_e\partial_r(ru^{(2)}_e-r\tilde{u}^{(2)}_e)+(v^{(2)}_e-\tilde{v}^{(2)}_e)\partial_r(r\bar{u}^{(1)}_e)\Big)(\theta,r)d\theta
=0, \ \forall r\in [r_0,1].
\end{align*}
For simplicity, we write $[v^{(1)}, u^{(1)}]=[v^{(1)}_e, \bar{u}^{(1)}_e]$, and $[v^{(2)}, u^{(2)}]=[v^{(2)}_e-\tilde{v}^{(2)}_e, u^{(2)}_e-\tilde{u}^{(2)}_e]$. Notice that $v^{(1)}$ satisfies equations (\ref{Euler equation normal}), i.e.
\begin{align*}
-r\Delta(rv^{(1)})+Brv^{(1)}=0,
\end{align*}
where $B=\frac{r}{u_e(r)}\Big(u''_e(r)+\frac{u'_e(r)}{r}-\frac{u_e(r)}{r^2}\Big)$ is a radial function. By equations (\ref{outer-2 order equation}) and (\ref{outer-2 order equation-2}), $v^{(2)}$ satisfies the following equation:
\begin{align*}
-r\Delta(rv^{(2)})+Brv^{(2)}=f^{(2)},
\end{align*}
where
\begin{align*}
f^{(2)}=&\frac{r}{u_e}\Big[r^2v^{(1)}\Delta A_1-r A_1\Delta(rv^{(1)})-\frac{A_1}{r}v^{(1)}\Big]\\
=&\frac{r}{u_e}\Big[r^2v^{(1)}\Delta A_1-rA_1B v^{(1)}-\frac{A_1}{r}v^{(1)}\Big].
\end{align*}
We set $I[h]=\sum\limits_{n\neq 0}\frac{e^{in\theta}}{in}h_n(r)$ for $h$ satisfying $\int_{0}^{2\pi}hd\theta=0$. Since $u^{(2)}_\theta=-(rv^{(2)})_r$, there holds
\begin{align*}
\Big(ru^{(2)}-\int_0^{2\pi}ru^{(2)}d\theta\Big)_r&=-I\Big[\big(r(rv^{(2)})_r\big)_r\Big]\\
           &=-I\Big[r\Delta(rv^{(2)})-\frac{v^{(2)}_{\theta\theta}}{r}\Big]
           =I[f^{(2)}]+\frac{v^{(2)}_\theta}{r}-rBI[v^{(2)}].
\end{align*}
Similarly,
\begin{align*}
\Big(ru^{(1)}-\int_0^{2\pi}ru^{(1)}d\theta\Big)_r=\frac{v^{(1)}_\theta}{r}-rBI[v^{(1)}].
\end{align*}
Thus we have
\begin{align*}
&\int_0^{2\pi}\Big(v^{(1)}_e\partial_r(ru^{(2)}_e-r\tilde{u}^{(2)}_e)+(v^{(2)}_e-\tilde{v}^{(2)}_e)\partial_r(r\bar{u}^{(1)}_e)\Big)(\theta,r)d\theta\\
=&\int_0^{2\pi}\Big(v^{(1)}\partial_r(ru^{(2)})-v^{(2)}\partial_r(r u^{(1)})\Big)d\theta\\
=&\int_0^{2\pi}\Big(v^{(1)}I[f^{(2)}]+\frac{v^{(1)}v^{(2)}_\theta+v^{(2)}v^{(1)}_\theta}{r}-rBv^{(1)}I[v^{(2)}]-rBv^{(2)}I[v^{(1)}]\Big)d\theta\\
=&\int_0^{2\pi}v^{(1)}I[f^{(2)}]d\theta.
\end{align*}
Moreover, notice that $f^{(2)}$ is a radial function times $v^{(1)}$, we deduce that
\begin{align*}
\int_0^{2\pi}v^{(1)}I[f^{(2)}]d\theta=0,
\end{align*}
this complete the proof.

\end{proof}

\subsubsection{Linearized Prandtl equations for $(u_p^{(3)},v_p^{(4)})$, $(\hat{u}_p^{(3)},\hat{v}_p^{(4)})$ and their solvabilities}

\indent

Let
\begin{align*}
&u^\varepsilon(\theta,r)=u_e(r)+u_p^{(0)}(\theta,Y)+\sum_{i=1}^2\varepsilon^i\big[\tilde{u}_e^{(i)}(\theta,r)+\tilde{u}_p^{(i)}(\theta,Y)\big]
+\varepsilon^3[\tilde{u}_e^{(3)}(\theta,r)+ u_p^{(3)}(\theta,Y)\big]+\text{h.o.t.},\\[5pt]
&v^\varepsilon(\theta,r)=\sum_{i=1}^3\varepsilon^i\big[\tilde{v}_e^{(i)}(\theta,r)+v_p^{(i)}(\theta,Y)\big]
+\varepsilon^4v_p^{(4)}(\theta,Z)+\text{h.o.t.},\\[5pt]
&p^\varepsilon(\theta,r)=p_e(r)+\sum_{i=1}^3\varepsilon^i\big[\tilde{p}_e^{(i)}(\theta,r)+p_p^{(i)}(\theta,Y)\big]
+\varepsilon^4p_p^{(4)}(\theta,Y)+\text{h.o.t.},
\end{align*}
with the boundary conditions
\begin{align*}
 \tilde{u}_e^{(3)}(\theta,1)+u_p^{(3)}(\theta,0)=0, \ \lim_{Y\rightarrow \infty}\partial_Yu_p^{(3)}(\theta,Y)=v_p^{(4)}(\theta,0)=0,
\end{align*}
we obtain the following linearized Prandtl problem for $(u_p^{(3)},v_p^{(4)})$
\begin{eqnarray}
\left \{
\begin {array}{ll}
\big(u_e(1)+u_p^{(0)}\big)\partial_\theta u_p^{(3)}+\big(v_e^{(1)}(\theta,1)+ v_p^{(1)}\big)\partial_Yu_p^{(3)}+(u_p^{(3)}+\tilde{u}_e^{(3)}(\theta,1))\partial_\theta u_p^{(0)} \\[5pt]
\quad \quad \quad \quad v_p^{(4)}\partial_{Y}u_p^{(0)} -\partial_{YY}u_p^{(3)}=f_3(\theta,Y)\\[5pt]
\partial_\theta u_p^{(3)}+\partial_Yv_p^{(4)}+\partial_Y(Yv_p^{(3)})=0,\\[5pt]
u_p^{(3)}(\theta,Y)=u_p^{(3)}(\theta+2\pi,Y),\quad v_p^{(4)}(\theta,Y)=v_p^{(4)}(\theta+2\pi,Y),\\[5pt]
u_p^{(3)}\big|_{Y=0}=-\tilde{u}_e^{(3)}\big|_{r=1},\ \ \lim\limits_{Y\rightarrow \infty}\partial_Yu_p^{(3)}(\theta,Y)=v_p^{(4)}(\theta,0)=0
\label{third linearized prandtl problem near 1}
\end{array}
\right.
\end{eqnarray}
and
\begin{align}\label{fourth pressure}
\partial_Yp_p^{(4)}(\theta, Y)=g_3(\theta,Y), \ p_p^{(4)}(\theta,0)=0,
\end{align}
where
{\small \begin{align*}
f_3(\theta,Y)=&-\partial_\theta p_p^{(3)}+Y\partial_{YY}u_p^{(3)}+\partial_Yu_p^{(3)}+\sum_{k=0}^1(-1)^k\frac{Y^k\partial_{\theta\theta}\tilde{u}_p^{(1-k)}}{k!}+2\partial_\theta v_p^{(2)}-2Y\partial_\theta v_p^{(1)}\\[5pt]
&-\sum_{k=0}^1(-1)^k\frac{Y^k\tilde{u}_p^{(1-k)}}{k!}-\sum_{i+j=3, 1\leq i\leq 2}\tilde{u}_p^{(i)}\partial_\theta\tilde{u}_p^{(j)}-\sum_{i+j=3 }[v_p^{(i)}Y\partial_Y\tilde{u}_p^{(j)}+\tilde{u}_p^{(i)}v_p^{(j)}]\\[5pt]
&-\sum_{k=0}^3\sum_{i+j=3-k, (k,j)\neq (0,3)}\Big(\frac{\partial_r^k\tilde{u}_e^{(i)}(\theta,1)}{k!}Y^k\partial_\theta\tilde{u}_p^{(j)}
+\tilde{u}_p^{(i)}\frac{\partial_r^k\partial_\theta\tilde{u}_e^{(j)}(\theta,1)}{k!}Y^k\Big)\\[5pt]
&-\sum_{k=0}^3\sum_{i+j=3-k}\Big(\frac{\partial_r^k\tilde{v}_e^{(i)}(\theta,1)}{k!}Y^{k+1}\partial_Y\tilde{u}_p^{(j)}
+v_p^{(i)}\frac{\partial_r^k(r\partial_r\tilde{u}_e^{(j)})(\theta,1)}{k!}Y^k\Big)\\[5pt]
&-\sum_{k=0}^3\sum_{i+j=4-k, (k,j)\neq (0,3),(0,0),i\leq 3}\frac{\partial_r^k\tilde{v}_e^{(i)}(\theta,1)}{k!}Y^k\partial_Y\tilde{u}_p^{(j)}\\[5pt]
&-\sum_{k=0}^3\sum_{i+j=3-k}\Big(\frac{\partial_r^k\tilde{u}_e^{(i)}(\theta,1)}{k!}Y^{k}v_p^{(j)}
+\tilde{u}_p^{(i)}\frac{\partial_r^k\tilde{v}_e^{(j)}(\theta,1)}{k!}Y^k\Big),\\
g_3(\theta,Y)=&\partial_{YY}v_p^{(3)}+Y\partial_{YY}v_p^{(2)}+\partial_{Y}v_p^{(2)}
+\partial_{\theta\theta}v_p^{(1)}-\sum_{k=0}^1(-1)^k\frac{Y^k}{k!}\partial_\theta \tilde{u}_p^{(2-k)}-v_p^{(1)}-Y\partial_Yp_p^{(3)}\\[5pt]
&-\sum_{i+j=3}[\tilde{u}_p^{(i)}\partial_\theta v_p^{(j)}+v_p^{(i)}Y\partial_Y v_p^{(j)}-\tilde{u}_p^{(i)}\tilde{u}_p^{(j)}]-\sum_{i+j=4} v_p^{(i)}Y\partial_Y v_p^{(j)}-\sum_{k=0}^2\sum_{i+j=4-k}\frac{\partial_r^k\tilde{v}_e^{(i)}(\theta,1)}{k!}Y^{k}\partial_Y v_p^{(j)}\\[5pt]
&-\sum_{k=0}^3\sum_{i+j=3-k}\Big(\frac{\partial_r^k\tilde{u}_e^{(i)}(\theta,1)}{k!}Y^{k}\partial_\theta v_p^{(j)}
+\tilde{u}_p^{(i)}\frac{\partial_r^k\partial_\theta \tilde{v}_e^{(j)}(\theta,1)}{k!}Y^k\Big)\\[5pt]
&+\sum_{k=0}^3\sum_{i+j=3-k}\Big(\frac{\partial_r^k\tilde{u}_e^{(i)}(\theta,1)}{k!}Y^{k}\tilde{u}_p^{(j)}
+\tilde{u}_p^{(i)}\frac{\partial_r^k\tilde{u}_e^{(j)}(\theta,1)}{k!}Y^k\Big)\\[5pt]
&-\sum_{k=0}^1\sum_{i+j=3-k}\Big(\frac{\partial_r^k\tilde{v}_e^{(i)}(\theta,1)}{k!}Y^{k+1}\partial_Y v_p^{(j)}
+v_p^{(i)}\frac{\partial_r^k(r\partial_r \tilde{v}_e^{(j)})(\theta,1)}{k!}Y^k\Big)
\end{align*}}
with $\tilde{u}_p^{(0)}=u_p^{(0)},\ \tilde{u}^{(0)}_e=u_e(r), \ \tilde{u}_p^{(3)}=u_p^{(3)}.$

Similarly, let
{\small \begin{align*}
&u^\varepsilon(\theta,r)=u_e(r)+\widehat{u}_p^{(0)}(\theta,Z)+\sum_{i=1}^2\varepsilon^i\big[\tilde{u}_e^{(i)}(\theta,r)
+(\widehat{u}_p^{(i)}(\theta,Z)-\hat{A}_{i\infty})\big]
+\varepsilon^3\big[\tilde{u}_e^{(3)}(\theta,r)+\widehat{u}_p^{(3)}(\theta,Z)\big]+\text{h.o.t.},\\[5pt]
&v^\varepsilon(\theta,r)=\sum_{i=1}^3\varepsilon^i\big[\tilde{v}_e^{(i)}(\theta,r)+\widehat{v}_p^{(i)}(\theta,Z)\big]+
\varepsilon^4\widehat{v}_p^{(4)}(\theta,Z)+\text{h.o.t.},\\[5pt]
&p^\varepsilon(\theta,r)=p_e(r)+\sum_{i=1}^3\varepsilon\big[\tilde{p}_e^{(i)}(\theta,r)+\widehat{p}_p^{(i)}(\theta,Z)\big]+
\varepsilon^4\widehat{p}_p^{(4)}(\theta,Z)+\text{h.o.t.},
\end{align*}}
with the following boundary conditions
\begin{align*}
 \tilde{u}_e^{(3)}(\theta,r_0)+\hat{u}_p^{(3)}(\theta,0)=0,\ \lim_{Z\rightarrow+\infty}\partial_Z\hat{u}_p^{(3)}(\theta,Z)=\hat{v}_p^{(4)}(\theta,0)=0,
\end{align*}
we obtain the following linearized Prandlt equations for $(\hat{u}_p^{(3)},\hat{v}_p^{(4)})$
\begin{eqnarray}
\left \{
\begin {array}{ll}
\big(u_e(r_0)+\hat{u}_p^{(0)}\big)\partial_\theta \hat{u}_p^{(3)}+\big(v_e^{(1)}(\theta,r_0)+ \hat{v}_p^{(1)}\big)r_0\partial_Z\hat{u}_p^{(3)}+\hat{u}_p^{(3)}\partial_\theta \hat{u}_p^{(0)}\\[5pt]
\quad \quad \quad \quad \quad \quad \quad \quad \quad +\hat{v}_p^{(4)}r_0\partial_Z\hat{u}_p^{(0)}-r_0\partial_{ZZ}\hat{u}_p^{(3)}=\hat{f}_3(\theta,Z),\\[5pt]
\partial_\theta \hat{u}_p^{(3)}+r_0\partial_Z\hat{v}_p^{(4)}+\partial_Z(Z\hat{v}_p^{(3)})=0,\\[5pt]
\hat{u}_p^{(3)}(\theta,Z)=\hat{u}_p^{(3)}(\theta+2\pi,Z),\ \hat{v}_p^{(4)}(\theta,Z)=\hat{v}_p^{(4)}(\theta+2\pi,Z),\\[5pt]
\hat{u}_p^{(3)}\big|_{Z=0}=-\tilde{u}_e^{(3)}\big|_{r=r_0},\ \lim\limits_{Z\rightarrow+\infty}\partial_Z\hat{u}_p^{(3)}(\theta,Z)=\hat{v}_p^{(4)}(\theta,0)=0
\label{third linearized prandtl problem another boundary}
\end{array}
\right.
\end{eqnarray}
and the pressure $\hat{p}_p^{(4)}$ satisfies
\begin{align}\label{fourth pressure another}
\partial_Z\hat{p}_p^{(4)}(\theta, Z)=\hat{g}_3(\theta,Z), \ \hat{p}_p^{(4)}(\theta,0)=0.
\end{align}
The functions $(\hat{f}_3(\theta,Z),\hat{g}_3(\theta,Z))$ are similar to $(f_3(\theta,Y),g_3(\theta,Y))$, we omit the details here.

\begin{proposition}There exists $\eta_0>0$ such that for any $\eta\in(0,\eta_0)$, the equations (\ref{third linearized prandtl problem near 1}) have a unique solution $(u_p^{(3)},v_p^{(4)})$ and the equations (\ref{third linearized prandtl problem another boundary}) have a unique solution $(\hat{u}_p^{(3)},\hat{v}_p^{(4)})$ which satisfy
\begin{align}\label{decay behavior-prandtl-3}
\begin{aligned}
&\sum_{0<j+k\leq m}\int_{-\infty}^0\int_0^{2\pi}\big|\partial_\theta^j\partial_Y^k (u_p^{(3)},v_p^{(4)} )\big|^2\big<Y\big>^{2l}d\theta dY\leq C(m,l), \ \forall \ m,l\geq 0, \\
&\sum_{0<j+k\leq m}\int_0^{+\infty}\int_0^{2\pi}\big|\partial_\theta^j\partial_Z^k (\hat{u}_p^{(3)},\hat{v}_p^{(4)})\big|^2\big<Z\big>^{2l}d\theta dZ\leq C(m,l),\ \forall \ m,l\geq 0,\\
& \int_0^{2\pi}v_p^{(4)}(\theta, Y)d\theta=\int_0^{2\pi}\hat{v}_p^{(4)}(\theta, Z)d\theta=0, \ \forall \ Y\leq 0,\ Z\geq 0
\end{aligned}
\end{align}
and
\begin{align*}
\|(u_p^{(3)},v_p^{(4)})\|_\infty+\|(\hat{u}_p^{(3)},\hat{v}_p^{(4)})\|_\infty\leq C.
\end{align*}
\end{proposition}
The proof is same with Proposition \ref{decay estimates of linearized Prandtl}, we omit the details. Moreover, we can construct $(p_p^{(4)},\hat{p}_p^{(4)} )$ by solving the equations (\ref{fourth pressure}) and \eqref{fourth pressure another}.

\subsection{Approximate solutions}\label{Approximate solutions}

\indent

In this subsection, we construct an approximate solution of Navier-Stokes equations (\ref{NS-curvilnear}).
Set
\begin{align*}
\tilde{u}_p^a(\theta,r):=&\chi(r)\Big(u_p^{(0)}(\theta,Y)+\varepsilon \tilde{u}_p^{(1)}(\theta,Y)
+\varepsilon^2\tilde{u}_p^{(2)}(\theta,Y)+\varepsilon^3u_p^{(3)}(\theta,Y)\Big)\\[5pt]
\quad \quad \quad &+(1-\chi(r))\Big(\hat{u}_p^{(0)}(\theta,Z)+\varepsilon\big[\hat{u}_p^{(1)}(\theta,Z)-\hat{A}_{1\infty}\big]
+\varepsilon^2[\hat{u}_p^{(2)}(\theta,Z)-\hat{A}_{2\infty}]+\varepsilon^3\hat{u}_p^{(3)}(\theta,Z)\Big)\\[5pt]
:=&\chi(r)u_p^a+(1-\chi(r))\hat{u}_p^a,\\[5pt]
\tilde{v}_p^a(\theta,r):=&\chi(r)\Big(\sum_{i=1}^4\varepsilon^i v_p^{(i)}(\theta,Y)\Big)
+(1-\chi(r))\Big(\sum_{i=1}^4\varepsilon^i \hat{v}_p^{(i)}(\theta,Z)\Big)\\[5pt]
:=&\chi(r)v_p^a+(1-\chi(r))\hat{v}_p^a,\\[5pt]
\tilde{p}_p^a(\theta,r):=&\chi^2(r)\Big(\sum_{i=1}^4\varepsilon^i p_p^{(i)}(\theta,Y)\Big)+(1-\chi(r))^2\Big(\sum_{i=1}^4\varepsilon^i \hat{p}_p^{(i)}(\theta,Z)\Big)\\[5pt]
:=&\chi^2(r)p_p^a+(1-\chi(r))^2\hat{p}_p^a,
\end{align*}
and
\begin{align*}
u_e^a(\theta,r):=&u_e(r)+\sum_{i=1}^3\varepsilon^i \tilde{u}_e^{(i)}(\theta,r), \\[5pt]
 v_e^a(\theta,r):=&\sum_{i=1}^3\varepsilon^i \tilde{v}_e^{(i)}(\theta,r),\\[5pt]
 p_e^a(\theta,r):=&p_e(r)+\sum_{i=1}^3\varepsilon^i \tilde{p}_e^{(i)}(\theta,r).
\end{align*}

We construct an approximate solution $(u^a,v^a,p^a)$
\begin{align}
u^a(\theta,r):&=u_e^a(\theta,r)+\tilde{u}_p^a(\theta,r)+\varepsilon^4h(\theta,r),\label{approximate solution-1}\\[5pt]
v^a(\theta,r):&=v_e^a(\theta,r)+\tilde{v}_p^a(\theta,r),\label{approximate solution-2}\\[5pt]
p^a(\theta,r):&=p_e^a(\theta,r)+\tilde{p}_p^a(\theta,r),\label{approximate solution-3}
\end{align}
where the corrector $h(\theta,r)$ will be given in Appendix B which satisfies
\begin{align*}
h(\theta, 1)=0, \ h(\theta,r_0)=0, \ \|\partial_\theta^j\partial_r^kh\|_2\leq C(j,k)\varepsilon^{-k}
\end{align*}
and makes $(u^a, v^a)$ be divergence-free
\begin{align*}
u^a_\theta+(rv^a)_r=0.
\end{align*}
Moreover, $(u^a, v^a)$  satisfies the following boundary conditions
\begin{align*}
 u^a(\theta+2\pi,r)&=u^a(\theta,r), \ v^a(\theta+2\pi,r)=v^a(\theta,r),\\[5pt]
u^a(\theta, 1)&=\alpha+\eta f(\theta), \ v^a(\theta,1)=0,\\[5pt]
u^a(\theta, r_0)&=\beta+\eta g(\theta),\ v^a(\theta, r_0)=0.
\end{align*}
By collecting the estimates (\ref{estimate on modified first Euler}), (\ref{estimate on modified second Euler}), (\ref{Estimate of third linearized Euler equation}),  we deduce that
\begin{align}\label{estimate on Euler parts}
&\|u^a_e-u_e(r)\|_\infty+\|(\partial_r(u^a_e-u_e(r)),\partial_\theta u^a_e)\|_\infty\leq C\varepsilon(\eta+\varepsilon), \nonumber\\[5pt] &\|v^a_e\|_\infty+\|(\partial_rv^a_e, \partial_\theta v^a_e)\|_\infty\leq C \varepsilon(\eta+\varepsilon),
\end{align}
and collecting the estimates (\ref{decay behavior-prandtl}), (\ref{decay behavior-prandtl-1}), (\ref{decay behavior-prandtl-2}), (\ref{decay behavior-prandtl-3}) one has
\begin{align}\label{estimate on Prandtl parts}
\|(Y^j\partial_Y^ku_p^a, Y^j\partial_\theta^ku_p^a)\|_\infty\leq C (\eta+\varepsilon), \ \|(Y^j\partial_Y^kv_p^a,Y^j\partial_\theta^kv_p^a)\|_\infty\leq C \varepsilon(\eta+\varepsilon),   \ \forall j\leq 2, k\leq 2.
\end{align}
Similarly, there holds
\begin{align*}
\|(Z^j\partial_Z^k\hat{u}_p^a, Z^j\partial_\theta^k\hat{u}_p^a)\|_\infty\leq C (\eta+\varepsilon),  \ \|(Z^j\partial_Z^k\hat{v}_p^a,Z^j\partial_\theta^k\hat{v}_p^a)\|_\infty\leq C \varepsilon(\eta+\varepsilon),   \ \forall j\leq 2, k\leq 2.
\end{align*}
By (\ref{Taylor-Couette flow}), (\ref{approximate solution-1}), (\ref{estimate on Euler parts}) and (\ref{estimate on Prandtl parts}), we deduce that there exists $\varepsilon_0>0, \eta_0>0$ such that for any $\varepsilon\in (0,\varepsilon_0), \eta\in (0,\eta_0)$, there holds
\begin{align}\label{positive lower bound}
u^a(\theta,r)\geq d_0>0, \ \forall (\theta, r)\in \Omega,
\end{align}
where $d_0$ is a fixed number.

Finally, set
 \begin{align*}
 R_u^a:&=u^au^a_\theta+v^aru^a_r+u^av^a+p^a_\theta-\varepsilon^2\Big( ru^a_{rr}
+\frac{u^a_{\theta\theta}}{r}+2\frac{v^a_{\theta}}{r}+u^a_r-\frac{u^a}{r}\Big)-\varepsilon^2rF_u,\\[5pt]
R_v^a:&=u^av^a_\theta+v^arv^a_r-(u^a)^2+rp^a_r-\varepsilon^2 \Big(rv^a_{rr}+\frac{v^a_{\theta\theta}}{r}-2\frac{u^a_{\theta}}{r}+v^a_r-\frac{v^a}{r}\Big)-\varepsilon^2rF_v,
 \end{align*}
 then there holds
 \begin{align*}
 \|R_u^a\|_2+\|\partial_\theta R_u^a\|_2\leq C\varepsilon^4, \  \|R_v^a\|_2+\|\partial_\theta R_v^a\|_2\leq C\varepsilon^4
 \end{align*}
and $(u^a, v^a, p^a)$ satisfies
{\small\begin{eqnarray}\label{app equation}
\left\{
\begin{array}{lll}
u^au^a_\theta+v^aru^a_r+u^av^a+p^a_{\theta}-\varepsilon^2\big( ru^a_{rr}
+\frac{u^a_{\theta\theta}}{r}+2\frac{v^a_{\theta}}{r}+u^a_r-\frac{u^a}{r}\big)=\varepsilon^2rF_u+R_u^a,\ (\theta,r)\in \Omega,\\[5pt]
u^av^a_\theta+v^arv^a_r-(u^a)^2+rp^a_r-\varepsilon^2\big( rv^a_{rr}+\frac{v^a_{\theta\theta}}{r}-2\frac{u^a_{\theta}}{r}+v^a_r-\frac{v^a}{r}\big)=\varepsilon^2rF_v+R_v^a,\ (\theta,r)\in \Omega,\\[5pt]
 u^a_\theta+(rv^a)_r=0, \ (\theta,r)\in \Omega, \\[5pt]
 u^a(\theta+2\pi,r)=u^a(\theta,r), \ v^a(\theta+2\pi,r)=v^a(\theta,r), \ (\theta,r)\in \Omega,\\[5pt]
u^a(\theta,1)=\alpha+f(\theta)\eta,\ v^a(\theta,1)=0, \ \theta\in [0,2\pi], \\[5pt]
 u^a(\theta,r_0)=\beta+g(\theta)\eta,\ v^a(\theta,r_0)=0, \ \theta\in [0,2\pi] .
\end{array}
\right.
\end{eqnarray}}

\section{Linear stability estimates of error equations}

\indent
In this section, we derive the error equations and establish the linear stability estimate of the error equations.

\subsection{Error equations}
\indent

Set the error
\beno
u:=u^\varepsilon-u^a,\ v:=v^\varepsilon-v^a,\ p:=p^\varepsilon-p^a,
\eeno
then there hold
{\small \begin{eqnarray}\label{error equation}
\left\{
\begin{array}{lll}
u^au_\theta+uu^a_\theta+uu_\theta+v^aru_r+vru^a_r+vru_r+v^au+vu^a+vu+p_\theta-\varepsilon^2\big( ru_{rr}
+\frac{u_{\theta\theta}}{r}+2\frac{v_{\theta}}{r}+u_r-\frac{u}{r}\big)=R_u^a,\\[5pt]
u^av_\theta+uv^a_\theta+uv_\theta+v^arv_r+vrv^a_r+vrv_r-(u^2+2u u^a)+rp_r-\varepsilon^2 \big( rv_{rr}+\frac{v_{\theta\theta}}{r}-2\frac{u_{\theta}}{r}+v_r-\frac{v}{r}\big)=R_v^a,\\[5pt]
u_\theta+(rv)_r=0,  \\[5pt]
u(\theta+2\pi,r)=u(\theta,r), \ v(\theta+2\pi,r)=v(\theta,r), \\[5pt]
u(\theta,1)=0,\ v(\theta,1)=0, \\[5pt]
 u(\theta,r_0)=0,\ v(\theta,r_0)=0.
\end{array}
\right.
\end{eqnarray}}

We set
\begin{align*}
S_u:&=u^au_\theta+v^aru_r+uu^a_\theta+vru^a_r+v^au+vu^a,\\[5pt]
S_v:&=u^av_\theta+v^arv_r+uv^a_\theta+vrv^a_r-2uu^a,
\end{align*}
then the error equations become
\begin{align}\label{e:error equation}
\left\{
\begin{array}{lll}
-\varepsilon^2\big(r u_{rr}
+\frac{u_{\theta\theta}}{r}+2\frac{v_{\theta}}{r}+u_r-\frac{u}{r}\big)+p_\theta+S_u=R_u,\\[5pt]
-\varepsilon^2\big( rv_{rr}+\frac{v_{\theta\theta}}{r}-2\frac{u_{\theta}}{r}+v_r-\frac{v}{r}\big)+rp_r+S_v=R_v,\\[5pt]
u_\theta+(rv)_r=0,  \\[5pt]
u(\theta,r)=u(\theta+2\pi,r),\  v(\theta,r)=v(\theta+2\pi,r), \\[5pt]
u(\theta,1)=0,\ v(\theta,1)=0,\\[5pt]
u(\theta,r_0)=0,\ v(\theta,r_0)=0,
 \end{array}
\right.
\end{align}
where
\beno
R_u:=R_u^a-uu_\theta-vru_r-vu,\ R_v:=R_v^a-uv_\theta-vrv_r+u^2.
\eeno
\subsection{Linear stability estimate}\label{section:Linear stability estimate}
\indent

In this part, we consider the following linear system in $\Omega$
 \begin{align}\label{linear error equation}
\left\{
\begin{array}{lll}
-\varepsilon^2\big(r u_{rr}
+\frac{u_{\theta\theta}}{r}+2\frac{v_{\theta}}{r}+u_r-\frac{u}{r}\big)+p_\theta+S_u=F_1,\\[5pt]
-\varepsilon^2\big( rv_{rr}+\frac{v_{\theta\theta}}{r}-2\frac{u_{\theta}}{r}+v_r-\frac{v}{r}\big)+rp_r+S_v=F_2,\\[5pt]
u_\theta+(rv)_r=0,  \\[5pt]
u(\theta,r)=u(\theta+2\pi,r),\  v(\theta,r)=v(\theta+2\pi,r), \\[5pt]
u(\theta,1)=0,\ v(\theta,1)=0,\\[5pt]
u(\theta,r_0)=0,\ v(\theta,r_0)=0,
 \end{array}
\right.
\end{align}
and establish its linear stability estimate (\ref{e:linear stability estimates for linear equation}).

Firstly, due to the third equation in \eqref{linear error equation} and the boundary condition of $v$, we easily deduce that
\begin{align*}
\int_0^{2\pi}v(\theta,r)d\theta=0,\ \forall \ r\in [r_0,1].
\end{align*}
Hence, there holds
\begin{align*}
\int_0^{2\pi}v^2(\theta,r)d\theta\leq \int_0^{2\pi}v_\theta^2(\theta,r)d\theta,  \ \forall \ r\in [r_0,1].
\end{align*}

The following stability estimate is a direct result of the basic energy estimate (\ref{e:basic Energy estimate}) and  positivity estimate (\ref{e:positivity estimate}).
\begin{proposition}
 Let $(u,v)$ be a smooth solution of (\ref{linear error equation}), then there exist $\varepsilon_0>0, \eta_0>0$ such that for any $\varepsilon\in (0,\varepsilon_0), \eta\in(0,\eta_0)$, there holds
\begin{align}\label{e:linear stability estimates for linear equation}
\varepsilon^2\|u_r\|_2^2+\|(u_\theta, v_\theta)\|_2^2\leq C\|(F_1, F_2)\|_2^2+C\Big|\int_{r_0}^{1}\int_0^{2\pi}\big(uF_1 +vF_2 \big)d\theta dr\Big|.
\end{align}
\end{proposition}

\subsubsection{Basic energy estimate}
\indent

In this subsection, we establish the following basic energy estimate.

 \begin{lemma}\label{basic energy estimate}
 Let $(u,v)$ be a smooth solution of (\ref{linear error equation}), then there exist $\varepsilon_0>0, \eta_0>0$ such that for any $\varepsilon\in (0,\varepsilon_0), \eta\in(0,\eta_0)$, there holds
\begin{align}\label{e:basic Energy estimate}
\varepsilon^2\|u_r\|_2^2\leq C\|(u_\theta,v_\theta)\|_2^2+C\Big|\int_{r_0}^{1}\int_0^{2\pi}\big[u F_1+vF_2 \big]d\theta dr\Big|.
\end{align}
\end{lemma}
\begin{proof}
Multiplying the first equation in (\ref{e:error equation}) by $u$, the second equation  in (\ref{e:error equation}) by $v$, adding them  together and integrating in $\Omega$, we obtain that
\begin{align*}
&\underbrace{-\varepsilon^2\int_{r_0}^{1}\int_0^{2\pi}\Big[\big(ru_{rr}
+\frac{u_{\theta\theta}}{r}+2\frac{v_{\theta}}{r}+u_r-\frac{u}{r}\big)u+ \big(rv_{rr}+\frac{v_{\theta\theta}}{r}-2\frac{u_{\theta}}{r}+v_r-\frac{v}{r}\big)v\Big]d\theta dr}_{diffusion \quad term}\\
&+\underbrace{\int_{r_0}^{1}\int_0^{2\pi}(p_\theta u+p_r rv)d\theta dr}_{pressure \quad term}+\underbrace{\int_{r_0}^{1}\int_0^{2\pi}\big(u S_u +vS_v \big)d\theta dr}_{linear \quad term}\\
=&\int_{r_0}^{1}\int_0^{2\pi}\big[u F_u +vF_v \big]d\theta dr.
\end{align*}

Next we deal with them term by term. \\
{\bf The diffusion term:} Integrating by parts, we obtain
\begin{align}\label{e:diffusion estimate in basic energy estimate}
&-\varepsilon^2\int_{r_0}^{1}\int_0^{2\pi}\Big[\big(ru_{rr}
+\frac{u_{\theta\theta}}{r}+2\frac{v_{\theta}}{r}+u_r-\frac{u}{r}\big)u+ \big(rv_{rr}+\frac{v_{\theta\theta}}{r}-2\frac{u_{\theta}}{r}+v_r-\frac{v}{r}\big)v\Big]d\theta dr\nonumber\\
=&\varepsilon^2\int_{r_0}^{1}\int_0^{2\pi}\bigg(\frac{ u_\theta^2+v_\theta^2}{r}+\frac{ u^2+v^2}{r}+2\frac{u_\theta v-uv_\theta}{r}\bigg)d\theta dr
+\varepsilon^2\int_{r_0}^{1}\int_0^{2\pi}r\big(u_r^2+v_r^2\big)d\theta dr\nonumber\\
= &\varepsilon^2\int_{r_0}^{1}\int_0^{2\pi}\frac{ (u_\theta+v)^2+(v_\theta-u)^2}{r}d\theta dr+\varepsilon^2\int_{r_0}^{1}\int_0^{2\pi}r\big(u_r^2+v_r^2\big)d\theta dr\geq r_0\varepsilon^2\|u_r\|_2^2.
\end{align}
{\bf Pressure term:} Integrating by parts and using the divergence-free condition, we deduce that
\begin{align}\label{e:pressure estimate in basic energy estimate}
\int_{r_0}^{1}\int_0^{2\pi}\big(p_\theta u+p_r (rv)\big)d\theta dr=0.
\end{align}
{\bf Linear term:} Integrating by parts and using the divergence-free condition, we obtain
\begin{align}
&\int_{r_0}^{1}\int_0^{2\pi}\big(uS_u +vS_v \big)d\theta dr\nonumber\\
=&\int_{r_0}^{1}\int_0^{2\pi}\big[u^au_\theta u+rv^au_ru+uu^a_\theta u+vu^a_rru+(v^au+vu^a)u\big] d\theta dr\nonumber\\
&+\int_{r_0}^{1}\int_0^{2\pi}\big[u^av_\theta v+rv^av_rv+uv^a_\theta v+vv^a_r rv-2uu^av\big] d\theta dr\nonumber\\
=&\int_{r_0}^{1}\int_0^{2\pi}\big[u^2u^a_\theta +v^2 r v^a_r+v^au^2+uvv_\theta^a +uv(ru^a_r-u^a)\big] d\theta dr,\label{energy estimate linear term}
\end{align}
where we used
$$\int_{r_0}^{1}\int_0^{2\pi}(u^au_\theta u+rv^au_ru)d\theta dr=\int_{r_0}^{1}\int_0^{2\pi}(u^av_\theta v+rv^av_rv)d\theta dr=0.$$

Next, we divide into four steps to estimate \eqref{energy estimate linear term} term by term.

{\bf 1)The term $u^2u^a_\theta$:} We divide it into the Euler part and Prandtl part. Notice that $\partial_\theta u_e=0$,
integrating by parts and using (\ref{estimate on Euler parts}), we obtain
\beno
&&\Big|\int_{r_0}^{1}\int_0^{2\pi} u^2\partial_\theta u^a_e d\theta dr\Big|=\Big|\int_{r_0}^{1}\int_0^{2\pi} u^2\partial_\theta (u^a_e-u_e)d\theta dr\Big|\\
&=&\Big|\int_{r_0}^{1}\int_0^{2\pi} 2uu_\theta(u^a_e-u_e)d\theta dr\Big|\leq C \varepsilon(\eta+\varepsilon)\|u_r\|_2\|u_\theta\|_2.
\eeno
Moreover, integrating by parts, using hardy inequality and (\ref{estimate on Prandtl parts}), we have
\begin{align*}
&\Big|\int_{r_0}^{1}\int_0^{2\pi} u^2\partial_\theta u^a_p d\theta dr\Big|=\Big|\int_{r_0}^{1}\int_0^{2\pi} 2uu_\theta u^a_p d\theta dr\Big|\\
=&\varepsilon\Big|\int_{r_0}^{1}\int_0^{2\pi} \frac{2u}{r-1}u_\theta Yu^a_p d\theta dr\Big|\leq C\varepsilon(\eta+\varepsilon)\|u_r\|_2\|u_\theta\|_2.
\end{align*}
Similarly, there holds
\begin{align*}
&\Big|\int_{r_0}^{1}\int_0^{2\pi} u^2\partial_\theta \hat{u}^a_p d\theta dr\Big|=\Big|\int_{r_0}^{1}\int_0^{2\pi} 2uu_\theta \hat{u}^a_p d\theta dr\Big|\\
=&\varepsilon\Big|\int_{r_0}^{1}\int_0^{2\pi} \frac{2u}{r-r_0}u_\theta Zu^a_p d\theta dr\Big|\leq C\varepsilon(\eta+\varepsilon)\|u_r\|_2\|u_\theta\|_2.
\end{align*}

Thus, we obtain
\beno
\Big|\int_{r_0}^{1}\int_0^{2\pi} u^2u_\theta^a d\theta dr\Big|\leq  C\varepsilon(\eta+\varepsilon)\|u_r\|_2\|u_\theta\|_2.
\eeno

{\bf 2)The term $v^2rv^a_r$ and $uvv^a_\theta$:} Combining (\ref{estimate on Euler parts}) and (\ref{estimate on Prandtl parts}), we get
\begin{align*}
\|v^a_r\|_\infty \leq C(\eta+\varepsilon).
\end{align*}
Thus
\beno
\Big|\int_{r_0}^{1}\int_0^{2\pi}v^2 r v^a_r d\theta dr\Big|=\Big|\int_{r_0}^{1}\int_0^{2\pi}(rv)^2 \frac{v^a_r}{r} d\theta dr\Big|\leq C(\eta+\varepsilon)
\int_{r_0}^{1}\int_0^{2\pi}u_\theta^2 d\theta dr.
\eeno
The same argument gives that
\beno
\Big|\int_{r_0}^{1}\int_0^{2\pi}v_\theta^a uv d\theta dr\Big|\leq C\varepsilon(\eta+\varepsilon) \|v_\theta\|_2\|u_r\|_2.
\eeno

{\bf 3)The term $v^au^2$:}
We decompose $u(\theta,r)=u_0(r)+\tilde{u}(\theta,r)$,  where $u_0(r)=\frac{1}{2\pi}\int_0^{2\pi}u(r,\theta)d\theta$, and notice that
$\int_0^{2\pi}v^a(r,\theta)d\theta=0, \ \forall r\in [r_0,1]$,  we obtain
\beno
\Big|\int_{r_0}^{1}\int_0^{2\pi}v^au^2 d\theta dr \Big|=\Big|\int_{r_0}^{1}\int_0^{2\pi}v^a(2u_0\tilde{u}+\tilde{u}^2)d\theta dr \Big|\leq
C\varepsilon(\eta+\varepsilon) (\|u_\theta\|_2\|u_r\|_2+ \|u_\theta\|_2^2),
\eeno
where we used $\|\tilde{u}\|_2\leq c\|u_\theta\|_2$.

{\bf 4)The term $uv(ru^a_r-u^a)$:} Finally, we deal with the remainder two terms which involve $u,v$ together
\beno
\int_{r_0}^{1}\int_0^{2\pi}\big[uvru^a_r-uvu^a\big] d\theta dr.
\eeno

We first consider the leading Euler flow $u_e(r)$. We decompose $u(\theta,r)=u_0(r)+\tilde{u}(\theta,r)$ as above, then
\begin{align*}
\Big|\int_{r_0}^{1}\int_0^{2\pi}uvru'_e d\theta dr\Big|=&\Big|\int_{r_0}^{1}\int_0^{2\pi}(u_0+\tilde{u})vru'_e d\theta dr\Big|\\
=&\Big|\int_{r_0}^{1}\int_0^{2\pi}\tilde{u}vru'_e d\theta dr\Big|\leq C \|u_\theta\|_2 \|v_\theta\|_2,
\end{align*}
where we used
$$\int_0^{2\pi}v(\theta,r)d\theta=0, \ \forall r\in [r_0, 1].$$

Moreover, by (\ref{estimate on Euler parts}), it's easy to obtain that
\beno
\Big|\int_{r_0}^{1}\int_0^{2\pi}uvr\partial_r(u_e^a-u_e)d\theta dr\Big|\leq C\varepsilon(\eta+\varepsilon) \|v_\theta\|_2\|u_r\|_2.
\eeno
Thus, we obtain that
\beno
\Big|\int_{r_0}^{1}\int_0^{2\pi}uvr\partial_ru_e^ad\theta dr\Big|\leq C\varepsilon(\eta+\varepsilon)\|v_\theta\|_2\|u_r\|_2+C \|u_\theta\|_2 \|v_\theta\|_2.
\eeno
The same argument gives
\beno
\Big|\int_{r_0}^{1}\int_0^{2\pi}v uu_e^a d\theta dr\Big|\leq C\varepsilon(\eta+\varepsilon) \|v_\theta\|_2\|u_r\|_2+C \|u_\theta\|_2 \|v_\theta\|_2.
\eeno

For the Prandtl part, using the Hardy inequality, divergence-free condition and (\ref{estimate on Prandtl parts}), we deduce that
\begin{align*}
\Big|\int_{r_0}^{1}\int_0^{2\pi}uvr\partial_ru_p^a d\theta dr\Big|=&\Big|\int_{r_0}^{1}\int_0^{2\pi}\frac{urv}{\varepsilon}\partial_Yu_p^a d\theta dr\Big|\\
=&\varepsilon\Big|\int_{r_0}^{1}\int_0^{2\pi}\frac{rv u}{(r-1)^2}Y^2\partial_Yu_p^a d\theta dr\Big|\\
\leq& C\varepsilon(\eta+\varepsilon)\Big(\int_{r_0}^{1}\int_0^{2\pi}\frac{u^2}{(r-1)^2}d\theta dr\Big)^{\frac12}\Big(\int_{r_0}^{1}\int_0^{2\pi}\frac{(rv)^2}{(r-1)^2}d\theta dr\Big)^{\frac12}\\
\leq& C\varepsilon(\eta+\varepsilon) \|u_\theta\|_2 \|u_r\|_2.
\end{align*}
The same argument gives
\beno
\Big|\int_{r_0}^{1}\int_0^{2\pi}uvr\partial_r\hat{u}_p^a d\theta dr\Big|\leq C\varepsilon(\eta+\varepsilon) \|u_\theta\|_2 \|u_r\|_2.
\eeno

Recalling that $\tilde{u}_p^a=\chi(r)u_p^a+(1-\chi(r))\hat{u}^a_p$ and the fast decay of $(u_p^a, \hat{u}^a_p)$ away from the boundary, we deduce that
\beno
\Big|\int_{r_0}^{1}\int_0^{2\pi}uvr\partial_r
\tilde{u}_p^a d\theta dr\Big|\leq \varepsilon(\eta+\varepsilon) \|u_\theta\|_2 \|u_r\|_2.
\eeno

Moreover, by the same argument, we can obtain
\beno
\Big|\int_{r_0}^{1}\int_0^{2\pi}uvu_p^a d\theta dr\Big|\leq C\varepsilon(\eta+\varepsilon) \|v_\theta\|_2 \|u_r\|_2,\\[5pt]
\Big|\int_{r_0}^{1}\int_0^{2\pi}uv\hat{u}_p^a d\theta dr\Big|\leq C\varepsilon(\eta+\varepsilon) \|v_\theta\|_2 \|u_r\|_2.
\eeno

Summing these estimates together, we obtain
\beno
\Big|\int_{r_0}^{1}\int_0^{2\pi}\big[uvru^a_r-v u u^a \big] d\theta dr\Big| \leq C\varepsilon(\eta+\varepsilon) (\|u_\theta\|_2+\|v_\theta\|_2) \|u_r\|_2+C \|u_\theta\|_2 \|v_\theta\|_2.
\eeno

Finally, summing the above four steps we obtain
\begin{align}\label{e:linear term in energy estimate}
\Big|\int_{r_0}^{1}\int_0^{2\pi}\big(uS_u +vS_v \big)d\theta dr\Big|&\leq C\varepsilon(\eta+\varepsilon) \|(u_\theta,v_\theta)\|_2\|u_r\|_2+C \|u_\theta\|_2 \|v_\theta\|_2+C\varepsilon(\eta+\varepsilon) \|u_\theta\|_2^2\nonumber\\&
\leq \frac{r_0\varepsilon^2}{4} \|u_r\|_2+C  \|(u_\theta,v_\theta)\|_2.
\end{align}

Based on the estimates (\ref{e:diffusion estimate in basic energy estimate}), (\ref{e:pressure estimate in basic energy estimate}) and (\ref{e:linear term in energy estimate}), we obtain (\ref{e:basic Energy estimate}).
\end{proof}

\subsubsection{Positivity estimate}
\indent

In this subsection, we establish the following positivity estimate.
\begin{lemma}
Let  $(u,v)$ be a smooth solution of (\ref{linear error equation}), then there exist $\varepsilon_0>0, \eta_0>0$ such that for any $\varepsilon\in(0,\varepsilon_0), \eta\in (0,\eta_0)$, there holds
\begin{align}\label{e:positivity estimate}
\|(u_\theta,v_\theta)\|^2_2\leq C\varepsilon^2(\eta+\e)\|u_r\|_2^2
+C\|(F_1, F_2)\|_2^2.
\end{align}
\end{lemma}
\begin{proof}
Multipling the first equation by $-\big(\frac{r^2v}{u^a}\big)_r$ and the second equation by $\big(\frac{r v}{u^a}\big)_\theta$, integrating in $\Omega$ and summing two terms together, we arrive at
\begin{align}\label{p:positivity estimate}
&-\int_{r_0}^{1}\int_0^{2\pi}\Big[-\varepsilon^2\big(r u_{rr}
+\frac{u_{\theta\theta}}{r}+2\frac{v_{\theta}}{r}+u_r-\frac{u}{r}\big)
+p_\theta+S_u\Big]\Big(\frac{r^2v}{u^a}\Big)_rd\theta dr\nonumber\\
&+\int_{r_0}^{1}\int_0^{2\pi}\Big[-\varepsilon^2 \big(rv_{rr}+\frac{v_{\theta\theta}}{r}-2\frac{u_{\theta}}{r}+v_r
-\frac{v}{r}\big)+rp_r+S_v\Big]\Big(\frac{rv}{u^a}\Big)_\theta d\theta dr\nonumber\\
=&\int_{r_0}^{1}\int_0^{2\pi}\Big[-F_1\Big(\frac{r^2v}{u^a}\Big)_r+F_2\Big(\frac{rv}{u^a}\Big)_\theta\Big]d\theta dr.
\end{align}
Firstly, we choose $\varepsilon_0>0, \eta_0>0$ such that (\ref{positive lower bound}) holds.

{\bf Positivity term:} We first deal with these terms which are related to $S_u, S_v$.
Recall that
\begin{align*}
S_u:&=u^au_\theta+v^aru_r+uu^a_\theta+vru^a_r+v^au+vu^a,\\[5pt]
S_v:&=u^av_\theta+v^arv_r+uv^a_\theta+vrv^a_r-2uu^a,
\end{align*}
we first handle the following terms
\beno
&&-\int_{r_0}^{1}\int_0^{2\pi}\big(u^au_\theta+vru^a_r+vu^a\big)\Big(\frac{r^2v}{u^a}\Big)_rd\theta dr\\
&&\quad \quad +\int_{r_0}^{1}\int_0^{2\pi}\big(u^av_\theta-2uu^a\big)\Big(\frac{rv}{u^a}\Big)_\theta d\theta dr:=I_1+I_2.
\eeno
By (\ref{estimate on Euler parts}),(\ref{estimate on Prandtl parts}) and (\ref{positive lower bound}), we deduce that
\begin{align*}
I_2:=&\int_{r_0}^{1}\int_0^{2\pi}\Big(rv_\theta^2-2ru v_\theta-\frac{vv_\theta r u_\theta^a}{u^a}+\frac{2uvr  u_\theta^a}{u^a}\Big)d\theta dr\\
=&\int_{r_0}^{1}\int_0^{2\pi}\Big(rv_\theta^2-\frac{vv_\theta r u_\theta^a}{u^a}+\frac{2uv r u_\theta^a}{u^a}\Big)d\theta dr\\
\geq& (1-C\eta-C\varepsilon)\int_{r_0}^{1}\int_0^{2\pi}rv_\theta^2d\theta dr+\int_{r_0}^{1}\int_0^{2\pi}\frac{2uvr  u_\theta^a}{u^a}d\theta dr,
\end{align*}
where we used
\beno
\int_{r_0}^{1}\int_0^{2\pi}-2ru v_\theta drd\theta=\int_{r_0}^{1}\int_0^{2\pi}2u_\theta rv drd\theta=-\int_{r_0}^{1}\int_0^{2\pi}2(rv)_r rv drd\theta=0.
\eeno

Moreover, by the Hardy inequality, (\ref{estimate on Euler parts}), (\ref{estimate on Prandtl parts}) and (\ref{positive lower bound}), there holds
\begin{align*}
&\Big|\int_{r_0}^{1}\int_0^{2\pi}\frac{2uvr  u_\theta^a}{u^a}d\theta dr\Big|\\[5pt]
\leq&\varepsilon\Big|\int_{r_0}^{1}\int_0^{2\pi}\frac{2uvr Y\partial_\theta u_p^a}{(r-1)u^a}d\theta dr\Big|+\varepsilon\Big|\int_{r_0}^{1}\int_0^{2\pi}\frac{2uvr Z\partial_\theta \hat{u}_p^a}{(r-r_0)u^a}d\theta dr\Big|+\Big|\int_{r_0}^{1}\int_0^{2\pi}\frac{2uvr \partial_\theta u_e^a}{u^a}d\theta dr\Big|\\
\leq& C\varepsilon(\eta+\varepsilon) \|u_r\|_2 \|rv_\theta\|_2+\Big|\int_{r_0}^{1}\int_0^{2\pi}\frac{2uvr \partial_\theta u_e^a}{u^a}d\theta dr\Big|\\
\leq& C(\eta+\varepsilon) \int_{r_0}^{1}\int_0^{2\pi}rv_\theta^2d\theta dr+C\varepsilon^2(\eta+\varepsilon)\|u_r\|_2^2+\Big|\int_{r_0}^{1}\int_0^{2\pi}\frac{2uv \partial_\theta u_e^a}{u^a}d\theta dr\Big|.
\end{align*}
By (\ref{estimate on Euler parts}), it is easy to get
\beno
\Big|\int_{r_0}^{1}\int_0^{2\pi}\frac{2uv \partial_\theta u_e^a}{u^a}d\theta dr\Big|\leq C\varepsilon(\eta+\varepsilon)\|u_r\|_2\|v_\theta\|_2.
\eeno
Thus,
\begin{align*}
I_2\geq \big(1-C\eta-C\varepsilon\big) \int_{r_0}^{1}\int_0^{2\pi}rv_\theta^2d\theta dr-C\varepsilon^2(\eta+\varepsilon)\|u_r\|_2^2.
\end{align*}

For $I_1$, we first decompose it
\begin{align*}
I_1=&-\int_{r_0}^{1}\int_0^{2\pi}\Big[u^au_\theta+rvu^a_r+vu^a\Big]\Big(-\frac{r}{u^a}u_\theta+\Big(\frac{r}{u^a}\Big)_rrv\Big)d\theta dr\\
=&\int_{r_0}^{1}\int_0^{2\pi}ru^2_\theta d\theta dr+\int_{r_0}^{1}\int_0^{2\pi}r^2v\frac{u^a_r}{u^a}u_\theta d\theta dr+\int_{r_0}^{1}\int_0^{2\pi}rvu_\theta d\theta dr\\
&-\int_{r_0}^{1}\int_0^{2\pi}u^a\Big(\frac{r}{u^a}\Big)_r u_\theta rv d\theta dr-\int_{r_0}^{1}\int_0^{2\pi}u^a_r\Big(\frac{r}{u^a}\Big)_r (rv)^2 d\theta dr\\
&-\int_{r_0}^{1}\int_0^{2\pi}\frac{u^a}{r}\Big(\frac{r}{u^a}\Big)_r (rv)^2 d\theta dr
:=\sum_{i=1}^6 I_{1i}.
\end{align*}
Then, integrating by parts and using the divergence-free condition, we deduce that
\begin{align*}
I_{12}=&\int_{r_0}^{1}\int_0^{2\pi}r^2v\frac{u^a_r}{u^a}u_\theta d\theta dr\\
=&-\int_{r_0}^{1}\int_0^{2\pi}r^2v\frac{u^a_r}{u^a}(rv)_rd\theta dr=\frac12\int_{r_0}^{1}\int_0^{2\pi}\Big(\frac{ru^a_r}{u^a}\Big)_r(rv)^2 d\theta dr,\\
I_{13}=&\int_{r_0}^{1}\int_0^{2\pi}rvu_\theta d\theta dr=-\int_{r_0}^{1}\int_0^{2\pi}rv(rv)_rd\theta dr=0,\\
I_{14}=&-\int_{r_0}^{1}\int_0^{2\pi}u^a\Big(\frac{r}{u^a}\Big)_r u_\theta rv d\theta dr\\
=&\int_{r_0}^{1}\int_0^{2\pi}u^a\Big(\frac{r}{u^a}\Big)_r (rv)_r rv d\theta dr=-\frac12\int_{r_0}^{1}\int_0^{2\pi}\Big[u^a\Big(\frac{r}{u^a}\Big)_r\Big]_r  (rv)^2 d\theta dr.
\end{align*}

Finally, summing the terms $I_{11},\cdots, I_{16}$, we obtain
\begin{align*}
I_1=&\int_{r_0}^{1}\int_0^{2\pi}ru^2_\theta d\theta dr\\
&+\int_{r_0}^{1}\int_0^{2\pi}\Big[\frac12\Big(\frac{ru^a_r}{u^a}\Big)_r
-\frac12\Big[u^a\Big(\frac{r}{u^a}\Big)_r\Big]_r-u^a_r\Big(\frac{r}{u^a}\Big)_r-\frac{u^a}{r}\Big(\frac{r}{u^a}\Big)_r\Big](rv)^2d\theta dr.
\end{align*}
Direct computation gives
\begin{align}
&\frac12\Big(\frac{ru^a_r}{u^a}\Big)_r
-\frac12\Big[u^a\Big(\frac{r}{u^a}\Big)_r\Big]_r-u^a_r\Big(\frac{r}{u^a}\Big)_r-\frac{u^a}{r}\Big(\frac{r}{u^a}\Big)_r\nonumber\\[5pt]
=&\Big(\frac{ru^a_r}{u^a}\Big)_r-\Big(u^a_r+\frac{u^a}{r}\Big)\Big(\frac{r}{u^a}\Big)_r\nonumber\\[5pt]
=&\frac{u^a(u^a_r+ru^a_{rr})-r(u^a_r)^2}{(u^a)^2}-\frac{r}{(u^a)^2}\Big[\Big(\frac{u^a}{r}\Big)^2-(u^a_r)^2\Big]\nonumber\\
=&\frac{1}{u^a}\Big(ru^a_{rr}+u^a_r-\frac{u^a}{r}\Big).\nonumber
\end{align}
Notice that $u^a=u^a_e+\tilde u_p^a$, there holds
\begin{align*}
I_1=&\int_{r_0}^{1}\int_0^{2\pi}ru^2_\theta d\theta dr+\int_{r_0}^{1}\int_0^{2\pi}\frac{1}{u^a}\Big(ru^a_{rr}+u^a_r-\frac{u^a}{r}\Big)(rv)^2d\theta dr
\\=&\int_{r_0}^{1}\int_0^{2\pi}ru^2_\theta d\theta dr+\int_{r_0}^{1}\int_0^{2\pi}\frac{1}{u^a}\Big(\partial_r u^a_e+r\partial{rr}u^a_e-\frac{u^a_e}{r}\Big)(rv)^2d\theta dr\\
&+\int_{r_0}^{1}\int_0^{2\pi}\frac{1}{u^a}\Big(\partial_r\tilde{u}^a_p+r\partial{rr}\tilde{u}^a_p-\frac{\tilde{u}_p^a}{r}\Big)(rv)^2d\theta dr
\\=&\int_{r_0}^{1}\int_0^{2\pi}ru^2_\theta d\theta dr+\int_{r_0}^{1}\int_0^{2\pi}\frac{1}{u_e}\Big(\partial_r u^a_e+r\partial{rr}u^a_e-\frac{u^a_e}{r}\Big)(rv)^2d\theta dr\\
&+\int_{r_0}^{1}\int_0^{2\pi}\Big(\frac{1}{u^a}-\frac{1}{u_e}\Big)\Big(\partial_r u^a_e+r\partial{rr}u^a_e-\frac{u^a_e}{r}\Big)(rv)^2d\theta dr\\
&+\int_{r_0}^{1}\int_0^{2\pi}\frac{1}{u^a}\Big(\partial_r\tilde{u}^a_p+r\partial{rr}\tilde{u}^a_p-\frac{\tilde{u}_p^a}{r}\Big)(rv)^2d\theta dr.
\end{align*}

By (\ref{estimate on Euler parts}), (\ref{estimate on Prandtl parts}) and (\ref{positive lower bound}), we deduce that $|\frac{1}{u^a}-\frac{1}{u_e}|\leq C(\eta+\varepsilon)$, hence
\beno
\Big|\int_{r_0}^{1}\int_0^{2\pi}\Big(\frac{1}{u^a}-\frac{1}{u_e}\Big)\Big(\partial_r u^a_e+r\partial{rr}u^a_e-\frac{u^a_e}{r}\Big)(rv)^2d\theta dr\Big|\leq C(\eta+\varepsilon) \|rv\|_2^2\leq C(\eta+\varepsilon) \|u_\theta\|_2^2.
\eeno

Moreover, by the Hardy inequality, (\ref{estimate on Prandtl parts}) and (\ref{positive lower bound}), we deduce that
\begin{align*}
&\Big|\int_{r_0}^{1}\int_0^{2\pi}\Big[\frac{1}{u^a}\Big(\partial_ru^a_p+r\partial{rr}u^a_p-\frac{u_p^a}{r}\Big)(rv)^2d\theta dr\Big|\\
=&\Big|\int_{r_0}^{1}\int_0^{2\pi}\Big[\frac{(r-1)^2}{u^a}\Big(\partial_ru^a_p+r\partial_{rr}u^a_p
-\frac{u_p^a}{r}\Big)\frac{(rv)^2}{(r-1)^2}d\theta dr\Big|\\
=&\Big|\int_{r_0}^{1}\int_0^{2\pi}\Big[\frac{Y^2}{u^a}\Big(\varepsilon\partial_Yu^a_p+r\partial_{YY}u^a_p
-\varepsilon^2\frac{u_p^a}{r}\Big)\frac{(rv)^2}{(r-1)^2}d\theta dr\Big|\leq C(\eta+\varepsilon)\|u_\theta\|_2^2.
\end{align*}
Similarly, there holds
\begin{align*}
&\Big|\int_{r_0}^{1}\int_0^{2\pi}\Big[\frac{1}{u^a}\Big(\partial_r\hat{u}^a_p+r\partial{rr}\hat{u}^a_p-\frac{\hat{u}_p^a}{r}\Big)(rv)^2d\theta dr\Big|\\
=&\Big|\int_{r_0}^{1}\int_0^{2\pi}\Big[\frac{(r-r_0)^2}{u^a}\Big(\partial_r\hat{u}^a_p+r\partial_{rr}\hat{u}^a_p
-\frac{\hat{u}_p^a}{r}\Big)\frac{(rv)^2}{(r-r_0)^2}d\theta dr\Big|\leq C(\eta+\varepsilon) \|u_\theta\|_2^2.
\end{align*}
Recall that $\tilde{u}_p^a=\chi(r)u_p^a+(1-\chi(r))\hat{u}^a_p$ and the fast decay of $(u_p^a, \hat{u}^a_p)$ away from the boundary, we deduce that
\beno
\Big|\int_{r_0}^{1}\int_0^{2\pi}\Big[\frac{1}{u^a}\Big(\partial_r\tilde{u}^a_p+r\partial{rr}\tilde{u}^a_p-\frac{\tilde{u}_p^a}{r}\Big)(rv)^2d\theta dr\Big|\leq C(\eta+\varepsilon) \|u_\theta\|_2^2.
\eeno
Thus, we deduce that
\begin{align*}
I_1\geq &\int_{r_0}^{1}\int_0^{2\pi}ru^2_\theta d\theta dr+\int_{r_0}^{1}\int_0^{2\pi}\frac{1}{u_e}\Big(\partial_r u^a_e+r\partial{rr}u^a_e-\frac{u^a_e}{r}\Big)(rv)^2d\theta dr-C(\eta+\varepsilon) \|u_\theta\|_2^2\\
\geq & \int_{r_0}^{1}\int_0^{2\pi}ru^2_\theta d\theta dr+\int_{r_0}^{1}\int_0^{2\pi}\frac{1}{u_e}\Big(\partial_r u_e+r\partial{rr}u_e-\frac{u_e}{r}\Big)(rv)^2d\theta dr-C(\eta+\e) \|u_\theta\|_2^2,
\end{align*}
where we used the fact $u_e^a(\theta,r):=u_e(r)+\sum_{i=1}^3\varepsilon^i \tilde{u}_e^{(i)}(\theta,r)$.

Finally, summing the estimates for $I_1$ and $I_2$, we obtain
\begin{align*}
-&\int_{r_0}^{1}\int_0^{2\pi}\big(u^au_\theta+vru^a_r+vu^a\big)\Big(\frac{r^2v}{u^a}\Big)_rd\theta dr
+\int_{r_0}^{1}\int_0^{2\pi}\big(u^av_\theta-2uu^a\big)\Big(\frac{rv}{u^a}\Big)_\theta d\theta dr\\
&\geq  \Big(1-C\eta-C\e\Big)\int_{r_0}^{1}\int_0^{2\pi}r(u^2_\theta+v^2_\theta) drd\theta\\
&\quad +\int_{r_0}^{1}\int_0^{2\pi}\frac{1}{u_e}\Big(\partial_r u_e+r\partial{rr}u_e-\frac{u_e}{r}\Big)(rv)^2d\theta dr
-C\varepsilon^2(\eta+\varepsilon)\|u_r\|_2^2.
\end{align*}

Next, we deal with the remainder terms in $S_u$
\begin{align*}
&-\int_{r_0}^{1}\int_0^{2\pi}\Big[v^aru_r+uu^a_\theta+v^au\Big]\Big(\frac{r^2v}{u^a}\Big)_rd\theta dr\\
=&-\int_{r_0}^{1}\int_0^{2\pi}\Big[rv^au_r+uu^a_\theta+ v^au\Big]\Big(\Big(\frac{r}{u^a}\Big)_r rv-\frac{r}{u^a}u_\theta\Big)d\theta dr\\
=&-\int_{r_0}^{1}\int_0^{2\pi}rv^au_r\Big(\Big(\frac{r}{u^a}\Big)_r rv-\frac{r}{u^a}u_\theta\Big)d\theta dr
-\int_{r_0}^{1}\int_0^{2\pi}uu^a_\theta\Big(\Big(\frac{r}{u^a}\Big)_r rv-\frac{r}{u^a}u_\theta\Big)d\theta dr\\
&-\int_{r_0}^{1}\int_0^{2\pi} v^au\Big(\Big(\frac{r}{u^a}\Big)_r rv-\frac{r}{u^a}u_\theta\Big)d\theta dr:=I^u_1+I^u_2+I^u_3
\end{align*}
and the remainder terms in $S_v$
\beno
&&\int_{r_0}^{1}\int_0^{2\pi}\Big[v^arv_r+uv^a_\theta+vrv^a_r\Big]\Big(\frac{r v}{u^a}\Big)_\theta d\theta dr\\
&=&\int_{r_0}^{1}\int_0^{2\pi}\Big[rv^av_r+uv^a_\theta+rv_r^av\Big]\Big(\frac{r v_\theta}{u^a}-\frac{rv  u_\theta^a}{(u^a)^2}\Big)d\theta dr\\
&=&\int_{r_0}^{1}\int_0^{2\pi}rv^av_r\Big(\frac{r v_\theta}{u^a}-\frac{rv  u_\theta^a}{(u^a)^2}\Big)d\theta dr+
\int_{r_0}^{1}\int_0^{2\pi}uv^a_\theta\Big(\frac{r v_\theta}{u^a}-\frac{rv  u_\theta^a}{(u^a)^2}\Big)d\theta dr\\
&&+\int_{r_0}^{1}\int_0^{2\pi}rv_r^av\Big(\frac{r v_\theta}{u^a}-\frac{rv  u_\theta^a}{(u^a)^2}\Big)d\theta dr:=I^v_1+I^v_2+I^v_3.
\eeno
{\bf The remainder terms in $S_u$:}  By (\ref{estimate on Euler parts}), (\ref{estimate on Prandtl parts}) and (\ref{positive lower bound}), we have
\begin{align*}
\Big|\int_{r_0}^{1}\int_0^{2\pi}rv^au_r\frac{r u_\theta}{u^a} d\theta dr\Big|\leq& C\varepsilon(\eta+\varepsilon)\|u_r\|_2 \|u_\theta\|_2, \\
\Big|\int_{r_0}^{1}\int_0^{2\pi}rv^au_r\Big(\frac{r}{u^a}\Big)_r rvd\theta dr\Big|=&\Big|\int_{r_0}^{1}\int_0^{2\pi}\Big(\frac{r^2vv^au_r}{u^a}-\frac{r^2vv^au_r ru^a_r}{(u^a)^2}\Big)d\theta dr\Big|\\
\leq&C\varepsilon(\eta+\varepsilon)\|u_r\|_2 \|u_\theta\|_2+\Big|\int_{r_0}^{1}\int_0^{2\pi}\frac{r^2vv^au_r ru^a_r}{(u^a)^2}d\theta dr\Big|.
\end{align*}
Moreover, by the Hardy inequality, (\ref{estimate on Euler parts}), (\ref{estimate on Prandtl parts}) and (\ref{positive lower bound}), we deduce
\begin{align*}
\Big|\int_{r_0}^{1}\int_0^{2\pi}\frac{r^2vv^au_r r\partial_ru^a_p}{(u^a)^2}d\theta dr\Big|=&\Big|\int_{r_0}^{1}\int_0^{2\pi}\frac{r^2vv^au_r r(r-1)\partial_ru^a_p}{(r-1)(u^a)^2}d\theta dr\Big|\\
=&\Big|\int_{r_0}^{1}\int_0^{2\pi}\frac{r^2vv^au_r rY\partial_Yu^a_p}{(r-1)(u^a)^2}d\theta dr\Big|\leq C\varepsilon(\eta+\varepsilon)\|u_r\|_2 \|u_\theta\|_2,\\
\Big|\int_{r_0}^{1}\int_0^{2\pi}\frac{r^2vv^au_r r\partial_r\hat{u}^a_p}{(u^a)^2}d\theta dr\Big|=&\Big|\int_{r_0}^{1}\int_0^{2\pi}\frac{r^2vv^au_r r(r-r_0)\partial_r\hat{u}^a_p}{(r-r_0)(u^a)^2}d\theta dr\Big|\leq C\varepsilon(\eta+\varepsilon)\|u_r\|_2 \|u_\theta\|_2,\\
\Big|\int_{r_0}^{1}\int_0^{2\pi}\frac{r^2vv^au_r r\partial_ru^a_e}{(u^a)^2}d\theta dr\Big|\leq& C\varepsilon(\eta+\varepsilon)\|u_r\|_2 \|u_\theta\|_2,
\end{align*}
thus, there holds
\begin{align*}
\Big|\int_{r_0}^{1}\int_0^{2\pi}\frac{r^2vv^au_r r\partial_ru^a}{(u^a)^2}d\theta dr\Big|\leq C\varepsilon(\eta+\varepsilon)\|u_r\|_2 \|u_\theta\|_2.
\end{align*}

Finally, we obtain
\beno
|I^u_1|\leq (\eta+\varepsilon)\int_{r_0}^{1}\int_0^{2\pi}ru_\theta^2 drd\theta+ C\varepsilon^2 (\eta+\varepsilon)\int_{r_0}^{1}\int_0^{2\pi}u_r^2 drd\theta .
\eeno
The same argument gives
\beno
|I^u_3|\leq (\eta+\varepsilon)\int_{r_0}^{1}\int_0^{2\pi}ru_\theta^2 drd\theta+ C\varepsilon^2 (\eta+\varepsilon)\int_{r_0}^{1}\int_0^{2\pi}u_r^2 drd\theta .
\eeno

For $I^u_2$, the same argument as above gives
\beno
\Big|\int_{r_0}^{1}\int_0^{2\pi}uu^a_\theta\frac{r u_\theta}{u^a}d\theta dr\Big|\leq C\varepsilon(\eta+\varepsilon)\|u_r\|_2 \|u_\theta\|_2.
\eeno

Moreover, direct computation gives
\beno
\int_{r_0}^{1}\int_0^{2\pi}uu^a_\theta\Big(\frac{r}{u^a}\Big)_r rvd\theta dr=\int_{r_0}^{1}\int_0^{2\pi}
\Big(\frac{uu^a_\theta rv}{u^a}-\frac{uu^a_\theta rv \cdot ru^a_r}{(u^a)^2}\Big)d\theta dr.
\eeno
The same argument as above gives
\beno
\Big|\int_{r_0}^{1}\int_0^{2\pi}\frac{uu^a_\theta rv}{u^a}d\theta dr\Big|\leq C\varepsilon(\eta+\varepsilon)\|u_r\|_2 \|u_\theta\|_2.
\eeno
By the Hardy inequality, (\ref{estimate on Euler parts}), (\ref{estimate on Prandtl parts}) and (\ref{positive lower bound}), we deduce that
\begin{align*}
\Big|\int_{r_0}^{1}\int_0^{2\pi}
\frac{uu^a_\theta rv\cdot r\partial_ru^a_p}{(u^a)^2}d\theta dr\Big|=&\Big|\int_{r_0}^{1}\int_0^{2\pi}\frac{u}{r-1}u^a_\theta \frac{rv}{r-1} \frac{r(r-1)^2\partial_ru^a_p}{(u^a)^2}d\theta dr\Big|\\
=&\varepsilon\Big|\int_{r_0}^{1}\int_0^{2\pi}\frac{u}{r-1}u^a_\theta \frac{rv}{r-1} \frac{rY^2\partial_Yu^a_p}{(u^a)^2}d\theta dr\Big|\leq C\varepsilon(\eta+\varepsilon)\|u_r\|_2 \|u_\theta\|_2,\\
\Big|\int_{r_0}^{1}\int_0^{2\pi}
\frac{uu^a_\theta rv \cdot r\partial_r\hat{u}^a_p}{(u^a)^2}d\theta dr\Big|=&\Big|\int_{r_0}^{1}\int_0^{2\pi}\frac{u}{r-r_0}u^a_\theta \frac{rv}{r-r_0} \frac{r(r-r_0)^2\partial_r\hat{u}^a_p}{(u^a)^2}d\theta dr\Big|\\[5pt]
\leq& C\varepsilon(\eta+\varepsilon)\|u_r\|_2 \|u_\theta\|_2.
\end{align*}

Furthermore, using the above argument and noting the fact $\|\partial_\theta u_e^a\|_2\leq C\varepsilon (\eta+\varepsilon)$,  there holds
\beno
\Big|\int_{r_0}^{1}\int_0^{2\pi}
\frac{uu^a_\theta rv\cdot r\partial_ru^a_e}{(u^a)^2}d\theta dr\Big|
\leq C\varepsilon(\eta+\varepsilon)\|u_r\|_2 \|u_\theta\|_2.
\eeno

Finally, we obtain
\beno
|I^u_2|\leq (\eta+\varepsilon)\int_{r_0}^{1}\int_0^{2\pi}r(u_\theta^2+v^2_\theta)drd\theta+ C\varepsilon^2 (\eta+\varepsilon)\int_{r_0}^{1}\int_0^{2\pi}u_r^2 drd\theta .
\eeno

Thus, there holds
\begin{align}
&\Big|\int_{r_0}^{1}\int_0^{2\pi}\Big[v^aru_r+uu^a_\theta+v^au\Big]\Big(\frac{r^2v}{u^a}\Big)_rd\theta dr\Big|\nonumber\\
\leq& C(\eta+\varepsilon)\int_{r_0}^{1}\int_0^{2\pi}r(u_\theta^2+v^2_\theta)drd\theta+ C\varepsilon^2(\eta+\varepsilon)\int_{r_0}^{1}\int_0^{2\pi}u_r^2 drd\theta.\label{remainder terms in Su}
\end{align}

{\bf The remainder terms in $S_v$:}
Using $v^a=O(\varepsilon(\eta+\varepsilon))$, we deduce that
\beno
|I_1^v|\leq C\varepsilon (\eta+\varepsilon) \|u_\theta\|_2 \|v_\theta\|_2, \ |I_2^v|\leq C\varepsilon (\eta+\varepsilon) \|u_r\|_2 \|v_\theta\|_2.
\eeno

For $I^v_3$, using the Hardy inequality, (\ref{estimate on Euler parts}), (\ref{estimate on Prandtl parts}) and (\ref{positive lower bound}), we deduce
\begin{align*}
\Big|\int_{r_0}^{1}\int_0^{2\pi}r\partial_rv_e^av\frac{r v_\theta}{u^a}d\theta dr\Big|\leq& C\varepsilon (\eta+\varepsilon) \|u_\theta\|_2 \|v_\theta\|_2,\\
\Big|\int_{r_0}^{1}\int_0^{2\pi}r\partial_rv_p^av\frac{r v_\theta}{u^a}d\theta dr\Big|
=&\Big|\int_{r_0}^{1}\int_0^{2\pi}(r-1)\partial_rv_p^a\frac{rv}{r-1}\frac{r v_\theta}{u^a}d\theta dr\Big|\\
=&\Big|\int_{r_0}^{1}\int_0^{2\pi}Y\partial_Yv_p^a\frac{rv}{r-1}\frac{r v_\theta}{u^a}d\theta dr\Big|\leq C\varepsilon (\eta+\varepsilon) \|u_\theta\|_2 \|v_\theta\|_2,\\
\Big|\int_{r_0}^{1}\int_0^{2\pi}r\partial_r\hat{v}_p^av\frac{r v_\theta}{u^a}d\theta dr\Big|
=&\Big|\int_{r_0}^{1}\int_0^{2\pi}(r-r_0)\partial_rv_p^a\frac{rv}{r-r_0}\frac{r v_\theta}{u^a}d\theta dr\Big|\leq C\varepsilon (\eta+\varepsilon)\|u_\theta\|_2 \|v_\theta\|_2,
\end{align*}
hence there holds
\begin{align*}
\Big|\int_{r_0}^{1}\int_0^{2\pi}rv_r^av\frac{r v_\theta}{u^a}d\theta dr\Big|\leq C\varepsilon (\eta+\varepsilon) \|u_\theta\|_2\|v_\theta\|_2.
\end{align*}
The same argument gives
\beno
\Big|\int_{r_0}^{1}\int_0^{2\pi}rv_r^av\frac{rv  u_\theta^a}{(u^a)^2}d\theta dr\Big|\leq C\varepsilon(\eta+\varepsilon) \|u_\theta\|_2^2.
\eeno

Collecting these estimates, we get
\beno
|I_3^v|\leq C\varepsilon (\eta+\varepsilon) \|u_\theta\|_2(\|u_\theta\|_2+\|v_\theta\|_2).
\eeno
Thus, we obtain
 \begin{align}
&\Big|\int_{r_0}^{1}\int_0^{2\pi}\Big[v^arv_r+uv^a_\theta+vrv^a_r\Big]\Big(\frac{r v}{u^a}\Big)_\theta d\theta dr\Big|\nonumber\\
\leq&C(\eta+\varepsilon)\int_{r_0}^{1}\int_0^{2\pi}r(u_\theta^2+v^2_\theta)drd\theta+ C\varepsilon^2 (\eta+\varepsilon)\int_{r_0}^{1}\int_0^{2\pi}u_r^2 drd\theta.\label{remainder terms in Sv}
\end{align}

Finally, collecting \eqref{remainder terms in Su} and \eqref{remainder terms in Sv}, we can choose $\varepsilon_0>0, \eta_0>0$ such that for any $\varepsilon\in (0, \varepsilon_0), \eta\in (0,\eta_0)$, there holds
\begin{align}\label{e:linear term in positivity estimate}
&-\int_{r_0}^{1}\int_0^{2\pi}S_u \Big(\frac{r^2v}{u^a}\Big)_rd\theta dr+\int_{r_0}^{1}\int_0^{2\pi}S_v\Big(\frac{rv}{u^a}\Big)_\theta d\theta dr\nonumber\\[5pt]
\quad \quad \geq &\Big(1-C\eta-C\e\Big)\int_{r_0}^{1}\int_0^{2\pi}r(u^2_\theta+v^2_\theta)drd\theta\\\nonumber
&+\int_{r_0}^{1}\int_0^{2\pi}\frac{1}{u_e}\Big(\partial_r u_e+r\partial{rr}u_e-\frac{u_e}{r}\Big)(rv)^2d\theta dr-C\varepsilon^2(\eta+\varepsilon)\int_{r_0}^{1}\int_0^{2\pi}u_r^2 drd\theta.
\end{align}
{\bf Pressure estimate:} Integrating by parts, we deduce that
\begin{align}\label{e:pressure estimate in positivity estimate}
&-\int_{r_0}^{1}\int_0^{2\pi}p_\theta\Big(\frac{r^2v}{u^a}\Big)_rd\theta dr
+\int_{r_0}^{1}\int_0^{2\pi}rp_r\Big(\frac{rv}{u^a}\Big)_\theta d\theta dr\nonumber\\
=&\int_{r_0}^{1}\int_0^{2\pi}p_{r\theta}\frac{r^2v}{u^a}d\theta dr
-\int_{r_0}^{1}\int_0^{2\pi}p_{r\theta}\frac{r^2v}{u^a}d\theta dr=0.
\end{align}
{\bf Diffusion term:}
Finally, we deal with the diffusion term:
\begin{align}
&-\int_{r_0}^{1}\int_0^{2\pi}\Big[-\varepsilon^2\big(r u_{rr}
+\frac{u_{\theta\theta}}{r}+2\frac{v_{\theta}}{r}+u_r-\frac{u}{r}\big)
\Big]\Big(\frac{r^2v}{u^a}\Big)_rd\theta dr\nonumber\\
&+\int_{r_0}^{1}\int_0^{2\pi}\Big[-\varepsilon^2 \big(rv_{rr}+\frac{v_{\theta\theta}}{r}-2\frac{u_{\theta}}{r}+v_r
-\frac{v}{r}\big)\Big]\Big(\frac{rv}{u^a}\Big)_\theta d\theta dr\nonumber\\
=&\varepsilon^2\int_0^{2\pi}\int_{r_0}^1 \Big(\frac{1}{r} u_{\theta\theta}+ r u_{rr}+  u_{r}+\frac{2}{r} v_\theta- \frac{1}{r} u\Big)\frac{ (r^2v)_r}{u^a}drd\theta\nonumber\\
&-\varepsilon^2\int_0^{2\pi}\int_{r_0}^1 \Big(\frac{1}{r} u_{\theta\theta}+ r u_{rr}+  u_{r}+\frac{2}{r} v_\theta- \frac{1}{r} u\Big)\frac{r^2v u_r^a}{(u^a)^2}drd\theta\nonumber\\
&-\varepsilon^2\int_0^{2\pi}\int_{r_0}^1 \Big(\frac{1}{r} v_{\theta\theta}+ r v_{rr}+  v_{r}-\frac{2}{r} u_\theta- \frac{1}{r} v\Big)\frac{rv_\theta}{u^a}drd\theta\nonumber\\
&+\varepsilon^2\int_0^{2\pi}\int_{r_0}^1 \Big(\frac{1}{r} v_{\theta\theta}+ r v_{rr}+  v_{r}-\frac{2}{r} u_\theta- \frac{1}{r} v\Big)\frac{rv u_\theta^a}{(u^a)^2}drd\theta:=\sum_{i=1}^4I_i.\label{diffusion term}
\end{align}
We first compute $I_1$ term by term. By the divergence-free condition and (\ref{positive lower bound}), we have
\begin{align}
\Big|\int_0^{2\pi}\int_{r_0}^1 \frac{1}{r} u_{\theta\theta}\frac{(r^2v)_r}{u^a}drd\theta\Big|
=&\Big|\int_0^{2\pi}\int_{r_0}^1 \frac{1}{r} u_{\theta\theta}\frac{r(rv)_r+rv}{u^a}drd\theta\Big|\nonumber\\
=&\Big|\int_0^{2\pi}\int_{r_0}^1 u_{\theta\theta}\frac{-u_\theta+v}{u^a}drd\theta\Big|\nonumber\\
=&\Big|\int_0^{2\pi}\int_{r_0}^1\Big[u_\theta^2 \Big(\frac{1}{2u^a}\Big)_\theta-u_\theta\Big(\frac{v_\theta}{u^a}-\frac{vu^a_\theta}{(u^a)^2}\Big)\Big]drd\theta\Big|\nonumber\\
\leq& C\int_0^{2\pi}\int_{r_0}^1(u_\theta^2+v_\theta^2)drd\theta\nonumber
\end{align}
and
\begin{align}
\Big|\int_0^{2\pi}\int_{r_0}^1  ru_{rr}\frac{(r^2v)_r}{u^a}drd\theta\Big|=&\Big|\int_0^{2\pi}\int_{r_0}^1 ru_{rr}\frac{r(rv)_r+rv}{u^a}drd\theta\Big|\nonumber\\
=&\Big|\int_0^{2\pi}\int_{r_0}^1\frac{-r^2u_{rr}u_\theta+ r^2u_{rr}v}{u^a}drd\theta\Big|\nonumber\\
=&\Big|\int_0^{2\pi}\int_{r_0}^1\Big[u_{r}\Big(\frac{r^2u_\theta}{u^a}\Big)_r-u_{r}\Big(\frac{ r^2v}{u^a}\Big)_r\Big]drd\theta\Big|\nonumber\\
=&\Big|\int_0^{2\pi}\int_{r_0}^1\Big[u_{r}\Big(\frac{r^2u_{r\theta}+2ru_\theta}{u^a}-\frac{r^2u_\theta u^a_r}{(u^a)^2}\Big)\nonumber\\
&-u_{r}\Big(\frac{ rv+r(rv)_r}{u^a}-\frac{r^2v u^a_r}{(u^a)^2}\Big)\Big]drd\theta\Big|\nonumber\\
\leq &\frac{\eta}{ \varepsilon^2}\|\sqrt{r}u_\theta\|_2^2+C\eta\|u_r\|_2^2+\Big|\int_0^{2\pi}\int_{r_0}^1\frac{r^2u_{r\theta}u_{r}}{u^a}drd\theta\Big|\nonumber\\
\leq &\frac{\eta}{ \varepsilon^2}\|\sqrt{r}u_\theta\|_2^2+C\eta\|u_r\|_2^2.\nonumber
\end{align}

The other low order terms can be easily estimated as follows
\begin{align}
&\Big|\int_0^{2\pi}\int_{r_0}^1 \Big(u_{r}+\frac{2}{r} v_\theta- \frac{1}{r} u\Big)\frac{(r^2v)_r}{u^a}drd\theta\Big|\nonumber\\
\leq& \frac{1}{\varepsilon}\int_0^{2\pi}\int_{r_0}^1r(u_\theta^2+v_\theta^2)drd\theta+C\varepsilon\int_0^{2\pi}\int_{r_0}^1u_r^2drd\theta\nonumber.
\end{align}
Thus,
\beno
|I_1|\leq C(\eta+\varepsilon)\int_0^{2\pi}\int_{r_0}^1r(u_\theta^2+v_\theta^2)drd\theta+C\varepsilon^2(\eta+\varepsilon)\int_0^{2\pi}\int_{r_0}^1u_r^2drd\theta.
\eeno

By the same argument as above, we can deduce the following estimate for $I_2$:
\begin{align}
|I_2|&=\varepsilon^2\Big|\int_0^{2\pi}\int_{r_0}^1 \Big(\frac{1}{r} u_{\theta\theta}+ r u_{rr}+  u_{r}+\frac{2}{r} v_\theta- \frac{1}{r} u\Big)\frac{r^2v\partial_r u^a}{(u^a)^2}drd\theta\Big|\nonumber\\
&\leq C(\eta+\varepsilon)\int_0^{2\pi}\int_{r_0}^1r(u_\theta^2+v_\theta^2)drd\theta+C\varepsilon^2\int_0^{2\pi}\int_{r_0}^1u_r^2drd\theta.\nonumber
\end{align}
Summing the estiamtes of $I_1$ and $I_2$, we obtain that
\begin{align}
|I_1|+|I_2|\leq C(\eta+\varepsilon)\|\sqrt{r}(u_\theta, v_\theta)\|_2^2+C\varepsilon^2(\eta+\varepsilon)\|u_r\|_2^2.\nonumber
\end{align}

For $I_3$, the high order terms can be estimated by integrating by parts:
\begin{align}
\Big|\int_0^{2\pi}\int_{r_0}^1 \Big(\frac{1}{r} v_{\theta\theta}+ r v_{rr}\Big)\frac{r v_\theta}{u^a}drd\theta\Big|\leq \frac{C}{ \varepsilon}\|\sqrt{r}(u_\theta, v_\theta)\|_2^2.\nonumber
\end{align}
The lower order terms can be estimated by the divergence-free condition and (\ref{positive lower bound})
\begin{align}
\Big|\int_0^{2\pi}\int_{r_0}^1 \Big(v_{r}-\frac{2}{r} u_\theta- \frac{1}{r} v\Big)\frac{r\partial_\theta v}{u^a}drd\theta\Big|\leq
C\int_0^{2\pi}\int_{r_0}^1r( u_\theta^2+v_\theta^2)drd\theta.\nonumber
\end{align}
Thus, we obtain
\begin{align}
|I_3|\leq C\varepsilon\|\sqrt{r}(u_\theta, v_\theta)\|_2^2.\nonumber
\end{align}

For $I_4$, the same argument as above gives that
\begin{align}
|I_4|\leq C \varepsilon^2\|\sqrt{r}(u_\theta, v_\theta)\|_2^2..\nonumber
\end{align}

Finally, we can choose $\varepsilon_0>0, \eta_0>0$ such that for any $\varepsilon\in (0, \varepsilon_0), \eta\in (0,\eta_0)$, there holds
\begin{align}
|I_3|+|I_4|\leq C\varepsilon\|\sqrt{r}(u_\theta, v_\theta)\|_2^2.\nonumber
\end{align}

Hence, we deduce that there exist $\varepsilon_0>0, \eta_0>0$ such that for any $\varepsilon\in (0,\varepsilon_0), \eta\in (0,\eta_0)$, the diffusive term \eqref{diffusion term} can be controlled by
\begin{align}
C(\eta+\varepsilon)\int_0^{2\pi}\int_{r_0}^1r(u_\theta^2+v_\theta^2)drd\theta+C\varepsilon^2(\eta+\varepsilon)\|u_r\|_2^2\nonumber
\end{align}
which combining with the estimates (\ref{e:linear term in positivity estimate}) and (\ref{e:pressure estimate in positivity estimate}) implies that there exist $\varepsilon_0>0, \eta_0>0$ such that for any $\varepsilon\in (0,\varepsilon_0), \eta\in (0,\eta_0)$, the lift hand side of $(\ref{p:positivity estimate})$ has a lower bound
\begin{align}\label{e:left hand in positivity estimate old}
&\Big(1-C\eta-C\e\Big)\int_0^{2\pi}\int_{r_0}^1r(u_\theta^2+v_\theta^2)drd\theta\nonumber\\
&+\int_{r_0}^{1}\int_0^{2\pi}\frac{1}{u_e}\Big(\partial_r u_e+r\partial{rr}u_e-\frac{u_e}{r}\Big)(rv)^2d\theta dr-C\varepsilon^2(\eta+\e)\|u_r\|_2^2.
\end{align}
Let $g=\frac{rv}{u_e}$, then
\begin{align*}
&\int_0^{2\pi}\int_{r_0}^1r(u_\theta^2+v_\theta^2)drd\theta=\int_0^{2\pi}\int_{r_0}^1 \Big(r[(u_eg)_r]^2+\frac{u_e^2}{r}g_\theta^2\Big) drd\theta\\
=&\int_0^{2\pi}\int_{r_0}^1 \Big(ru^2_eg_r^2+\frac{u_e^2}{r}g_\theta^2\Big) drd\theta-\int_0^{2\pi}\int_{r_0}^1 u_e\Big(\partial_r u_e+r\partial{rr}u_e\Big)g^2 drd\theta.
\end{align*}
Thus,
\begin{align*}
&\Big(1-C\eta-C\e\Big)\int_0^{2\pi}\int_{r_0}^1r(u_\theta^2+v_\theta^2)drd\theta+\int_{r_0}^{1}\int_0^{2\pi}\frac{1}{u_e}\Big(\partial_r u_e+r\partial{rr}u_e-\frac{u_e}{r}\Big)(rv)^2d\theta dr\\
=&\int_0^{2\pi}\int_{r_0}^1 \Big(ru^2_eg_r^2+\frac{u_e^2}{r}g_\theta^2-\frac{u_e^2}{r}g^2\Big) drd\theta-C(\eta+\e)\int_0^{2\pi}\int_{r_0}^1r(u_\theta^2+v_\theta^2)drd\theta.
\end{align*}
By the Poincar\'{e} inequality,
\begin{align*}
\int_0^{2\pi}\int_{r_0}^1 \frac{u_e^2}{r}g_\theta^2 drd\theta&\geq\int_0^{2\pi}\int_{r_0}^1 \frac{u_e^2}{r}g^2 drd\theta,\\
C\int_0^{2\pi}\int_{r_0}^1 ru^2_eg_r^2 drd\theta&\geq \int_0^{2\pi}\int_{r_0}^1 \frac{u_e^2}{r}g^2 drd\theta,
\end{align*}
here we have used the fact that $u_e$ has a positive lower bound and upper bound in the second inequality.

Thus
\begin{align*}
&\int_0^{2\pi}\int_{r_0}^1 \Big(ru^2_eg_r^2+\frac{u_e^2}{r}g_\theta^2-\frac{u_e^2}{r}g^2\Big) drd\theta\\
\geq &\int_0^{2\pi}\int_{r_0}^1 \Big(\frac12 ru^2_eg_r^2+\frac{1}{2C}\frac{u_e^2}{r}g^2+\frac{1}{2C}\frac{u_e^2}{r}g_\theta^2+\Big(1-\frac{1}{2C}\Big)\frac{u_e^2}{r}g-\frac{u_e^2}{r}g^2\Big)drd\theta\\
=&\int_0^{2\pi}\int_{r_0}^1 \Big(\frac12 ru^2_eg_r^2+\frac{1}{2C}\frac{u_e^2}{r}g_\theta^2\Big)drd\theta\\
\geq &\widetilde{C}\int_0^{2\pi}\int_{r_0}^1r(u_\theta^2+v_\theta^2)drd\theta,
\end{align*}
where we used
\begin{align*}
|u_\theta|^2=|(rv)_r|^2=|(gu_e)_r|^2\leq C(g^2+g^2_r).
\end{align*}
So (\ref{e:left hand in positivity estimate old}) has a lower bound
\begin{align}\label{e:left hand in positivity estimate}
&\Big(\widetilde{C}-C\eta-C\e\Big)\int_0^{2\pi}\int_{r_0}^1r(u_\theta^2+v_\theta^2)drd\theta-C\varepsilon^2(\eta+\e)\|u_r\|_2^2.
\end{align}

Thanks to the Young inequality and Poincar\'{e} inequality, the right hand side of $(\ref{p:positivity estimate})$ can be directly estimated as follows:
\begin{align}\label{e:right hand in positivity estimate}
&\Big|\int_{r_0}^{1}\int_0^{2\pi}\Big[F_1\Big(\frac{r^2v}{u^a}\Big)_r+F_2\Big(\frac{rv}{u^a}\Big)_\theta\Big]d\theta dr\Big|\nonumber\\
=&\Big|\int_{r_0}^{1}\int_0^{2\pi}\Big[F_1\Big(\frac{r(rv)_r+rv}{u^a}-\frac{r^2vu^a_r}{(u^a)^2}\Big)+F_2\Big(\frac{r v_\theta}{u^a}-\frac{rv u^a_\theta}{(u^a)^2}\Big)\Big]d\theta dr\Big|\nonumber\\
\leq &\frac{\widetilde{C}}{4}\int_0^{2\pi}\int_{r_0}^1r(u_\theta^2+v_\theta^2)drd\theta+C\|(F_1, F_2)\|_2^2.
\end{align}

Finally, collecting the estimate (\ref{e:left hand in positivity estimate}) and (\ref{e:right hand in positivity estimate}) together, we obtain (\ref{e:positivity estimate}), this complete the proof of this lemma.
\end{proof}


\section{Existence of error equations}

To deal with the nonlinear term and close the estimate, we need to estimate $\|(u,v)\|_\infty$. To this end, we consider the following ``$H^2$ estimate" for a Stokes system.
\subsection{``$H^2$ estimate" for Stokes system}
\indent

We consider the Stoke equations
\begin{align}\label{e:stoke error equation}
\left\{
\begin{array}{lll}
-\varepsilon^2 \big(ru_{rr}
+\frac{u_{\theta\theta}}{r}+2\frac{v_{\theta}}{r}+u_r-\frac{u}{r}\big)+p_\theta=f_u,\\[5pt]
-\varepsilon^2\big(rv_{rr}+\frac{v_{\theta\theta}}{r}-2\frac{u_{\theta}}{r}+v_r-\frac{v}{r}\big)+rp_r=g_v,\\[5pt]
u_\theta+\partial_r(rv)=0,  \\[5pt]
 u(\theta,r)=u(\theta+2\pi,r),\ v(\theta,r)=v(\theta+2\pi,r),\\[5pt]
u(\theta,1)=0,\ v(\theta,1)=0, \\[5pt]
u(\theta,r_0)=0, \ v(\theta,r_0)=0.
\end{array}
\right.
\end{align}
\begin{lemma}
Let $(u,v)$ be a smooth solution of (\ref{e:stoke error equation}), then we have
\begin{align}\label{s:Stokes estiamate}
&\int_{r_0}^{1}\int_0^{2\pi} [(u_{\theta\theta}+v_\theta)^2+(v_{\theta\theta}-u_\theta)^2]d\theta dr+\int_{r_0}^{1}\int_0^{2\pi}\big(u_{r\theta}^2+v_{r\theta}^2\big)d\theta dr\nonumber\\
\leq&\frac{C}{\varepsilon^2}\Big|\int_{r_0}^1\int_0^{2\pi}\big[f_u u_{\theta\theta}+g_v v_{\theta\theta}\big]drd\theta\Big|.
\end{align}
\end{lemma}
\begin{proof}
Multipling $u_{\theta\theta}$ for the first equation and $v_{\theta\theta}$ for the second equation in (\ref{e:stoke error equation}), integrating in $\Omega$ and summing two terms together, we arrive at
\begin{align*}
&\int_{r_0}^1\int_0^{2\pi}\Big[-\varepsilon^2 \big(ru_{rr}
+\frac{u_{\theta\theta}}{r}+2\frac{v_{\theta}}{r}+u_r-\frac{u}{r}\big)+p_\theta\Big]u_{\theta\theta}drd\theta\\
&+\int_{r_0}^1\int_0^{2\pi}\Big[-\varepsilon^2\big(rv_{rr}+\frac{v_{\theta\theta}}{r}
-2\frac{u_{\theta}}{r}+v_r-\frac{v}{r}\big)+rp_r\Big]v_{\theta\theta}drd\theta\\
=&\int_{r_0}^1\int_0^{2\pi}\big[f_u u_{\theta\theta}+g_v v_{\theta\theta}\big]drd\theta.
\end{align*}
Integrating by parts and using the divergence-free condition, we deduce
\beno
\int_{r_0}^1\int_0^{2\pi}\big[p_\theta u_{\theta\theta}+p_r rv_{\theta\theta}\big]drd\theta=0.
\eeno
Then, we compute the diffusive term. Integrating by parts, we obtain
\beno
&&\int_{r_0}^1\int_0^{2\pi}-\varepsilon^2 u_{rr}r u_{\theta\theta}drd\theta=-\int_{r_0}^1\int_0^{2\pi}\varepsilon^2 u_{r\theta}(r u_{\theta})_rdrd\theta
=-\varepsilon^2\int_{r_0}^1\int_0^{2\pi} r u_{r\theta}^2drd\theta,\\
&&\int_{r_0}^1\int_0^{2\pi}-\varepsilon^2 v_{rr}r v_{\theta\theta}drd\theta=-\int_{r_0}^1\int_0^{2\pi}\varepsilon^2 v_{r\theta}(r v_{\theta})_rdrd\theta
=-\varepsilon^2\int_{r_0}^1\int_0^{2\pi} r v_{r\theta}^2drd\theta,
\eeno
and
\beno
&&\int_{r_0}^1\int_0^{2\pi}\varepsilon^2u_r u_{\theta\theta}drd\theta=\int_{r_0}^1\int_0^{2\pi}\varepsilon^2 v_r v_{\theta\theta}drd\theta=0,\\
&&\int_{r_0}^1\int_0^{2\pi}\varepsilon^2\frac{u}{r} u_{\theta\theta}drd\theta+ \int_{r_0}^1\int_0^{2\pi}\varepsilon^2 \frac{v}{r} v_{\theta\theta}drd\theta=-\varepsilon^2\int_{r_0}^1\int_0^{2\pi}\frac{u_\theta^2+v^2_\theta}{r}drd\theta.
\eeno
Summing all terms, we obtain
\beno
&&\int_{r_0}^1\int_0^{2\pi}\Big[-\varepsilon^2 \big(ru_{rr}
+\frac{u_{\theta\theta}}{r}+2\frac{v_{\theta}}{r}+u_r-\frac{u}{r}\big)+p_\theta\Big]u_{\theta\theta}drd\theta\\
&&+\int_{r_0}^1\int_0^{2\pi}\Big[-\varepsilon^2\big(rv_{rr}+\frac{v_{\theta\theta}}{r}
-2\frac{u_{\theta}}{r}+v_r-\frac{v}{r}\big)+rp_r\Big]v_{\theta\theta}drd\theta\\
&=&-\varepsilon^2\int_{r_0}^{1}\int_0^{2\pi}\frac{ (u_{\theta\theta}+v_\theta)^2+(v_{\theta\theta}-u_\theta)^2}{r}d\theta dr-\varepsilon^2\int_{r_0}^{1}\int_0^{2\pi}r\big(u_{r\theta}^2+v_{r\theta}^2\big)d\theta dr.
\eeno

Thus we obtain (\ref{s:Stokes estiamate}) and then
 complete the proof of this lemma.
\end{proof}

\subsection{Anisotropic Sobolev embedding}
\indent

To obtain the $L^\infty$ estimate, we need the following anisotropic Sobolev embedding.
\begin{lemma}
Let $(u,v)$ be smooth function and satisfy $(u,v)|_{r=1}=(u,v)|_{r=r_0}=(0,0)$, then
\begin{align}\label{s:Sobolev embedding}
\|(u,v)\|_\infty \leq C(\|(u_\theta,v_\theta)\|_2+\|(u_r,v_r)\|_2+\|(u_{r\theta},v_{r\theta})\|_2).
\end{align}
\end{lemma}
\begin{proof}
In fact, let
\beno
u(\theta,r)=\sum_{k\in Z}u_k(r)e^{ik\theta}, \ \forall r\in [r_0,1],
\eeno
thus
\beno
|u(\theta,r)|\leq\sum_{k\in Z}|u_k(r)|, \ \forall (\theta,r)\in [0,2\pi]\times[r_0,1].
\eeno
Notice that $u(\theta,r_0)=0, \ \forall \theta\in [0,2\pi]$ gives $u_k(\theta,r_0)=0,\ \forall \theta\in [0,2\pi], k\in Z$, hence
\beno
\|u_k\|_{\infty}\leq \sqrt{2}\|u_k\|^{\frac12}_2\|u'_k\|^{\frac12}_2.
\eeno
Thus,
\beno
\sum_{k\neq 0}\|u_k\|_\infty\leq \sqrt{2}\sum_{k\neq 0}\|u_k\|^{\frac12}_2\|u'_k\|^{\frac12}_2\leq C\Big(\sum_{k\neq 0}|k|^2\|u_k\|^2_2\Big)^{\frac14}\Big(\sum_{k\neq 0}|k|^2\|u'_k\|^2_2\Big)^{\frac14}.
\eeno

Moreover, it's easy to get
\beno
u_\theta(\theta,r)=\sum_{k\in Z}iku_k(r)e^{ik\theta}, \quad u_{r\theta}(\theta,r)=\sum_{k\in Z}iku'_k(r)e^{ik\theta}, \ \forall r\in [r_0,1].
\eeno
Thus,
\beno
\|u_\theta\|_2^2=\sum_{k\in Z}|k|^2\|u_k\|_2^2, \quad \|u_{r\theta}\|_2^2=\sum_{k\in Z}|k|^2\|u'_k\|_2^2.
\eeno
Hence, we obtain
\beno
\sum_{k\neq 0}\|u_k\|_\infty\leq C\|u_\theta\|^{\frac12}_2\|u_{r\theta}\|_2^{\frac12}.
\eeno
Moreover, due to $u_0(r)=\frac{1}{2\pi}\int_0^{2\pi}u(\theta,r)d\theta$ and $u(\theta,r_0)=0,\ \forall \theta\in [0,2\pi]$, it's easy to get
\beno
|u_0(r)|\leq C\|u_r\|_2, \quad \forall r\in [r_0,1].
\eeno

Finally, we obtain
\beno
\|u\|_{\infty}\leq C(\|u_\theta\|_2+\|u_r\|_2+\|u_{r\theta}\|_2).
\eeno

Similarly we can obtain the same result for $v$ and hence  we complete the proof of this lemma.
\end{proof}

\subsection{Existence for the error equations}
\indent

We apply the contraction mapping theorem to prove the existence for the error equations (\ref{e:error equation}).

\begin{proposition}\label{existence and error estimate of error equation}
There exist $\varepsilon_0>0,\eta_0>0$ such that for any $\varepsilon\in (0,\varepsilon_0), \eta\in (0,\eta_0)$, the error equations (\ref{e:error equation}) have a unique solution $(u,v)$ which satisfies
$$\|(u,v)\|_\infty\leq C\varepsilon.$$
\end{proposition}
\begin{proof}
For each smooth function $(u,v)$ which satisfies
\begin{align}\label{iterative conditon}
\left\{
\begin{array}{ll}
 u_\theta+(rv)_r=0,\\[5pt]
u(\theta,r)=u(\theta+2\pi,r),\ v(\theta,r)=v(\theta+2\pi,r),\\[5pt]
u(\theta,1)=0,\ u(\theta,r_0)=0,\\[5pt]
 v(\theta,1)=0, \ v(\theta,r_0)=0,
\end{array}
\right.
\end{align}
we consider the following linear problem:
\begin{align}
\left\{
\begin{array}{lll}
-\varepsilon^2 \big(r\bar{u}_{rr}
+\frac{\bar{u}_{\theta\theta}}{r}+2\frac{\bar{v}_{\theta}}{r}+\bar{u}_r-\frac{\bar{u}}{r}\big)+\bar{p}_\theta+
S_{\bar{u}}=R_u,\\[7pt]
-\varepsilon^2 \big(r\bar{v}_{rr}+\frac{\bar{v}_{\theta\theta}}{r}-2\frac{\bar{u}_{\theta}}{r}+\bar{v}_r-\frac{\bar{v}}{r}\big)
+r\bar{p}_r+S_{\bar{v}}=R_v,\\[7pt]
\partial_\theta \bar{u}+\partial_r(r\bar{v})=0,  \\[7pt]
 \bar{u}(\theta,r)=\bar{u}(\theta+2\pi,r),\  \bar{v}(\theta,r)=\bar{v}(\theta+2\pi,r), \\[7pt]
\bar{u}(\theta,1)=0,\ \bar{v}(\theta,1)=0, \\[7pt]
 \bar{u}(\theta,r_0)=0,\ \bar{v}(\theta,r_0)=0,
\end{array}
\right.\nonumber
\end{align}
where
\begin{align*}
S_{\bar{u}}:=&u^a\bar{u}_\theta+v^ar\bar{u}_r+\bar{u}u^a_\theta+\bar{v}ru^a_r+v^a\bar{u}+\bar{v}u^a,\\[5pt]
S_{\bar{v}}:=&u^a\bar{v}_\theta+v^ar\bar{v}_r+\bar{u}v^a_\theta+\bar{v}rv^a_r-2\bar{u}u^a,\\[5pt]
R_u:=& R_u^a-uu_\theta-vru_r-vu,\quad R_v:=R_v^a-uv_\theta-vrv_r-u^2.
\end{align*}

By linear stability estimate (\ref{e:linear stability estimates for linear equation}), we deduce that there exist $\varepsilon_0>0, \eta_0>0$ such that for any $\varepsilon\in (0,\varepsilon_0), \eta\in(0,\eta_0)$, there holds
\begin{align}
&\|(\bar{u}_\theta,\bar{v}_\theta)\|_2^2+\varepsilon^2\|\bar{u}_r\|_2^2\leq C\|(R_u, R_v)\|_2^2+\Big|\int_{r_0}^{1}\int_0^{2\pi}\big(R_u \bar{u}+R_v \bar{v}\big)d\theta dr\Big|.\nonumber
\end{align}
Direct computation gives that
\beno
\|(R_u, R_v)\|_2^2\leq C\|(R^a_u, R^a_v)\|_2^2+C\|(u,v)\|^2_\infty \|(u_r,u_\theta, v_\theta)\|_2^2.
\eeno
Moreover, there hold
\begin{align*}
\Big|\int_{r_0}^{1}\int_0^{2\pi}R_u \bar{u}d\theta dr\Big|=&\Big|\int_{r_0}^{1}\int_0^{2\pi}[R^a_u \bar{u}-uu_\theta \bar{u}-rvu_r\bar{u}-vu\bar{u}]d\theta dr\Big|\\[5pt]
\leq&C\|\bar{u}_r\|_2\|R^a_{u}\|_2+C\|\bar{u}_\theta\|_2\|u\|_2\|u\|_\infty
+C\|\bar{u}_r\|_2\|u_\theta\|_2\|u\|_\infty,\\[5pt]
\Big|\int_{r_0}^{1}\int_0^{2\pi}R_v \bar{v}d\theta dr\Big|=&\Big|\int_{r_0}^{1}\int_0^{2\pi}[R^a_v \bar{v}-uv_\theta \bar{v}-rvv_r\bar{v}-u^2\bar{v}]d\theta dr\Big|\\[5pt]
\leq&C\|\bar{u}_\theta\|_2\|R^a_v\|_2+C\|\bar{u}_\theta\|_2\|v_\theta\|_2\|u\|_\infty
+C\|\bar{u}_\theta\|_2\|u_\theta\|_2\|v\|_\infty\\[10pt]
&+C\|\bar{u}_\theta\|_2\|u_r\|_2\|u\|_\infty.
\end{align*}
Hence,
\begin{align*}
&\Big|\int_{r_0}^{1}\int_0^{2\pi}[\bar{u}R_u+\bar{v}R_v]d\theta dr\Big|\\
\leq&\frac12\|(\bar{u}_\theta,\bar{v}_\theta,\varepsilon \bar{u}_r)\|_2^2
+C\varepsilon^{-2}\|(R^a_u,R_v^a)\|^2_2+C\varepsilon^{-2}\|(u_\theta, v_\theta, \varepsilon u_r)\|_2^2\|(u,v)\|^2_\infty.
\end{align*}
Thus, there holds
\begin{align}\label{e:H1 estimate for linearized equation}
\|(\bar{u}_\theta,\bar{v}_\theta,\varepsilon\bar{u}_r\|_2^2
\leq C\varepsilon^{-2}\|(R^a_u, R^a_v)\|_2^2+C\varepsilon^{-2}\|(u_\theta, v_\theta, \varepsilon u_r)\|_2^2\|(u,v)\|^2_\infty.
\end{align}

Then, we consider the following Stokes problem:
\begin{align}
\left\{
\begin{array}{lll}
-\varepsilon^2\big( r\bar{u}_{rr}
+\frac{\bar{u}_{\theta\theta}}{r}+2\frac{\bar{v}_{\theta}}{r}+\bar{u}_r-\frac{\bar{u}}{r}\big)+\bar{p}_\theta
=R_u- S_{\bar{u}},\\[7pt]
-\varepsilon^2\big( r\bar{v}_{rr}+\frac{\bar{v}_{\theta\theta}}{r}-2\frac{\bar{u}_{\theta}}{r}+\bar{v}_r-\frac{\bar{v}}{r}\big)
+r\bar{p}_r=R_v-S_{\bar{v}},\\[7pt]
\partial_\theta \bar{u}+\partial_r(r\bar{v})=0,  \\[7pt]
 \bar{u}(\theta,r)=\bar{u}(\theta+2\pi,r),\  \bar{v}(\theta,r)=\bar{v}(\theta+2\pi,r),\\[7pt]
\bar{u}(\theta,1)=0,\ \bar{v}(\theta,1)=0, \\[7pt]
 \bar{u}(\theta,r_0)=0,\ \bar{v}(\theta,r_0)=0.
\end{array}
\right.\nonumber
\end{align}
By the Stoke estimate (\ref{s:Stokes estiamate}), we deduce
\beno
&&\|(\bar{u}_{\theta\theta}+\bar{v}_\theta,\bar{v}_{\theta\theta}-\bar{u}_\theta)\|_2^2
+\|(\bar{u}_{r\theta},\bar{v}_{r\theta})\|_2^2\\[5pt]
 &\leq& \frac{C}{\varepsilon^2}\Big|\int_{r_0}^1\int_0^{2\pi}\Big[(R_u-S_{\bar{u}}) \bar{u}_{\theta\theta}+(R_v-S_{\bar{v}}) \bar{v}_{\theta\theta}\Big]drd\theta\Big|.
\eeno

We compute the right hand term by terms. \\

{\bf 1)The term $\int_{r_0}^{1}\int_0^{2\pi}R_u \bar{u}_{\theta\theta}d\theta dr$ and $\int_{r_0}^{1}\int_0^{2\pi}R_v \bar{v}_{\theta\theta}d\theta dr$:} First, there holds
\begin{align*}
\Big|\int_{r_0}^{1}\int_0^{2\pi}R_u \bar{u}_{\theta\theta}d\theta dr\Big|&=\Big|\int_{r_0}^{1}\int_0^{2\pi}[R^a_u \bar{u}_{\theta\theta}-uu_\theta \bar{u}_{\theta\theta}-rvu_r\bar{u}_{\theta\theta}-vu\bar{u}_{\theta\theta}]d\theta dr\Big|\\[5pt]
&\leq C\|\bar{u}_\theta\|_2\|\partial_\theta R^a_u\|_2+\Big|\int_{r_0}^{1}\int_0^{2\pi}[uu_\theta \bar{u}_{\theta\theta}+rvu_r\bar{u}_{\theta\theta}+vu\bar{u}_{\theta\theta}]d\theta dr\Big|.
\end{align*}

It is easy to obtain
\begin{align*}
\Big|\int_{r_0}^{1}\int_0^{2\pi}[uu_\theta \bar{u}_{\theta\theta}+vu\bar{u}_{\theta\theta}]d\theta dr\Big|\leq \frac{\varepsilon^2}{10C}(\|\bar{u}_{\theta\theta}+\bar{v}_\theta\|_2^2+\|\bar{v}_{r\theta}\|_2^2)
+\frac{C}{\varepsilon^2}\|u\|_\infty^2\|u_\theta\|_2^2.
\end{align*}

Integrating by part, we deduce that
\begin{align*}
&\Big|\int_{r_0}^{1}\int_0^{2\pi}rvu_r\bar{u}_{\theta\theta}d\theta dr\Big|\\
\leq&\Big|\int_{r_0}^{1}\int_0^{2\pi}[(rv)_{\theta r}u\bar{u}_\theta+(rv)_{\theta }u\bar{u}_{\theta r}+(rv)_{r}u_\theta\bar{u}_\theta+rvu_\theta\bar{u}_{\theta r}]d\theta dr\Big|\\
\leq& \frac{\varepsilon^2}{10C}\|\bar{u}_{r\theta}\|_2^2+\frac{C}{\varepsilon^2}\|(u,v)\|_\infty^2\|(u_\theta,v_\theta)\|_2^2
+\Big|\int_{r_0}^{1}\int_0^{2\pi}[u_{\theta \theta}u\bar{u}_\theta+u_\theta^2\bar{u}_\theta]d\theta dr\Big|\\
\leq& \frac{\varepsilon^2}{10C}\|\bar{u}_{r\theta}\|_2^2+\frac{C}{\varepsilon^2}\|(u,v)\|_\infty^2\|(u_\theta,v_\theta)\|_2^2
+\Big|\int_{r_0}^{1}\int_0^{2\pi}u u_\theta\bar{u}_{\theta\theta}d\theta dr\Big|\\
\leq& \frac{\varepsilon^2}{10C}\|\bar{u}_{\theta\theta}+\bar{v}_\theta\|_2^2+ \frac{\varepsilon^2}{10C}\|(\bar{u}_{r\theta},\bar{v}_{r\theta})\|_2^2+\frac{C}{\varepsilon^2}\|(u,v)\|_\infty^2\|(u_\theta,v_\theta)\|_2^2.
\end{align*}

Thus, we obtain
\begin{align*}
\Big|\int_{r_0}^{1}\int_0^{2\pi}R_u \bar{u}_{\theta\theta}d\theta dr\Big|\leq& C\|\bar{u}_\theta\|_2\|\partial_\theta R^a_u\|_2\\ &+\frac{\varepsilon^2}{10C}\|\bar{u}_{\theta\theta}+\bar{v}_\theta\|_2^2+ \frac{\varepsilon^2}{10C}\|(\bar{u}_{r\theta},\bar{v}_{r\theta})\|_2^2+\frac{C}{\varepsilon^2}\|(u,v)\|_\infty^2\|(u_\theta,v_\theta)\|_2^2.
\end{align*}
Similarly,
\begin{align*}
\Big|\int_{r_0}^{1}\int_0^{2\pi}R_v \bar{v}_{\theta\theta}d\theta dr\Big|\leq& C\|\bar{v}_\theta\|_2\|\partial_\theta R^a_u\|_2\\ &+\frac{\varepsilon^2}{10C}\|\bar{v}_{\theta\theta}-\bar{u}_\theta\|_2^2+ \frac{\varepsilon^2}{10C}\|(\bar{u}_{r\theta},\bar{v}_{r\theta})\|_2^2+\frac{C}{\varepsilon^2}\|(u,v)\|_\infty^2\|(u_\theta,v_\theta)\|_2^2.
\end{align*}
{\bf 2)The term $\int_{r_0}^{1}\int_0^{2\pi}S_{\bar{u}}\bar{u}_{\theta\theta}d\theta dr$ and $\int_{r_0}^{1}\int_0^{2\pi}S_{\bar{v}}\bar{v}_{\theta\theta}d\theta dr$:}\\

First, there holds
 \beno
\Big| \int_{r_0}^{1}\int_0^{2\pi}S_{\bar{u}}\bar{u}_{\theta\theta}d\theta dr\Big|=\Big| \int_{r_0}^{1}\int_0^{2\pi}[u^a\bar{u}_\theta+v^ar\bar{u}_r+u^a_\theta\bar{u}+ru^a_r\bar{v}+v^a\bar{u}+u^a\bar{v}]\bar{u}_{\theta\theta}d\theta dr\Big|.
 \eeno

 Integrating by parts, we deduce that
\begin{align*}
\Big|\int_{r_0}^{1}\int_0^{2\pi}u^a\bar{u}_\theta\bar{u}_{\theta\theta}d\theta dr\Big|=&\frac12\Big|\int_{r_0}^{1}\int_0^{2\pi}u^a_\theta\bar{u}^2_\theta d\theta dr\Big|\leq C(\eta+\varepsilon) \|\bar{u}_\theta\|_2^2,\\
\Big|\int_{r_0}^{1}\int_0^{2\pi}rv^a\bar{u}_r\bar{u}_{\theta\theta}d\theta dr\Big|=&\Big|\int_{r_0}^{1}\int_0^{2\pi}[(rv^a)_\theta \bar{u}_r+rv^a\bar{u}_{r\theta}]\bar{u}_\theta d\theta dr\Big|\\[5pt]
\leq& C\varepsilon(\eta+\varepsilon) \|\bar{u}_\theta\|_2\|(\bar{u}_r,\bar{u}_{r\theta})\|_2,\\[5pt]
\Big|\int_{r_0}^{1}\int_0^{2\pi}u^a_\theta\bar{u}\bar{u}_{\theta\theta}d\theta dr\Big|=&\Big|\int_{r_0}^{1}\int_0^{2\pi}[u^a_\theta \bar{u}^2_\theta+u^a_{\theta\theta}\bar{u}_{\theta}\bar{u}] d\theta dr\Big|\\[5pt]
\leq &C(\eta+\varepsilon) \|\bar{u}_\theta\|^2_2+C\varepsilon(\eta+\varepsilon) \|\bar{u}_\theta\|_2\|\bar{u}_r\|_2,\\[5pt]
\Big|\int_{r_0}^{1}\int_0^{2\pi}ru^a_r\bar{v}\bar{u}_{\theta\theta}d\theta dr\Big|=&\Big|\int_{r_0}^{1}\int_0^{2\pi}[(ru_r^a)_\theta \bar{v} \bar{u}_\theta+ru^a_r\bar{v}_{\theta}\bar{u}_\theta] d\theta dr\Big|\\[5pt]
\leq& C(\eta+\varepsilon) \|\bar{u}_\theta\|^2_2+C\varepsilon^{-1}\eta \|\bar{u}_\theta\|_2\|\bar{v}_\theta\|_2,\\[5pt]
\Big|\int_{r_0}^{1}\int_0^{2\pi}v^a\bar{u}\bar{u}_{\theta\theta}d\theta dr\Big|=&\Big|\int_{r_0}^{1}\int_0^{2\pi}[v^a_\theta \bar{u} \bar{u}_\theta+v^a\bar{u}^2_\theta] d\theta dr\Big|\\[5pt]
\leq& C\varepsilon(\eta+\varepsilon) \|\bar{u}_\theta\|_2\|\bar{u}_r\|_2+C\varepsilon (\eta+\varepsilon) \|\bar{u}_\theta\|_2^2,\\[5pt]
\Big|\int_{r_0}^{1}\int_0^{2\pi}u^a\bar{v}\bar{u}_{\theta\theta}d\theta dr\Big|=&\Big|\int_{r_0}^{1}\int_0^{2\pi}[u^a_\theta \bar{v} \bar{u}_\theta+u^a\bar{v}_\theta\bar{u}_\theta] d\theta dr\Big|\leq C(\|\bar{u}_\theta\|_2\|\bar{v}_\theta\|_2+ \|\bar{u}_\theta\|_2^2).
\end{align*}
Thus, there exist $\varepsilon_0>0, \eta_0>0$ such that for any $\varepsilon\in (0,\varepsilon_0), \eta\in(0,\eta_0)$, there holds
\beno
\Big| \int_{r_0}^{1}\int_0^{2\pi}S_{\bar{u}}\bar{u}_{\theta\theta}d\theta dr\Big|\leq\frac{\varepsilon^2}{10C}\|\bar{u}_{r\theta}\|_2^2+C\varepsilon^{-1}\|(\bar{u}_\theta,\bar{v}_\theta,\varepsilon \bar{u}_r)\|^2_2.
\eeno
Similarly, we can obtain
\beno
\Big| \int_{r_0}^{1}\int_0^{2\pi}S_{\bar{v}}\bar{v}_{\theta\theta}d\theta dr\Big|\leq\frac{\varepsilon^2}{10C}\|\bar{u}_{r\theta}\|_2^2+C\varepsilon^{-1}\|(\bar{u}_\theta,\bar{v}_\theta,\varepsilon \bar{u}_r)\|^2_2.
\eeno
Thus, we deduce that
\begin{align}\label{Higher order tangential estimate}
\|(\bar{u}_{r\theta}, \bar{v}_{r\theta})\|_2^2\leq C\varepsilon^{-2}\|(\partial_\theta R_u^a, \partial_\theta R_v^a)\|_2^2+ C\varepsilon^{-3}\|(\bar{u}_\theta,\bar{v}_\theta,\varepsilon \bar{u}_r)\|^2_2+C\varepsilon^{-4}\|(u,v)\|_\infty^2\|(u_\theta,v_\theta)\|_2^2.
\end{align}

Finally, combining (\ref{s:Sobolev embedding}),(\ref{e:H1 estimate for linearized equation}) and (\ref{Higher order tangential estimate}), we obtain the following $L^\infty$ estimate: there exist $\varepsilon_0>0, \eta_0>0$ such that for any $\varepsilon\in (0,\varepsilon_0), \eta\in(0,\eta_0)$, there holds
\begin{align}\label{e:L infinity estimate}
\|(\bar{u},\bar{v})\|_\infty^2
\leq& C\varepsilon^{-4}\|(R_u^a, R_v^a, \partial_\theta R_u^a, \partial_\theta R_v^a)\|_2^2+C\varepsilon^{-3}\|(\bar{u}_\theta,\bar{v}_\theta,\varepsilon\bar{u}_r)\|_2^2\nonumber\\[5pt]
 &\quad \quad + C\varepsilon^{-4}\|(u_\theta,v_\theta,\varepsilon u_r)\|_2^2\|(u,v)\|_\infty^2.
\end{align}

Set
\beno
\|(u,v)\|_Y^2:=\|(u_\theta,v_\theta,\varepsilon u_r)\|_2^2+ \kappa \varepsilon^3\|(u,v)\|_\infty^2,
\eeno
where $\kappa\ll 1$ is a fixed positive number, and combining the estimates (\ref{e:H1 estimate for linearized equation}) and (\ref{e:L infinity estimate}), we arrive at
\begin{align}\label{e:iterative estimate}
\|(\bar{u},\bar{v})\|_Y^2\leq C\varepsilon^{-2}\|(R^a_u, R^a_v, \partial_\theta R_u^a, \partial_\theta R_v^a)\|_2^2+C\varepsilon^{-5}\|(u, v)\|_Y^4.
\end{align}

Let $Y=\{(u,v)\in C^\infty: (u,v) \ satisfes \ (\ref{iterative conditon})\ and \ \|(u,v)\|_Y<+\infty\}$.
Thus, due to
\beno
\|(R_u^a, R_v^a,\partial_\theta R_u^a, \partial_\theta R_v^a)\|_2\leq C\varepsilon^4,
\eeno
there exist $\varepsilon_0>0, \eta_0>0$ such that for any $\varepsilon\in (0,\varepsilon_0), \eta\in(0,\eta_0)$, the operator
\beno
(u,v)\mapsto (\bar{u},\bar{v})
\eeno
maps the ball $\{(u,v): \|(u,v)\|^2_Y\leq 2C^3\varepsilon^6\}$  in $Y$ into itself.

Moreover, for every two pairs
$(u_1, v_1)$ and $(u_2, v_2)$ in the ball, we have
\begin{align}\label{e:contraction estimate}
\|(\bar{u}_1-\bar{u}_2, \bar{v}_1-\bar{v}_2)\|^2_Y\leq C\varepsilon^{-5}(\|(u_1,v_1)\|_Y^2+\|(u_2,v_2)\|_Y^2)\|(u_1-u_2, v_1-v_2)\|_Y^2.
\end{align}
In fact, set
\beno
\bar{U}:=\bar{u}_1-\bar{u}_2,\quad \bar{V}:=\bar{v}_1-\bar{v}_2,\quad \bar{P}:=\bar{p}_1-\bar{p}_2,
\eeno
then we have
\begin{align}
\left\{
\begin{array}{lll}
-\varepsilon^2\big(r \bar{U}_{rr}
+\frac{\bar{U}_{\theta\theta}}{r}+2\frac{\bar{V}_{\theta}}{r}+\bar{U}_r-\frac{\bar{U}}{r}\big)+\bar{P}_\theta+
S_{\bar{U}}=R_U,\\[7pt]
-\varepsilon^2\big( r\bar{V}_{rr}+\frac{\bar{V}_{\theta\theta}}{r}-2\frac{\bar{U}_{\theta}}{r}+\bar{V}_r-\frac{\bar{V}}{r}\big)
+r\bar{P}_r+S_{\bar{V}}=R_V,\\[7pt]
\partial_\theta \bar{U}+\partial_r(r\bar{V})=0,  \\[7pt]
 \bar{U}(\theta,r)=\bar{U}(\theta+2\pi,r),\  \bar{V}(\theta,r)=\bar{V}(\theta+2\pi,r),\\[7pt]
\bar{U}(\theta,1)=0,\  \bar{V}(\theta,1)=0,\\[7pt]
 \bar{U}(\theta,r_0)=0,\ \bar{V}(\theta,r_0)=0,
\end{array}
\right.\nonumber
\end{align}
where
\begin{align*}
R_U \triangleq R_{u_1}-R_{u_2}=&u_2u_{2\theta}+v_2ru_{2r}+v_2u_2-u_1u_{1\theta}-v_1ru_{1r}-v_1u_1\\[5pt]
=&(u_2-u_1)\partial_\theta u_2+u_1\partial_\theta(u_2-u_1)\\[5pt]
&+(v_2-v_1)r\partial_r u_2+v_1r\partial_r(u_2-u_1)+(v_2-v_1)u_2+v_1(u_2-u_1),\\[5pt]
R_V \triangleq R_{v_1}-R_{v_2}=&u_2v_{2\theta}+v_2rv_{2r}-u_2^2-u_1v_{1\theta}-v_1rv_{1r}+u^2_1\\[5pt]
=&(u_2-u_1)\partial_\theta v_2+u_1\partial_\theta(v_2-v_1)\\[5pt]
&+(v_2-v_1)r\partial_r v_2+v_1r\partial_r(v_2-v_1)+(u_1+u_2)(u_1-u_2).
\end{align*}

Thus, following the estimate (\ref{e:iterative estimate}) line by line, we obtain (\ref{e:contraction estimate}).
Hence, there exist $\varepsilon_0>0,\eta_0>0$ such that for any $\varepsilon\in (0,\varepsilon_0), \eta\in (0,\eta_0)$, the operator
\beno
(u,v)\mapsto (\bar{u},\bar{v})
\eeno
maps the ball $\{(u,v): \|(u,v)\|^2_Y\leq 2C^3\varepsilon^6\}$ in $Y$ into itself and is a contraction mapping. This complete the proof of this proposition.
\end{proof}

Now we can give the proof of Theorem \ref{main theorem}.
\begin{proof}
Combining Proposition \ref{existence and error estimate of error equation} and the approximate solution (\ref{approximate solution-1})-(\ref{approximate solution-3}) of the Navier-Stokes equations (\ref{NS-curvilnear}), we easily obtain Theorem \ref{main theorem}.
\end{proof}

\section{Appendix}

{\bf Appendix A:  Constant coefficient periodic PDE}

In this appendix we give a brief argument to solve the following problem which also appeared  in Appendix A in \cite{FGLT},
\begin{eqnarray}\label{periodic heat equation}
\left \{
\begin {array}{ll}
(Q_0)_\theta=u_{e}(1)(Q_0)_{\psi\psi},\\[5pt]
Q_0(\theta,\psi)=Q_0(\theta+2\pi,\psi),\\[5pt]
Q_0\big|_{\psi=0}=f_{\eta}(\theta),\  Q_0\big|_{\psi\rightarrow-\infty}=0,\label{Q0-0}
\end{array}
\right.
\end{eqnarray}
where
\begin{align*}
f_{\eta}(\theta)=\alpha^2+2\alpha\eta f(\theta)+\eta^2 f^2(\theta)-u_{e}^2(1)=2\alpha\eta f(\theta)+\eta^2 f^2(\theta)-\frac{\eta^2}{2\pi} \int_0^{2\pi}f^2(\theta)d\theta.
\end{align*}

Let  $Q_0(\theta,\psi)=\sum\limits_{k\in\mathbb{Z}}e^{ik\theta}Q_{0k}(\psi)$  and substitute it into (\ref{periodic heat equation}), we obtain
\begin{eqnarray*}
\left \{
\begin {array}{ll}
ikQ_{0k}=u_{e}(1)Q_{0k}'',\\[5pt]
Q_{0k}\big|_{\psi=0}=\widehat{f_{\eta}}(k)=\frac{1}{2\pi}\int_0^{2\pi}e^{-ik\theta}f_{\eta}(\theta) d\theta,
\\[5pt]
Q_{0k}\big|_{\psi\rightarrow-\infty}=0.
\end{array}
\right.
\end{eqnarray*}

It is easy to get
\begin{align*}
Q_{0k}(\psi)=\widehat{f_{\eta}}(k)e^{\alpha_k\psi}
\end{align*}
with
$\alpha_k=\sqrt{\frac{|k|}{2u_e(1)}}(1+\text{sgn}k\cdot i)$. Then
\begin{align*}
Q_0(\theta,\psi)=\sum\limits_{k\in\mathbb{Z}}e^{ik\theta}\widehat{f_{\eta}}(k)e^{\alpha_k\psi}\in X
\end{align*}
and
\begin{align}
\|Q_0\|_X\leq C\eta.\label{norm of q0}
\end{align}

{\bf Appendix B: Construction of corrector $h(\theta,r)$}

In this section, we give a construction of corrector $h(\theta,r)$ defined in subsection \ref{Approximate solutions}. Firstly, we give a simple lemma which is similar to Lemma 6.1 in Appendix B in \cite{FGLT}.

\begin{Lemma}\label{corector equation}
Assume that $K(\theta,r)$ is a $2\pi$-periodic smooth function which satisfies
\beno
\int_0^{2\pi}K(\theta,r)d\theta=0, \ \forall r\in [r_0,1]; \quad K(\theta, r_0)=K(\theta, 1)=0,
\eeno
then there exists a $2\pi$-periodic function $h(\theta,r)$ such that
\begin{align}\label{corrector $h$}
&\partial_\theta h(\theta,r)=K(\theta,r); \ h(\theta, r_0)=h(\theta, 1)=0;\nonumber\\
&\int_0^{2\pi}h(\theta,r)d\theta=0, \ \|\partial_\theta^j\partial_r^kh\|_2\leq C\|\partial_\theta^j\partial_r^kK\|_2.
\end{align}
\end{Lemma}
\begin{proof}
Let
\beno
K(\theta,r)=\sum_{n\neq 0}K_n(r)e^{in\theta}, \quad K_n(r_0)=K_n(1)=0.
\eeno
Set
\beno
h(\theta,r)=\sum_{n\neq 0}\frac{K_n(r)}{in}e^{in\theta}.
\eeno
It's easy to justify that $h(\theta,r)$ satisfies (\ref{corrector $h$}) which complete the proof.
\end{proof}

Next, we construct the corrector $h(\theta,r)$ by the above lemma.
Direct computation gives
\begin{align*}
u^a_\theta+(rv^a)_r=&\varepsilon^4\partial_\theta h(\theta,r)+K(\theta,r),
\end{align*}
where
\begin{align*}
K(\theta,r)=&\varepsilon^4\chi(r)[Y\partial_Yv_p^{(4)}(\theta,Y)+v_p^{(4)}(\theta,Y)]\\[7pt]
&+\varepsilon^4(1-\chi(r))[Z\partial_Z\hat{v}_p^{(4)}(\theta,Z)+\hat{v}_p^{(4)}(\theta,Z)]\\[5pt]
&+r\chi'(r)\Big(\sum_{i=1}^4\varepsilon^i v_p^{(i)}(\theta,Y)-\sum_{i=1}^4\varepsilon^i \hat{v}_p^{(i)}(\theta,Z)\Big).
\end{align*}

Notice that $\chi'(r)=0,\ r\in [r_0,r_1]\cup [r_2,1]$ and the property of $(v_p^{(i)},  \hat{v}_p^{(i)})$, we know that $K(\theta,r)=O(\varepsilon^4)$ and
\beno
K(\theta,1)=0,\ K(\theta,r_0)=0.
\eeno

Moreover, notice that
\beno
\int_0^{2\pi}v_p^{(i)}(\theta,Y)d\theta=0, \ \forall \ Y\leq 0; \quad \int_0^{2\pi}\hat{v}_p^{(i)}(\theta,Z)d\theta=0, \ \forall \ Z\geq 0,\quad i=1,\cdot\cdot\cdot,4,
\eeno
we deduce that
\beno
\int_0^{2\pi}K(\theta,r)d\theta=0, \ \forall r\in [r_0,1].
\eeno

Thus, we can choose $h(\theta,r)$ by Lemma \ref{corector equation} such that
\beno
\varepsilon^4\partial_\theta h(\theta,r)+K(\theta,r)=0, \ h(\theta, 1)=0, \ h(\theta,r_0)=0, \ \|\partial_\theta^j\partial_r^kh\|_2\leq C\varepsilon^{-k}.\\
\eeno

{\bf Appendix C: Generalized Prandtl-Batchelor theory in an annulus.}

In this appendix, we derive a generalized Prandtl-Batchelor theory for the forced steady Navier-Stokes equations on a two-dimensional
annulus. For Prandtl-Batchelor theory on a simply connected
domain, we refer to \cite{B, K1998, K2000} or Appendix C in \cite{FGLT}.

Consider the two-dimensional forced Navier-Stokes equations (\ref{ns}).
Let $w^\varepsilon=u^\varepsilon_y-v^\varepsilon_x$ be the vorticity, then it's easy to obtain that
\begin{align*}
(v^\varepsilon,-u^\varepsilon)^T w^\varepsilon+\nabla\Big(p^\varepsilon+\frac12|\mathbf{u}^\varepsilon|^2\Big)=\varepsilon^2(\triangle\mathbf{u}^\varepsilon +\mathbf{F}).
\end{align*}
Let $\Phi^\varepsilon(x,y)$ be the stream function which satisfies $\partial_y\Phi^\varepsilon=u^\varepsilon, -\partial_x\Phi^\varepsilon=v^\varepsilon$. Assuming that
\begin{align*}
\Gamma^\varepsilon_c:=\{(x,y)|\Phi^\varepsilon(x,y)=c\}
\end{align*}
is a closed curve for any $c\in \Phi^\varepsilon(\Omega)$.
By integrating the Navier-Stokes equations over the closed curve $\Gamma^\varepsilon_c$ we deduce that
\begin{align*}
\int_{\Gamma^\varepsilon_c}\Big((v^\varepsilon,-u^\varepsilon)^T  w^\varepsilon+\nabla\Big(p^\varepsilon+\frac12|\mathbf{u}^\varepsilon|^2\Big)\Big)d\vec{\tau}=0,
\end{align*}
thus we obtain that
\begin{align*}
\int_{\Gamma^\varepsilon_c}(\triangle\mathbf{u}^\varepsilon +\mathbf{F})d\vec{\tau}=0.
\end{align*}
Assume that $\mathbf{u}^\varepsilon\rightarrow \mathbf{u}_e$ in some strong sense and $\mathbf{u}_e$ is the solution of forced steady Euler equations, then there holds
\begin{align}\label{Necessary condition in the velocity formulation}
\int_{\Gamma_c^e}(\triangle\mathbf{u}_e+\mathbf{F})d\vec{\tau}=0,
\end{align}
where $\Gamma_c^e$ is the stream line of the Euler flow.

If the Euler flow $\mathbf{u}_e$ has no stagnation point in the annulus, then it's a rotating shear flow $\mathbf{u}_e=u_e(r)\vec{e}_\theta$ and the stream line is circle, see \cite{HN2}. From this and (\ref{Necessary condition in the velocity formulation}), we deduce that
\begin{align}\label{Necessary condition for velocity}
u''_e(r)+\frac{1}{r}u'_e(r)-\frac{1}{r^2}u_e(r)=-\bar{F}_u(r),
\end{align}
where $F_u(\theta,r)=\mathbf{F}\cdot \vec{e}_\theta$ and $\bar{F}_u(r):=\frac{1}{2\pi}\int_0^{2\pi}F_u(\theta,r)d\theta.$\\

Similarly, for a finite channel $\{(x,y):x\in\lbrack0,2\pi],\ y_{0}\leq y\leq y_{1}\}$,
if the inviscid limit of Navier-Stokes flow is a shear flow $(u_e(y),0)^T$,
then there also holds
$$u''_{e}(y)+\frac{1}{2\pi}\int_0^{2\pi}F_1(x,y)dx=0,$$
where $F_1$ is the first component of $\mathbf{F}$.

\section*{Acknowledgments}

M.Fei is partially supported by NSF of China under Grant No.11871075, 11971357 and 12271004. Z.Lin is partially supported by the NSF grants DMS-1715201 and DMS-2007457. T.Tao is partially supported by the NSF of China under Grant 11901349. C.Gao and T.Tao are deeply grateful to professor L.Zhang for very valuable suggestion.

\end{document}